\pdfoutput=1
\RequirePackage{ifpdf}
\ifpdf 
\documentclass[pdftex]{sigma}
\else
\documentclass{sigma}
\fi

\usepackage{enumerate}
\usepackage{bm}
\usepackage{mathrsfs}
\usepackage[normalem]{ulem}
\usepackage{tikz}
\usetikzlibrary{arrows,matrix}

\tikzset{cross/.style={cross out, draw=black, minimum size=2*(#1-\pgflinewidth), inner sep=0pt, outer sep=0pt}, cross/.default={1pt}}

\numberwithin{equation}{section}
\theoremstyle{plain}
\newtheorem{thm}{Theorem}[section]
\newtheorem{lem}[thm]{Lemma}
\newtheorem{prop}[thm]{Proposition}
\newtheorem{cor}[thm]{Corollary}
\theoremstyle{definition}
\newtheorem{defn}[thm]{Definition}
\newtheorem{rem}[thm]{Remark}

\newcommand{\g}{{\mathfrak g}}
\newcommand{\n}{{\mathfrak n}}
\newcommand{\h}{{\mathfrak h}}
\newcommand{\half}{\frac{1}{2}}

\newcommand{\on}{.}
\newcommand{\FL}{F\!L}

\newcommand{\lf}{(}
\newcommand{\rf}{)}
\renewcommand{\t}{t}
\newcommand{\wc}{t}
\newcommand{\bi}{{\bar\imath}}

\begin{document}

\allowdisplaybreaks

\newcommand{\arXivNumber}{1505.07582}

\renewcommand{\PaperNumber}{091}

\FirstPageHeading

\ShortArticleName{Populations of Solutions to Cyclotomic Bethe Equations}

\ArticleName{Populations of Solutions\\ to Cyclotomic Bethe Equations}

\Author{Alexander {VARCHENKO}~$^\dag$ and Charles~A.S.~{YOUNG}~$^\ddag$}

\AuthorNameForHeading{A.~Varchenko and C.A.S.~Young}

\Address{$^\dag$~Department of Mathematics, University of North Carolina at Chapel Hill,\\
\hphantom{$^\dag$}~Chapel Hill, NC 27599-3250, USA}
\EmailD{\href{mailto:anv@email.unc.edu}{anv@email.unc.edu}}

\Address{$^\ddag$~School of Physics, Astronomy and Mathematics, University of Hertfordshire,\\
\hphantom{$^\ddag$}~College Lane, Hatfield AL10 9AB, UK}
\EmailD{\href{mailto:charlesyoung@cantab.net}{charlesyoung@cantab.net}}

\ArticleDates{Received June 17, 2014, in f\/inal form November 05, 2015; Published online November 14, 2015}

\Abstract{We study solutions of the Bethe Ansatz equations for the cyclotomic Gaudin model of~[Vicedo B., Young C.A.S., arXiv:1409.6937].
We give two interpretations of such solutions: as critical points of a cyclotomic master function, and as critical points with cyclotomic symmetry of a certain ``extended'' master function.
In f\/inite types, this yields a~correspondence between the Bethe eigenvectors and eigenvalues of the cyclotomic Gaudin model and those of an ``extended'' non-cyclotomic Gaudin model.
We proceed to def\/ine \emph{populations} of solutions to the cyclotomic Bethe equations, in the sense of~[Mukhin E., Varchenko A., \textit{Commun.\ Contemp.\ Math.} \textbf{6} (2004), 111--163, math.QA/0209017], for diagram automorphisms of Kac--Moody Lie algebras.
In the case of type A with the diagram automorphism, we associate to each population a~vector space of quasi-polynomials with specif\/ied ramif\/ication conditions. This vector space is equipped with a ${\mathbb Z}_2$-gradation and a~non-degenerate bilinear form which is (skew-)symmetric on the even (resp.~odd) graded subspace. We show that the population of cyclotomic critical points is isomorphic to the variety of isotropic full f\/lags in this space.}

\Classification{82B23; 32S22; 17B81; 81R12}

\section{Introduction}
Let $\g$ be a complex Kac--Moody Lie algebra and $\sigma\colon \g\to\g$ an automorphism of order $M\in {\mathbb Z}_{\geq 1}$. Let $\omega\in {\mathbb C}^\times$ be a primitive $M$th root of unity. We may choose a Cartan subalgebra $\h\subset \g$ such that $\sigma(\h)=\h$. We have the canonical pairing $\langle\cdot,\cdot\rangle\colon  \h^*\otimes \h\to {\mathbb C}$, and the simple roots $\alpha_i\in \h^*$ and coroots $\alpha^\vee_i\in \h$, where $i$ runs over the set $I$ of nodes of the Dynkin diagram.

Consider the following system of equations in $m\in {\mathbb Z}_{\geq 0}$ variables $\bm t = (t_1,\dots, t_m)\in {\mathbb C}^m$ and labels $\bm{\mathsf c} = ({\mathsf c}(1),\dots,{\mathsf c}(m))\in I^m$:
\begin{gather}
0= \sum_{k=0}^{M-1} \sum_{i=1}^N
\frac{\big\langle\sigma^k\Lambda_i, \alpha^\vee_{{\mathsf c}(j)}\big\rangle}{\wc_j-\omega^kz_i}
-  \sum_{k=0}^{M-1} \sum_{\substack{i=1\\i\neq j}}^m
\frac{\big\langle \sigma^k\alpha_{{\mathsf c}(i)},\alpha^\vee_{{\mathsf c}(j)}\big\rangle}{\wc_j-\omega^r\wc_i}
\nonumber\\
\hphantom{0=}{}+
\frac{1}{\wc_j}
\left(\big\langle  \Lambda_0,\alpha^\vee_{{\mathsf c}(j)} \big\rangle -\sum_{k=1}^{M-1} \frac{\big\langle \sigma^k\alpha_{{\mathsf c}(j)},\alpha_{{\mathsf c}(j)}^\vee\big\rangle}{1-\omega^k }  \right), \qquad j=1,\dots,m,
\label{cbeintro}
\end{gather}
where \looseness=1 $\Lambda_0,\Lambda_1,\dots,\Lambda_N\in \h^*$ are weights (with $\sigma \Lambda_0 = \Lambda_0$) and $z_1,\dots, z_N$ are non-zero points in the complex plane whose orbits, under the action of the cyclic group $\omega^{\mathbb Z}$, are pairwise disjoint.

When $\sigma =\operatorname{id}$, $\omega = 1$ and $\Lambda_0=0$, these equations reduce to the following well-known set of equations in mathematical physics:
\begin{gather}
0= \sum_{i=0}^{N}
\frac{\big\langle \Lambda_i,\alpha^\vee_{{\mathsf c}(j)}\big\rangle}{\t_j-z_i}
-  \sum_{\substack{i=1\\i\neq j}}^{m}
\frac{\big\langle \alpha_{{\mathsf c}(i)},\alpha^\vee_{{\mathsf c}(j)}\big\rangle}{\t_j-\t_i} , \qquad j=1,\dots,m.
\label{beintro}
\end{gather}
These are the equations for critical points of the \emph{master functions}~\cite{SV} which appear in the integral expressions for hypergeometric solutions to the Knizhnik--Zamolodchikov (KZ) equations. They are also the \emph{Bethe equations} of the quantum Gaudin model~\cite{BabujianFlume,FFR, RV}.

The equations \eqref{cbeintro} were introduced (for simple~$\g$) in the study of cyclotomic generalizations of the Gaudin model~\cite{VY1,VY2}~-- see also \cite{CrampeYoung,Skrypnyk1,Skrypnyk2}~-- as we recall in Section~\ref{gsec} below. Let us call them the \emph{cyclotomic Bethe equations}. (Cyclotomic generalizations of the KZ equations were studied in \cite{Brochier, Enriquez}, and appear in, in particular, the representation theory of cyclotomic Hecke algebras~\cite{VV}.)

It is natural to ask whether the cyclotomic Bethe equations~\eqref{cbeintro} can be interpreted as the equations for critical points of some master function. In the present paper we begin by giving
two dif\/ferent such interpretations.
First, they are indeed the critical point equations for a~\emph{cyclotomic master function}, which we write down in~\eqref{cmf}. But they are also the equations for critical points with cyclotomic~-- more precisely $S_m\ltimes ({\mathbb Z}/M{\mathbb Z})^m$~-- symmetry of what we call an \emph{extended master function},~\eqref{emf}.

Recall that a master function is specif\/ied by a \emph{weighted arrangement of hyperplanes}: that is, by a f\/inite collection ${\mathcal C}$ of af\/f\/ine hyperplanes in a complex af\/f\/ine space of f\/inite dimension, together with an assignment of a number $a(H)\in {\mathbb C}$ to each hyperplane $H\in {\mathcal C}$. Indeed, for each $H\in {\mathcal C}$, let $\ell_H=0$ be an af\/f\/ine equation for~$H$; then the master function is $\Phi =  \sum\limits_{H\in {\mathcal C}} a(H) \log \ell_H$.

The cyclotomic master function corresponds to a hyperplane arrangement in ${\mathbb C}^m$ whose hyperplanes include $t_i=\omega^k t_j$, $1\leq i<j\leq m$, for each $k\in {\mathbb Z}/M{\mathbb Z}$. By contrast, the extended master function corresponds to a hyperplane arrangement in ${\mathbb C}^{mM}$, but has only those hyperplanes corresponding to the type~$A$ root system, i.e., $t_i=t_j$, $1\leq i<j\leq mM$, etc.
Because the extended master function is a master function of this standard form, its critical point equations are the Bethe equations for a certain standard (i.e., non-cyclotomic) Gaudin model, which we call the \emph{extended Gaudin model}. This observation leads to our f\/irst result: a correspondence between the spectrum of the cyclotomic Gaudin model and a ``cyclotomic'' part of the spectrum of the extended Gaudin model. See Theorem~\ref{matchthm}.

Solutions to the Bethe equations~\eqref{beintro} form families called \emph{populations}. Populations were f\/irst introduced in~\cite{MV1, ScV}, where a~\emph{generation procedure} was given which produces families of new solutions to the Bethe equations starting from a given solution. A population is then def\/ined to be the Zariski closure of the set of all solutions to the Bethe equations obtained by repeated application of this generation procedure, starting from a given solution.
It is known that if~$\g$ is simple then every population is isomorphic to the f\/lag variety of the Langlands dual Lie algebra~${}^L\g$. This was shown in~\cite{MV1} for types~$A$,~$B$,~$C$ and in all f\/inite types in \cite{Freview, MVopers}. (A~population can also be understood as the variety of Miura opers with a given underlying oper; see \cite{Freview, MVopers}.)

In the present work our main goal is to initiate the study of \emph{cyclotomic populations}: populations of solutions to the equations~\eqref{cbeintro}.

We formulate in Section~\ref{cycgensec} a def\/inition of cyclotomic populations for $\g$ a general Kac--Moody Lie algebra and $\sigma$ any diagram automorphism of~$\g$ satisfying the \emph{linking condition}. (We also place certain restrictions on the weight $\Lambda_0$; see Section~\ref{l0sec}.) The linking condition~\cite{FSS} states that, for every node $i\in I$, the restriction of the Dynkin diagram to the orbit $\sigma^{\mathbb Z}(i)$ consists either of disconnected nodes (in which case~$i$ has \emph{linking number $L_i =1$}), or of a number of disconnected copies of the~${\mathrm A}_2$ Dynkin diagram (in which case~$i$ has \emph{linking number $L_i=2$}). What the linking condition ensures is that it is possible to ``fold'' the Dynkin diagram by the automorphism~$\sigma$. See Section~\ref{foldsec} and~\cite{FSS}.

In Section~\ref{cycgensec} we def\/ine the cyclotomic population to be the Zariski closure of the set of all cyclotomic critical points obtained by repeated application of a certain ``cyclotomic generation procedure'', starting from a given cyclotomic critical point. So the key ingredient is this ge\-ne\-ration procedure. Let us describe it, in outline. There is an ``elementary cyclotomic generation'' step associated to each orbit~$\sigma^{\mathbb Z}(i)$. There are two cases: $L_i=1$ and $L_i=2$.

First, suppose $i\in I$ is a node with linking number $L_i=1$.
A critical point $(\bm t,\bm{\mathsf c})$ is represented by a tuple of polynomials, $\bm y = (y_i(x))_{i\in I}$, where the roots of the polynomial~$y_i(x)$, $i\in I$, are the Bethe variables $t_s$ of ``colour''~$i$, i.e., those such that ${\mathsf c}(s)=i$. Following~\cite{MV1}, one def\/ines a~func\-tion of~$x$,
\begin{gather}
 y_i^{(i)}(x;c) := y_i(x) \int^x\xi^{\langle \Lambda_0,\alpha^\vee_i\rangle} T_i(\xi) \prod_{j\in I} y_j(\xi)^{-\langle\alpha_j,\alpha^\vee_i\rangle}d\xi + c y_i(x),\label{sg}
\end{gather}
depending on a parameter $c\in {\mathbb C}$ .
Here $T_i(x)$, $i\in I$, are certain functions encoding the ``frame'' data, i.e., the points $z_1,\dots, z_N$ and the weights $\Lambda_1,\dots,\Lambda_N$; see~\eqref{Tdef}.
The Bethe equations ensure that $y_i^{(i)}(x;c)$ is in fact a polynomial, and moreover that if we consider the new tuple~$\bm y^{(i)}(c)$ in which~$y_i(x)$ is replaced by~$y_i^{(i)}(x;c)$, then for almost all values of~$c$ this new tuple again represents a solution to the Bethe equations. Call the replacement $\bm y \mapsto \bm y^{(i)}(c)$ \emph{elementary generation in direction~$i$}.
Now suppose the initial tuple~$\bm y$ represents a cyclotomic point. That means
\begin{gather*}  y_{\sigma j}( \omega x) \simeq y_j(x),\qquad j\in I;
\end{gather*}
see Lemma \ref{cycptlem}. Since the orbit $\sigma^{\mathbb Z}(i)$ consists of disconnected nodes of the Dynkin diagram, the operations of elementary generation in the directions~$\sigma^{\mathbb Z}(i)$ commute. By performing each of them once, in any order, we can arrange to arrive at a new cyclotomic point. See Theorem \ref{s1thm}.

Next, suppose $i\in I$ is a node with linking number $L_i=2$. Then for every copy of the~${\mathrm A}_2$ diagram, with nodes say~$j$ and~${\bar\jmath}$, one must perform the sequence of generation steps~$j$,~${\bar\jmath}$,~$j$. Doing this for each copy of~${\mathrm A}_2$ in turn, in any order, we can arrange to arrive at a new cyclotomic point. See Theorem~\ref{s2thm}.

When $L_i=2$ there is a subtlety coming from our assumptions about the weight at the origin,~$\Lambda_0$. Throughout Section~\ref{cycgensec}, motivated by~\cite{VY1}, we assume that $\langle\Lambda_0,\alpha^\vee_i\rangle$ is non-integral when $L_i=2$. That means that the expression~\eqref{sg} develops a branch point at the origin. The upshot is that at certain intermediate steps, the weight at the origin is shifted to ${\mathsf s}_i \cdot \Lambda_0$, before eventually being shifted back to $\Lambda_0$. See Proposition~\ref{nb} and compare~\cite{MV2}.

In either case, $L_i=1$ or $L_i=2$, we write $\bm y^{(i,\sigma)}(c)$ for the tuple of polynomials representing the new cyclotomic critical point. It depends on a single parameter~$c$. The replacement $\bm y \mapsto \bm y^{(i,\sigma)}(c)$ is the elementary cyclotomic generation, in the direction of the orbit~$\sigma^{\mathbb Z}(i)$.

To a critical point $(\bm t,\bm{\mathsf c})$ represented by a tuple of polynomials $\bm y$ one can associate a~weight~$\Lambda_\infty$. See~\eqref{l8def} and~\eqref{l8y}. For f\/ixed $\Lambda_0,\Lambda_1,\dots,\Lambda_N$, we may regard $\Lambda_\infty$ as encoding the number of roots $t_s$ of each ``colour'' $i\in I$, i.e., the degrees of the polynomials $y_i(x)$. It is known that $\Lambda_\infty(\bm y^{(i)}(c))$ is equal either to $\Lambda_\infty(\bm y)$ or to ${\mathsf s}_i \cdot \Lambda_\infty(\bm y)$, where ${\mathsf s}_i\cdot{}$ denotes the shifted action of the Weyl ref\/lection in root~$\alpha_i$. See~\cite{MV1}.
We have an analogous statement in the cyclotomic case. Namely, there is a ``folded'' Weyl group $W^\sigma$ with generators~${\mathsf s}_i^\sigma$. See Section~\ref{foldsec}. And we show that $\Lambda_\infty(\bm y^{(i,\sigma)}(c))$ is equal either to $\Lambda_\infty(\bm y)$ or to ${\mathsf s}_i^\sigma\cdot\Lambda_\infty(\bm y)$. For the precise statement see Theorems~\ref{s1thm} and~\ref{s2thm}.

We proceed in Section~\ref{ARsec} to treat in detail the case of type A with the diagram automorphism.

Recall f\/irst from \cite{MV1} the structure of populations in type $A_R$, $R\in {\mathbb Z}_{\geq 1}$, for the master functions associated to marked points $z_1,\dots,z_N$ and integral dominant weights $\Lambda_1,\dots,\Lambda_N$. In that setting, every population of critical points is isomorphic to a variety of full f\/lags in a certain $(R+1)$-dimensional vector space ${\mathcal K}$ of polynomials. The ramif\/ication points of ${\mathcal K}$ are $z_1,\dots,z_N$ and $\infty$, and the ramif\/ication data at these points are specif\/ied by the weights $\Lambda_1,\dots,\Lambda_N$ and an integral dominant weight $\tilde\Lambda_\infty$.
Given a full f\/lag ${\mathcal F} = \{0 =F_0 \subset F_1 \subset F_2 \subset \dots \subset F_{R+1}= {\mathcal K} \}$ in ${\mathcal K}$, pick any basis $(u_i(x))_{i=1}^{R+1}$ of polynomials adjusted to this f\/lag, i.e., such that $F_k = \operatorname{span}_{\mathbb C}(u_1(x),\dots, u_k(x))$. Then def\/ine a tuple of functions $\bm y^{\mathcal F} = (y^{\mathcal F}_k(x))_{k=1}^R$ by
\begin{gather*} y^{\mathcal F}_k(x) = \operatorname{Wr}(u_1(x),\dots,u_k(x)) / \big(T_1^{k-1}(x) T_2^{k-2}(x) \dots T_{k-1}(x)\big),
\end{gather*}
where~-- as in \eqref{sg} above~-- the $(T_i(x))_{i=1}^R$ are functions encoding the ``frame'' data $z_1,\dots,z_N$ and $\Lambda_1,\dots,\Lambda_N$, and where $\operatorname{Wr}(u_1(x),\dots,u_k(x))$ denotes the Wronskian determinant. The ramif\/ication properties of ${\mathcal K}$ ensure that the $y^{\mathcal F}_k(x)$ are in fact polynomials. Moreover the map ${\mathcal F} \mapsto \bm y^{\mathcal F}$ is an isomorphism of varieties from the variety of full f\/lags in~${\mathcal K}$ to the population associated with~${\mathcal K}$.
The space~${\mathcal K}$ is the kernel of a certain linear dif\/ferential operator~${\mathcal D}$ of order $R+1$ (essentially a~type~$A$ oper). This operator~${\mathcal D}$ can be def\/ined in terms of the $(T_i(x))_{i=1}^R$ together with the polynomials $(y_i(x))_{i=1}^R$ of (any) point in the population. (See Section~\ref{Dsec}.)

Now let us discuss how the picture changes in our present setting.
For us, the weight at the origin~$\Lambda_0$ need not be integer dominant. We assume it satisf\/ies weaker assumptions given in~\eqref{l0a}. These assumptions mean that we are led to consider vector spaces~${\mathcal K}$ of \emph{quasi-}polynomials: that is, polynomials in~$x^\half$. The local behaviour of these quasi-polynomials near the origin is encoded in $\Lambda_0$. The remaining ramif\/ication points are $z_1,\dots, z_N$, $-z_1,\dots,-z_N$, and $\infty$. See Def\/inition~\ref{Kfdef}.

The space of quasi-polynomials~${\mathcal K}$ admits a natural ${\mathbb Z}_2$ gradation ${\mathcal K} = {\mathcal K}_{\textup{O}} \oplus {\mathcal K}_{\textup{Sp}}$. We call f\/lags which respect this gradation \emph{decomposable}. Decomposable full f\/lags are classif\/ied by their \emph{type}; see Section~\ref{flagsec}. In particular the f\/lags ${\mathcal F}\in \FL_S({\mathcal K})$ of a certain preferred type~$S$,~\eqref{sh}, are sent to polynomials under the map ${\mathcal F}\mapsto \bm y^{\mathcal F}$. This map of varieties $\FL_S({\mathcal K})\to {\mathbb P}({\mathbb C}[x])^R$ is an isomorphism onto its image. The \emph{cyclotomic population} is then the set of cyclotomic tuples in this image, i.e., the set of tuples $\bm y^{\mathcal F}$, ${\mathcal F}\in \FL_S({\mathcal K})$, such that $y_i(x) \simeq y_{R+1-i}(-x)$, $i=1,\dots,R$. The question is: which f\/lags in~$\FL_S({\mathcal K})$ map to cyclotomic tuples?

To answer this question we introduce the notion of a \emph{cyclotomically self-dual space of quasi-polynomials}. The space ${\mathcal K}$ has a natural dual space ${\mathcal K}^\dag$ of quasi-polynomials~-- see Section~\ref{Kdagsec}~-- and we say ${\mathcal K}$ is cyclotomically self-dual if for all $v(x)\in {\mathcal K}$, $v(-x)\in {\mathcal K}^\dag$. (Compare the very similar notion of a self-dual space of polynomials in~\cite{MV1}.)
We show that a suf\/f\/icient condition for ${\mathcal K}$ to be cyclotomically self-dual is that there exists at least one full f\/lag ${\mathcal F}$ in ${\mathcal K}$ such that $\bm y^{\mathcal F}$ is cyclotomic (Theorem~\ref{cptthm}).
If~${\mathcal K}$ is cyclotomically self-dual then it admits a canonical non-degenerate bilinear form $B$. We show that, for all full f\/lags ${\mathcal F}$ in ${\mathcal K}$, the tuple $\bm y^{\mathcal F}$ is cyclotomic if and only if ${\mathcal F}$ is isotropic with respect to~$B$ (Theorem~\ref{cycisothm}).

Therefore the cyclotomic population is isomorphic to the variety $\FL_S^\perp({\mathcal K})$ of isotropic f\/lags of type~$S$ in~${\mathcal K}$. The bilinear form~$B$ is symmetric on~${\mathcal K}_{\textup{O}}$ and skew-symmetric on~${\mathcal K}_{\textup{Sp}}$, and these subspaces are mutually orthogonal with respect to~$B$ (Theorem~\ref{Bthm}). Hence this variety $\FL_S^\perp({\mathcal K})$ is isomorphic to the direct product of spaces of isotropic f\/lags $\FL^\perp({\mathcal K}_{\textup{Sp}}) \times \FL^\perp({\mathcal K}_{\textup{O}})$.

\section{Master functions and cyclotomic symmetry}\label{mfsec}

\subsection{Kac--Moody algebras}\label{KMa}

 Let $I$ be a f\/inite set of indices and $A = (a_{i,j})_{i,j\in I}$ a generalized Cartan matrix, i.e.,  $a_{i,i}=2$ and $a_{i,j}\in {\mathbb Z}_{\leq 0}$ whenever $i\neq j$, with $a_{i,j}=0$ if and only if $a_{j,i}=0$.
Let $\g:=\g(A)$ be the corresponding complex Kac--Moody Lie algebra \cite[Section~1]{KacBook}, $\h\subset\g$ a Cartan subalgebra, and
\begin{gather*}
 \g= \n_- \oplus \h \oplus \n_+
 \end{gather*}
a triangular decomposition.
Let $\alpha_i\in \h^*$, $\alpha^\vee_i\in \h$, $i\in I$ be collections of simple roots and coroots respectively. We have $\dim \h=|I|+\dim\ker A = 2|I|-\operatorname{rank} A$. By def\/inition,
\begin{gather*}
 \langle \alpha_i,\alpha^\vee_j\rangle = a_{j,i},
 \end{gather*}
where $\langle\cdot,\cdot\rangle\colon \h^*\otimes \h\to {\mathbb C}$ is the canonical pairing.

We assume that $A$ is symmetrizable, i.e., there exists a diagonal matrix $D= \operatorname{diag}(d_i)_{i\in I}$, whose entries are coprime positive integers, such that the matrix $B=DA$ is symmetric. Let $\lf\cdot,\cdot\rf$ be the associated symmetric bilinear form on $\h^*$. We have $\lf\alpha_i,\alpha_j\rf = d_i a_{i,j}$ and
\begin{gather*}
 \langle\lambda,\alpha^\vee_i\rangle  = 2\lf\lambda,\alpha_i\rf/\lf\alpha_i,\alpha_i\rf\qquad\text{for all $\lambda\in \h^*$}.
 \end{gather*}
The form $\lf\cdot,\cdot\rf$ is non-degenerate. Therefore it gives an identif\/ication $\h\cong_{\mathbb C} \h^*$ and hence a non-degenerate symmetric bilinear  form on $\h$ which we also write as $(\cdot,\cdot)$.

Let ${\mathcal P} := \{\lambda\in \h^*\colon \langle \lambda , \alpha^\vee_i\rangle \in {\mathbb Z}\}$ be the integral weight lattice and ${\mathcal P}_+ := \{\lambda\in \h^*\colon \langle \lambda , \alpha^\vee_i\rangle \in {\mathbb Z}_{\geq 0}\}$ the set of dominant integral weights.

Let $W \subset \operatorname{End}(\h^*)$ be the Weyl group. It is generated by the ref\/lections ${\mathsf s}_i$, $i\in I$, given by
\begin{gather*}
 {\mathsf s}_i(\lambda) := \lambda - \langle \lambda,\alpha^\vee_i\rangle \alpha_i, \qquad \lambda\in \h^* .
\end{gather*}

Let $\rho\in \h^*$ be a vector such that $\langle\rho,\alpha^\vee_i\rangle = 1$ for $i\in I$.
We use $\cdot$ to denote the \emph{shifted action} of the Weyl group, i.e.,
\begin{gather*}
 {\mathsf s}\cdot \lambda := w(\lambda+\rho) - \rho, \qquad {\mathsf s}\in W,\quad \lambda\in \h^*.
 \end{gather*}

\subsection{Diagram automorphism}
Suppose $\sigma$ is an automorphism of the Dynkin diagram \cite[Section~4.7]{KacBook} of~$A$. That is, $\sigma$ is a~permutation of the index set~$I$ such that
\begin{gather*} a_{\sigma i,\sigma j} = a_{i,j}.
\end{gather*}

Let $M$ be the order of $\sigma$ and let $\omega\in {\mathbb C}^\times$ be a primitive $M$th root of unity.

To such a permutation is associated a \emph{diagram automorphism} $\g\to \g$ of the Kac--Moody Lie algebra \cite{FSS}, which we shall also write as $\sigma$. We have
\begin{gather*}
 \sigma E_i = E_{\sigma i},\qquad \sigma F_i = F_{\sigma i}, \qquad \sigma \alpha^\vee_i = \alpha^\vee_{\sigma i} , \qquad i\in I,
\end{gather*}
where $E_i\in \n$, $F_i\in\n^-$, $i\in I$, are a set of Chevalley generators of $[\g,\g]$. This def\/ines $\sigma$ on the derived subalgebra $[\g,\g]$ of $\g$. For the action of $\sigma$ on the derivations, i.e., on a complement of~$[\g,\g]$ in~$\g$, see \cite[Section~3.2]{FSS}. This action may be chosen to ensure that $\sigma\colon \g\to\g$ has order~$M$ and respects the bilinear form $\lf\cdot,\cdot\rf$ on~$\h$:
\begin{gather*}
 \lf\sigma X, \sigma Y\rf = \lf X,Y\rf \qquad\text{for all}\quad  X,Y\in \h.
 \end{gather*}
The action of $\sigma$ on $\h^*$ is def\/ined by $\sigma \lambda := \lambda\circ \sigma^{-1}$ so that $\langle\sigma\lambda,\sigma X\rangle = \langle\lambda,X\rangle$ for all $\lambda\in \h^*$, $X\in \h$. Note that then $\sigma \alpha_i = \alpha_{\sigma i}$ for all $i\in I$.

Let $\g^\sigma
\subset\g$
be the Lie subalgebra of elements invariant under $\sigma$. We have
\begin{gather*}
 \g^\sigma = \n_-^\sigma \oplus \h^\sigma \oplus \n_+^\sigma
 \end{gather*}
with $\n_\pm^\sigma = \g^\sigma \cap \n_\pm$ and $\h^\sigma = \g^\sigma \cap \h$.

\subsection{The linking condition and the folded diagram}\label{foldsec}

For any $i\in I$ let
\begin{gather*}
 M_i:= \big|\big\{i,\sigma i,\sigma^2 i, \dots, \sigma^{M-1} i\big\}\big|
 \end{gather*}
be the length of the orbit of the node $i$ under the automorphism $\sigma$ of the Dynkin diagram~$A$. Def\/ine
\begin{gather*}
L_i := 1- \sum_{k=1}^{M_i-1} a_{\sigma^k i, i} .
\end{gather*}
Note that $L_i\geq 1$. Following~\cite{FSS}, we say that  $\sigma$ obeys the \emph{linking condition} if and only if
\begin{gather}
 L_i \leq 2\qquad\text{for all}\quad i\in I. \label{lc}
\end{gather}

To understand the meaning of this condition, consider the restriction of the Dynkin diagram to the orbit
of the node~$i$.
If $L_i=1$ then this induced subgraph has no edges at all. If $L_i=2$ then it consists of $M_i/2$ disconnected copies of the type~${\mathrm A}_2$ Dynkin diagram.

\begin{rem}\label{frem}
If $A$ is of f\/inite type, then all diagram automorphisms obey the linking condition. Moreover, in all f\/inite types except ${\mathrm A}_{2n}$, $n\in {\mathbb Z}_{\geq 1}$, we in fact have $L_i=1$ for every node $i$: that is, no two distinct nodes in the same $\sigma$-orbit are ever linked by an edge of the Dynkin diagram. In type ${\mathrm A}_{2n}$ the non-trivial diagram automorphism gives $L_i=2$ for $i\in \{n,n+1\}$ and $L_i=1$ otherwise:
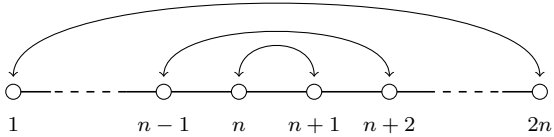
\begin{gather*}
\begin{tikzpicture}[baseline =-5,scale=1,font=\scriptsize]
\draw[semithick] (-.5,0) -- (3.5,0);
\draw[semithick,dashed] (-.5,0) -- (-1.5,0);
\draw[semithick,dashed] (3.5,0) -- (4.5,0);
\draw[semithick] (-2,0) -- (-1.5,0);
\draw[semithick] (4.5,0) -- (5,0);
\filldraw[fill=white] (0,0) circle (1mm) node [below=2mm] {$n-1$};
\filldraw[fill=white] (1,0) circle (1mm) node [below=2mm] {$\phantom1n\phantom1$};
\filldraw[fill=white] (2,0) circle (1mm)  node [below=2mm] {$n+1$};
\filldraw[fill=white] (3,0) circle (1mm)  node [below=2mm] {$n+2$};
\filldraw[fill=white] (5,0) circle (1mm)  node [below=2mm] {$2n$};
\filldraw[fill=white] (-2,0) circle (1mm)  node [below=2mm] {$1$};
\draw[<->,shorten >=2mm,shorten <=2mm] (1,0) .. controls (1,.75) and (2,.75) .. (2,0);
\draw[<->,shorten >=2mm,shorten <=2mm] (0,0) .. controls (0,1) and (3,1) .. (3,0);
\draw[<->,shorten >=2mm,shorten <=2mm] (-2,0) .. controls (-2,1.5) and (5,1.5) .. (5,0);
\end{tikzpicture}
\end{gather*}
\end{rem}

\begin{rem}\label{arem}
If $A$ is of af\/f\/ine type then all diagram automorphisms obey the linking condition with the following exception. In type $\mathrm A_{n}^{(1)}$, $n\in {\mathbb Z}_{\geq 2}$, let $R$ be a generator of the cyclic sub\-group~$C_{n+1}$ of the full automorphism group of the Dynkin diagram (which is the dihedral group~$D_{n+1}$). Then $R$ does not obey the linking condition. Indeed, the $R$-orbit of any node~$i$ is the whole diagram, and $L_i = 1 + n$.
\end{rem}

Given any diagram automorphism satisfying the linking condition it is possible to def\/ine a~\emph{folded} Dynkin diagram. Let us make a choice of subset
\begin{gather*}
 I_\sigma\subseteq I
\end{gather*}
consisting of exactly one representative of each $\sigma$-orbit. Then the Cartan matrix $A^{\sigma} = (a^\sigma_{i,j})_{i,j\in I_\sigma}$ of the folded diagram is given by
\begin{gather*}
 a^\sigma_{i,j} = L_i \sum_{k=0}^{M_i-1} a_{\sigma^k i,j}.
\end{gather*}

\begin{rem} Compare Section~3.3 of \cite{FSS}, noting that our convention $a_{j,i} = \langle \alpha_i,\alpha^\vee_j\rangle$ dif\/fers from that of~\cite{FSS}.
\end{rem}

\begin{lem}[\cite{FSS}]
If $\sigma$ obeys the linking condition then $A^\sigma$ $($and its transpose$)$ is a symmetrizable Cartan matrix whose type $($finite, affine, or indefinite$)$ is the same as that of $A$.
\end{lem}

For each $i\in I_\sigma$ let us def\/ine also
\begin{gather*} \alpha^{\vee,\sigma}_i :=  L_i \sum_{k=0}^{M_i-1} \alpha^\vee_{\sigma^k i}\qquad
\text{and}\qquad
 E^\sigma_i :=        \sum_{k=0}^{M_i-1} E_{\sigma i}, \qquad
    F^\sigma_i := L_i\sum_{k=0}^{M_i-1} F_{\sigma i}.
\end{gather*}

Then we have
\begin{gather*}
    \left[ E^\sigma_i,F^\sigma_j\right ] = \delta_{i,j} \alpha^{\vee,\sigma}_i, \qquad
\left[ \alpha^{\vee,\sigma}_i, E^\sigma_j\right] =  E^\sigma_j a^\sigma_{j,i}, \qquad
\left[ \alpha^{\vee,\sigma}_i, F^\sigma_j\right] = -F^\sigma_j a^\sigma_{j,i} \qquad i,j\in I_\sigma.
\end{gather*}
Thus $\alpha^{\vee,\sigma}_i$, $E^\sigma_i$, $F^\sigma_i$, $i\in I_\sigma$ generate a copy of (the derived subalgebra of) the Kac--Moody Lie algebra $\g({A^\sigma})$ inside $\g^\sigma := \{ X\in \g \colon \sigma X = X\}$.
Next, for all $i\in I_\sigma$, if we let
\begin{gather*}
 \alpha^\sigma_i := \frac{L_i}{M_i}\sum_{k=0}^{M_i-1} \alpha_{\sigma^k i} \in \h^*
\end{gather*}
then $\langle \alpha^\sigma_i ,\alpha^{\vee,\sigma}_j \rangle = a^\sigma_{j,i}$.
Def\/ine $W^\sigma$ to be the group generated by the elements ${\mathsf s}^\sigma_i\in \operatorname{End}(\h^*)$ given by
\begin{gather*}
 {\mathsf s}^\sigma_i(\lambda) := \lambda- \langle \lambda,\alpha^{\vee,\sigma}_i\rangle \alpha^\sigma_i,\qquad i\in I_\sigma.
 \end{gather*}

\begin{lem}\label{Wsiglem} $W^\sigma$ is a subgroup of $W$. Indeed, we have
\begin{gather*}  {\mathsf s}^\sigma_i =
\begin{cases} \displaystyle \prod_{k=0}^{M_i-1} {\mathsf s}_{\sigma^k i}, & L_i = 1, \vspace{1mm}\\
\displaystyle\left(\prod_{k=0}^{M_i/2-1} {\mathsf s}_{\sigma^k i}\right)
\left(\prod_{k=0}^{M_i/2-1} {\mathsf s}_{\sigma^{k+M_i/2} i}\right)
\left(\prod_{k=0}^{M_i/2-1} {\mathsf s}_{\sigma^k i}\right), & L_i = 2.
\end{cases}
\end{gather*}
\end{lem}

\subsection{The cyclotomic master function}\label{cycmfsec}
Let $\bm\Lambda = (\Lambda_i)_{i=1}^N$ be a collection of $N\in {\mathbb Z}_{\geq 0}$ integral dominant weights $\Lambda_i\in {\mathcal P}_+$. Let $\bm z = (z_i)_{i=1}^N$ be a collection of nonzero points $z_i\in {\mathbb C}^\times$ such that $\omega^{\mathbb Z} z_i \cap \omega^{\mathbb Z} z_j = \varnothing$ whenever $i\neq j$. We shall call $\Lambda_i$ \emph{the weight at $z_i$}.

In addition, we pick a weight $\Lambda_0\in \h^{\sigma,*}$. We call $\Lambda_0$ \emph{the weight at the origin}.

Let $\bm{\mathsf c}= ({\mathsf c}(j))_{j=1}^m$ be an $m$-tuple of elements of $I$, and introduce variables $\bm \wc = (\wc_j)_{j=1}^m$. We shall say that $\wc_j$ is a \emph{variable of colour} ${\mathsf c}(j)$.

We def\/ine the \emph{cyclotomic master function} $\Phi = \Phi_{\g,\sigma}(\bm \wc;\bm{\mathsf c};\bm z;\bm \Lambda,\Lambda_0)$ associated to these data to be
\begin{gather}
\Phi :=
\sum_{i=1}^N
\left(\half \sum_{k=1}^{M-1} \big( \Lambda_i,\sigma^k\Lambda_i\big) + \lf \Lambda_i, \Lambda_0 \rf\right) \log z_i
+ \sum_{k=0}^{M-1} \sum_{1\leq i<j\leq n}  \big(\Lambda_i,\sigma^k \Lambda_j \big)\log\big(z_i - \omega^k z_j\big) \nonumber\\
\hphantom{\Phi :=}{}
- \sum_{k=0}^{M-1} \sum_{i=1}^N\sum_{j=1}^m  \big( \alpha_{{\mathsf c}(j)},\sigma^k\Lambda_i\big) \log  \big(\wc_j-\omega^kz_i\big)\nonumber\\
\hphantom{\Phi :=}{}
+\sum_{k=0}^{M-1} \sum_{1\leq i<j \leq m} \big( \alpha_{{\mathsf c}(i)},\sigma^k\alpha_{\mathsf c}(j)\big)\log \big(\wc_i-\omega^k\wc_j\big)\nonumber\\
\hphantom{\Phi :=}{}
+\sum_{i=1}^m \left(
\half \sum_{k=1}^{M-1} \big( \alpha_{{\mathsf c}(i)},\sigma^k\alpha_{{\mathsf c}(i)}\big) - \lf \alpha_{{\mathsf c}(i)}, \Lambda_0 \rf \right)  \log \wc_i.\label{cmf}
\end{gather}

A point $\bm \wc$ with complex coordinates  is called a \emph{critical point} of the cyclotomic master function if
\begin{gather*}
 \frac{\partial \Phi}{\partial \wc_i} =0 ,\qquad i=1,\dots,m,
 \end{gather*}
or equivalently (in view of Lemma \ref{hl} below) if the following equations are satisf\/ied:
\begin{gather}
0= \sum_{k=0}^{M-1} \sum_{i=1}^N
\frac{\big( \alpha_{{\mathsf c}(j)},\sigma^k\Lambda_i\big)}{\wc_j-\omega^kz_i}
-  \sum_{k=0}^{M-1} \sum_{\substack{i=1\\i\neq j}}^m
\frac{\big(\alpha_{{\mathsf c}(j)},\sigma^k\alpha_{{\mathsf c}(i)}\big)}{\wc_j-\omega^r\wc_i}\nonumber\\
\hphantom{0=}{}
+
\frac{1}{\wc_j}
\left(-\sum_{k=1}^{M-1} \frac{\big( \alpha_{{\mathsf c}(j)},\sigma^k\alpha_{{\mathsf c}(j)}\big)}{1-\omega^k } + \lf \alpha_{{\mathsf c}(j)}, \Lambda_0 \rf \right)
\label{cbe}
\end{gather}
for $j=1,\dots,m$. Call this system of equations the \emph{cyclotomic Bethe equations}.

\begin{lem} \label{hl}
For any $\lambda\in \h^*$,
\begin{gather*}
\sum_{k=1}^{M-1} \frac{\lf \lambda, \sigma^k\lambda\rf}{1-\omega^k}
=\half\sum_{k=1}^{M-1}\left( \frac{\lf \lambda, \sigma^k\lambda\rf}{1-\omega^k} +  \frac{\lf \lambda, \sigma^k\lambda\rf}{1-\omega^{-k}} \right)
= \half\sum_{k=1}^{M-1} \lf \lambda,\sigma^k\lambda\rf.
\end{gather*}
\end{lem}

Def\/ine  $\Lambda_\infty$, \emph{the weight at infinity}, to be
\begin{gather}
 \Lambda_\infty := \Lambda_0 + \sum_{k=0}^{M-1} \sum_{i=1}^N \Lambda_{\sigma^k(i)} - \sum_{k=0}^{M-1} \sum_{i=1}^m \alpha_{\sigma^k {\mathsf c}(i)}. \label{l8def}
\end{gather}

The group $S_{m}$ acts on pairs of $m$-tuples $(\bm \wc,\bm {\mathsf c})$ by permuting indices:
\begin{gather*}
\rho \on (\bm \wc,\bm {\mathsf c}) = \big( \big(\wc_{\rho^{-1}(1)},\dots,\wc_{\rho^{-1}(m)}\big), \big({\mathsf c}\big(\rho^{-1}(1)\big),\dots,{\mathsf c}\big(\rho^{-1}(m)\big)\big) \big) .
\end{gather*}
The group ${\mathbb Z}/M{\mathbb Z}$ acts on pairs $(\wc,{\mathsf c}) \in {\mathbb C} \times I$ by $k\on (\wc, {\mathsf c}) =  (\omega^k \wc_i, \sigma^k {\mathsf c})$.
This gives rise to an action of the wreath product
$S_m \wr ({\mathbb Z}/M{\mathbb Z}) := S_m \ltimes ({\mathbb Z}/M{\mathbb Z})^m$ on pairs of tuples  $(\bm \wc, \bm {\mathsf c})\in {\mathbb C}^m \times I^m$.

\begin{lem} Up to an additive constant, the cyclotomic master function $\Phi$ is invariant under the pull-back of this action of $S_m \wr ({\mathbb Z}/M{\mathbb Z})$.
In particular, if $\bm \wc$ is a critical point of $\Phi(\bm \wc;\bm{\mathsf c})$ then~$X\on \bm \wc$ is a critical point of $\Phi(X\on \bm \wc; X\on\bm{\mathsf c})$, for all $X\in S_m\wr({\mathbb Z}/M{\mathbb Z})$.
\end{lem}

\subsection{The extended master function}\label{extmfsec}

The equations \eqref{cbe} admit another, closely related, interpretation. Recall the def\/inition of the (usual) master function~\cite{SV}. Namely, let
 $\tilde{\bm\Lambda} = (\tilde\Lambda_i)_{i=0}^{\tilde N}$ be a collection of $\tilde N+1\in {\mathbb Z}_{\geq 0}$  weights $\tilde\Lambda_i\in \h^*$, and let $\tilde{\bm z} = (\tilde z_i)_{i=0}^{\tilde N}$ be a collection of nonzero points $\tilde z_i\in {\mathbb C}^\times$.
Pick $\tilde m\in {\mathbb Z}_{\geq 0}$, let $\bm {\mathsf c}=({\mathsf c}(j))_{j=1}^{\tilde m}$ be an $\tilde m$-tuple of elements of $I$ and introduce variables $\bm \t = (\t_j)_{j=1}^{\tilde m}$.
The \emph{master function} associated to these data is
\begin{gather}
{\widetilde \Phi} :=
\sum_{0\leq i<j\leq \tilde N} \big( \tilde \Lambda_i,\tilde \Lambda_j\big) \log (\tilde z_i- \tilde z_j)
- \sum_{i=0}^{\tilde N}\sum_{j=1}^{\tilde m}  \big(\alpha_{{\mathsf c}(j)},\tilde\Lambda_i\big) \log (\t_j-\tilde z_i)\nonumber\\
\hphantom{{\widetilde \Phi} :=}{}
+ \sum_{1\leq i< j \leq \tilde m} \big(\alpha_{{\mathsf c}(i)},\alpha_{{\mathsf c}(j)}\big) \log (\t_i-\t_j) .
 \label{mf}
\end{gather}
It is a function of the variables $\bm\t$, depending on the parameters $\bm{\mathsf c}$, $\tilde{\bm z}$  and $\tilde{\bm \Lambda}$. The critical points of the master function are those points  $\bm \t$ with complex coordinates such that $\partial {\widetilde \Phi}/\partial \t_j =0$ for $j=1,\dots, \tilde m$,
i.e., those points such that the following equations are satisf\/ied:
\begin{gather}
0= \sum_{i=0}^{\tilde N}
\frac{\lf \alpha_{{\mathsf c}(j)},\tilde \Lambda_i\rf}{\t_j-\tilde z_i}
-  \sum_{\substack{i=1\\i\neq j}}^{\tilde m}
\frac{\lf \alpha_{{\mathsf c}(j)},\alpha_{{\mathsf c}(i)}\rf}{\t_j-\t_i} , \qquad j=1,\dots,\tilde m.
\label{be}
\end{gather}
In this paper we are concerned with the following special case.
Let $\tilde N= N M$, choose $(\tilde z_i)_{i=0}^{N M}$ to be
\begin{subequations}\label{symz}
\begin{gather} \tilde z_0 = 0,\qquad \tilde z_{k+M i} = \omega^k z_i, \qquad
k = 0,1,\dots, M-1, \quad i=1,\dots, N,
\end{gather}
and choose the weights at these points to be
\begin{gather}
 \tilde \Lambda_0 = \Lambda_0, \qquad \tilde \Lambda_{k+M i} = \sigma^k \Lambda_i,
\end{gather}
\end{subequations}
where $z_i$, $\Lambda_i$, $i=1,\dots, N$, and $\Lambda_0$ are as in Section~\ref{cycmfsec}.
We call the master function in this case the \emph{extended master function}, $\widehat\Phi = \widehat\Phi_{\g,\sigma}(\bm t; \bm {\mathsf c}; \bm z; \bm \Lambda;\Lambda_0)$. It is given by
\begin{gather}
\widehat\Phi :=
\sum_{k=0}^{M-1} \sum_{i=1}^N \big( \Lambda_0,\sigma^k \Lambda_i\big)\log \big({-}\omega^kz_i\big)
+ \sum_{k,\,l=0}^{M-1}  \sum_{1\leq i<j\leq  n} \big(\sigma^k \Lambda_i,\sigma^l\Lambda_j\big) \log \big(\omega^{k}z_i- \omega^{l} z_j\big)\nonumber\\
\hphantom{\widehat\Phi :=}{}
 + \sum_{0\leq k<l\leq T-1} \sum_{i=1}^N \big(\sigma^k \Lambda_i,\sigma^l\Lambda_i\big) \log\big(\omega^k - \omega^l\big) z_i
- \sum_{j=1}^{\tilde m}  \lf \alpha_{{\mathsf c}(j)},\Lambda_0\rf \log (\t_j)\nonumber\\
\hphantom{\widehat\Phi :=}{}
- \sum_{k=0}^{M-1}\sum_{i=1}^N\sum_{j=1}^{\tilde m} \big( \alpha_{{\mathsf c}(j)},\omega^k\Lambda_i\big) \log \big(\t_j-\omega^kz_i\big)
+ \sum_{1\leq i< j \leq \tilde m}  ( \alpha_{{\mathsf c}(i)},\alpha_{{\mathsf c}(j)} ) \log (\t_i-\t_j) ,\label{emf}
\end{gather}
and the critical point equations \eqref{be} take the form
\begin{gather}
0= \sum_{k=0}^{M-1} \sum_{i=1}^N
\frac{\big(\alpha_{{\mathsf c}(j)},\sigma^k\Lambda_i\big)}{\t_j-\omega^k z_i}
+  \frac{\lf \alpha_{{\mathsf c}(j)},\Lambda_0\rf}{\t_j}
-  \sum_{\substack{i=1\\i\neq j}}^{\tilde m}
\frac{\lf \alpha_{{\mathsf c}(j)},\alpha_{{\mathsf c}(i)}\rf}{\t_j-\t_i} , \qquad j=1,\dots,\tilde m.
\label{ebe}
\end{gather}

The group $S_{\tilde m}$ acts on pairs of $\tilde m$-tuples $(\bm t,\bm {\mathsf c})$ by permuting indices:
\begin{gather}
\rho \on (\bm t,\bm {\mathsf c}) = \big( \big (t_{\rho^{-1}(1)},\dots,t_{\rho^{-1}(\tilde m)}\big), \big({\mathsf c}\big(\rho^{-1}(1)\big),\dots,{\mathsf c}\big(\rho^{-1}(\tilde m)\big)\big) \big). \label{Stildemaction}
\end{gather}

\begin{lem} Any master function of the form~\eqref{mf} is invariant under the pull-back of this action of $S_{\tilde m}$. In particular the extended master function~\eqref{emf} is invariant.
\end{lem}

Let us call a point $(\bm t,\bm {\mathsf c})\in {\mathbb C}^{\tilde m}\times I^{\tilde m}$ a \emph{cyclotomic point} if we have $\tilde m = M m$ for some $m\in {\mathbb Z}_{\geq 0}$ and, by acting with some permutation in $S_{\tilde m}$, we can arrange that
\begin{gather}
 \t_{i + mk} = \omega^k \t_i \qquad {\mathsf c}(i + mk) = \sigma^k{\mathsf c}(i), \qquad i=1,\dots, m, \quad k=0,\dots,M-1. \label{symt}
 \end{gather}

\begin{lem}\label{lem: cp} This point
$(\t_i)_{i=1}^{\tilde m}$ is a critical point of the extended master function if and only if~$(\t_i)_{i=1}^m$ is a critical point of the cyclotomic master function, i.e., $(\t_i)_{i=1}^m$ obeys~\eqref{cbe}.
\end{lem}

\begin{proof} Given~\eqref{symt}, the equation \eqref{ebe} for $\t_{j}$ is nothing but the corresponding equation in~\eqref{cbe} and the equation for $\t_{j+km}$, $k=1,\dots,M-1$, is actually the same equation up to an overall factor of $\omega^{-k}$.
(To see this one must use the compatibility of $\sigma$ with the inner product: $\lf \sigma x,y\rf = \lf x,\sigma^{-1} y\rf$.)
\end{proof}

Thus, the cyclotomic Bethe equations \eqref{cbe} are also the equations for cyclotomic critical points of the extended master function.

\section{Gaudin models and the Bethe ansatz equations}\label{gsec}
Our f\/irst result, Theorem \ref{matchthm}, concerns the relationship between critical points and the eigenvalues of Gaudin Hamiltonians. Suppose, for this section only, that the Cartan matrix is of f\/inite type, i.e., that $\g$ is semisimple, and that $\sigma$ is an automorphism of $\g$ of order $M>1$.
Recall \cite{Gaudin,Gbook} that the \emph{quadratic Gaudin Hamiltonians} are the following $\tilde N+1$ elements of $U(\g)^{\otimes (\tilde N+1)}$:
\begin{gather*}
 \tilde{\mathcal H}^{(i)} := \sum_{\substack{j=0\\j\neq i}}^{\tilde N} \sum_{a=1}^{\dim \g} \frac{ I^{a(i)} I_a^{(j)}}{\tilde z_i - \tilde z_j} ,\qquad i=0,1,\dots, \tilde N,
 \end{gather*}
where $I_a$, $a=1,\dots,\dim\g$, is a basis of~$\g$, $I^a$ is the dual basis with respect to the non-degenerate invariant bilinear form $\lf\cdot,\cdot\rf\colon \g\times \g\to {\mathbb C}$, and we write~${X}^{(i)}$ for~$X$ acting in the $i$th tensor factor. (For convenience we number these factors starting from $0$.)

For $\Lambda\in \h^*$, let $M_{\Lambda}$ denote the Verma module over $\g$ with highest weight~$\Lambda$,  $M_\Lambda := \operatorname{Ind}_{\h \oplus \n_+}^\g {\mathbb C} \mathsf v_{\Lambda}$. Let us represent the $\tilde{\mathcal H}^{(i)}$ as linear maps in $\operatorname{End}\big(\bigotimes_{i=0}^{\tilde N} M_{\tilde\Lambda_i}\big)$. Then the following can be shown using the techniques of the Bethe Ansatz.

\begin{thm}[\cite{BabujianFlume,RV}]\label{Gthm}
To any critical point $\bm t$ of the master function ${\widetilde \Phi}$, i.e., to any solution to the equations~\eqref{be}, there corresponds a simultaneous eigenvector
$\tilde\psi_{\bm t}$
of the linear operators $\tilde{\mathcal H}^{(i)} \in \operatorname{End}\big(\bigotimes_{i=0}^{\tilde N} M_{\tilde \Lambda_i}\big)$. For each $i=0,\dots,\tilde N$ the eigenvalue of $\tilde{\mathcal H}^{(i)}$
on $\tilde\psi_{\bm t}$
is
\begin{gather}
 \tilde E^{(i)} :=
 \frac{\partial {\widetilde \Phi}}{\partial \tilde z_i}
= \sum_{\substack{j=0\\j\neq i}}^{\tilde N} \frac{\lf \tilde\Lambda_i, \tilde\Lambda_j\rf}{\tilde z_i-\tilde z_j}
- \sum_{j=1}^{\tilde m} \frac{\lf \Lambda_i,\alpha_{{\mathsf c}(j)} \rf}{\tilde z_i-\t_j}.
\label{ee}
\end{gather}
The eigenvector $\tilde\psi_{\bm t}$ is given explicitly by
\begin{gather} \tilde\psi_{\bm t} =
  \sum_{\bm n\in P_{\tilde m,\tilde N+1}}
  \bigotimes_{i=0}^{\tilde N} \frac{F_{c(n^i_{1})}F_{c(n^i_2)} \cdots F_{c(n^i_{p_{i}-1})} F_{c(n^i_{p_i})} \mathsf{v}_{\tilde\Lambda_i}}
   {\big(w_{n^i_1} - w_{n^i_2}\big)\cdots
    \big(w_{n^i_{p_{i}-1}} -  w_{n^i_{p_i}}\big)
        \big(w_{n^i_{p_i}} -   z_i\big)  },
\label{psi}
\end{gather}
where the sum $\bm n\in P_{\tilde m,\tilde N+1}$ is over ordered partitions of the labels $\{1,\dots,\tilde m\}$ into $\tilde N+1$ parts.
\end{thm}

(The fact that this simultaneous eigenvector is nonzero is proved for $\g={\mathfrak{sl}}_n$ nondegenerate critical points in \cite{MV00}, for $\g={\mathfrak{sl}}_n$ isolated critical points in \cite{MTV1}, and for semisimple $\g$ and isolated critical points in \cite{V11}. See also \cite{V6}.)

In \cite{VY1}\footnote{In \cite{VY1} $\sigma\colon \g\to \g$ is allowed to be any automorphism commuting with the Cartan involution, not necessarily a~diagram involution. A posteriori the Bethe equations and energy eigenvalues depend on the inner part of $\sigma$ only through the def\/inition of~$\Lambda_0$, \eqref{l0def}.}, B. Vicedo and one of the present authors def\/ined \emph{cyclotomic Gaudin Hamiltonians}. The quadratic cyclotomic Gaudin Hamiltonians are the elements of $U(\g)^{\otimes N}$ given by
\begin{gather} \mathcal{H}_i :=
\sum_{p = 0}^{M-1} \sum_{\substack{j=1\\j \neq i}}^N\sum_{a=1}^{\dim\g} \frac{I^{a (i)} \sigma^p I_a^{(j)}}{z_i - \omega^{-p} z_j}
 + \frac 1 {z_i} \sum_{p = 1}^{M-1} \sum_{a=1}^{\dim\g}\frac{I^{a (i)} \sigma^p I_a^{(i)}}{(1 - \omega^p)}, \qquad i = 1, \ldots, n.
 \label{Hioo}
 \end{gather}

\begin{rem} These Hamiltonians can be understood in a number of ways. Physically, one thinks of them as describing the dynamics of a ``long-range spin chain'' in which the ``spin'' at~$z_i$  interacts not only directly with the other spins at the points~$z_j$, $j\neq i$, but also with their images under rotations of the spectral plane~\cite{CrampeYoung}. At the level of the Lax matrix, this corresponds to replacing the usual rational skew-symmetric solution to the classical Yang--Baxter equation, $r(u,v) = I^a \otimes I_a/(u-v)$, by a certain \emph{non}-skew-symmetric solution -- see \cite{Skrypnyk1, Skrypnyk2} and discussion in~\cite{VY1}.
The motivation for such models comes in part from physics, where in certain important cases the Lax matrix has cyclotomic symmetry in the spectral variable~\cite{KY, Y}.
\end{rem}

Let us assign to the point $z_i$ the Verma module $M_{\Lambda_i}$, $\Lambda_i\in \h^*$. In other words, let us represent the Hamiltonians~\eqref{Hioo} as linear maps
\begin{gather}
 \mathcal H^{(i)} \in \operatorname{End}\left(\bigotimes_{i=1}^N M_{\Lambda_i}\right),\qquad i=1,\dots, N.\label{Hi}
\end{gather}

Let (in this section, Section~\ref{gsec}) $\Lambda_0\in \h^{\sigma,*}$  be the weight given by
\begin{gather}
\Lambda_0(h)  := \sum_{r=1}^{M-1} \frac {\operatorname{tr}_\n (\sigma^{-r} \operatorname{ad}_h)} {1 - \omega^r}.\label{l0def}
\end{gather}

\begin{thm}[\cite{VY1}]\label{VYthm} To any  critical point of the cyclotomic master function, i.e., to any solution~$\bm \wc$ to the equations~\eqref{cbe}, there corresponds a simultaneous eigenvector
$\psi_{\bm \wc}$ of the linear operators~$\mathcal H^{(i)}$, $i=1,\dots, N$. The eigenvalue of $\mathcal H^{(i)}$ on $\psi_{\bm \wc}$ is
\begin{gather}
E^{(i)} := \frac{\partial \Phi}{\partial z_i} =  \sum_{\substack{j=1\\j\neq i}}^N \sum_{s=0}^{M-1} \frac{\lf \Lambda_i,\sigma^s \Lambda_j\rf}{z_i-\omega^sz_j}
- \sum_{j=1}^m \sum_{s=0}^{M-1} \frac{\lf \Lambda_i,\sigma^s \alpha_{{\mathsf c}(j)} \rf}{z_i-\omega^s\wc_j}\nonumber\\
\hphantom{E^{(i)} :=}{}
+ \frac{1}{z_i} \left( \lf \Lambda_i, \Lambda_0\rf
+ \sum_{s=1}^{M-1} \frac{\lf \Lambda_i,\sigma^s \Lambda_i\rf}{1-\omega^s}\right).
\label{eev}
\end{gather}
The explicit form of the eigenvector $\psi_{\bm\wc}$ is
\begin{gather}
\psi_{\bm \wc} =\label{cycpsi}\\
  =\sum_{\substack{\bm n\in P_{m,N}\\  (k_1,\dots,k_m) \in {\mathbb Z}_M^m}}
  \bigotimes_{i=1}^N \frac{ \check\sigma^{k_{n^i_1}}\big(F_{c(n^i_{1})}\big)\check\sigma^{k_{n^i_2}}\big(F_{c(n^i_2)}\big)\cdots \check\sigma^{k_{n^i_{p_{i}-1}}}\big(F_{c(n^i_{p_{i}-1})}\big) \check\sigma^{k_{n^i_{p_i}}}\big(F_{c(n^i_{p_i})}\big) \mathsf{v}_{\Lambda_i}}
   {\big(\omega^{k_{n^i_1}} w_{n^i_1} - \omega^{k_{n^i_2}} w_{n^i_2}\big)\cdots
    \big(\omega^{k_{n^i_{p_{i}-1}}}w_{n^i_{p_{i}-1}} - \omega^{k_{n^i_{p_i}}} w_{n^i_{p_i}}\big)
        \big(\omega^{k_{n^i_{p_i}}}w_{n^i_{p_i}} -   z_i\big)  },
\nonumber
\end{gather}
where $\check\sigma(X):=\omega\sigma(X)$.
\end{thm}
(It is an interesting open problem to determine under what circumstances the vector $\psi_{\bm \wc}$ is non-zero.)

On the other hand, consider the (usual) quadratic Gaudin Hamiltonians in the special case~\eqref{symz}. We refer to this situation as the \emph{extended Gaudin model}, and write $ \tilde{\mathcal H}^{(i)} $ as $\mathcal H^{(i)}_{\rm ext}$.  Note that
\begin{gather}
 \mathcal H^{(i)}_{\rm ext}  \in \operatorname{End}\left(M_{\Lambda_0}\otimes  \bigotimes_{k=0}^{M-1} \bigotimes_{i=1}^N M_{\sigma^k\Lambda_i}\right), \qquad i = 0,1,\dots, N M. \label{Hti}
\end{gather}
The following is then a corollary of Theorem~\ref{Gthm}.

\begin{cor}
To any  critical point of the cyclotomic master function, i.e., to any solution~$\bm\wc$ to the equations~\eqref{cbe}, there corresponds a simultaneous eigenvector of the linear opera\-tors~$\mathcal H^{(i)}_{\rm ext}$, $i=0,1,\dots,nM$, such that $\mathcal H^{(0)}_{\rm ext}$ has eigenvalue zero and, for each $k=0,\dots, M-1$ and $i=1,\dots, N$, the eigenvalue of $\mathcal H^{(k+M i)}_{\rm ext}$ is given by $\omega^{-k}E_i$ with $E_i$ as in~\eqref{eev}.
\end{cor}

\begin{proof}
Let $\bm t$ be the corresponding (by Lemma \ref{lem: cp}) cyclotomic critical point of the extended master function $\widehat\Phi$. Then the result is a special case of Theorem \ref{Gthm}, by substituting~\eqref{symz} and~\eqref{symt} into \eqref{ee}. (To see that $\mathcal H^{(0)}_{\rm ext}$ has eigenvalue zero note that
\begin{gather*}
\sum_{i=1}^N \sum_{s=0}^{M-1} \frac{\lf \Lambda_0, \sigma^s \Lambda_i\rf}{0- \omega^s z_i}
- \sum_{j=1}^m \sum_{s=0}^{M-1}  \frac{\lf \Lambda_0, \sigma^s \alpha_{{\mathsf c}(j)}\rf}{0- \omega^s \wc_j} =0,
\end{gather*}
because $\sum\limits_{s=0}^{M-1} \omega^{-s}\sigma^{-s} \Lambda_0 = \Lambda_0\sum\limits_{s=0}^{M-1} \omega^{-s} = 0$ since $\sigma\Lambda_0=\Lambda_0$ and $M>1$.)
\end{proof}

In summary, we have the following observation.
\begin{thm}\label{matchthm}
To any critical point of the cyclotomic master function
there corresponds both a~simultaneous eigenvector~\eqref{cycpsi} of the Hamiltonians~$\mathcal H^{(i)}$ of the cyclotomic Gaudin model and a~simultaneous eigenvector~\eqref{psi} of the Hamiltonians $\mathcal H^{(i)}_{\rm ext}$ of the extended Gaudin model, $i=1,\dots,n$, with the corresponding eigenvalues equal and in both cases being given by~\eqref{eev}.
\end{thm}

\begin{rem} The operators $\mathcal H^{(i)}$ and $\mathcal H^{(i)}_{\rm ext}$ are acting in dif\/ferent spaces, \eqref{Hi}~and~\eqref{Hti} respectively. It would be interesting to relate these operators by some means independent of the Bethe ansatz.
\end{rem}

\section{Cyclotomic generation procedure}\label{cycgensec}
In \cite{MV1, ScV} a procedure was introduced which generates new critical points of master functions starting from a given initial critical point. There is an ``elementary generation'' step associated to each $i\in I$. The Zariski closure of the collection of all critical points obtained by recursively applying elementary generations in all possible ways is called the ``population'' to which the initial critical point belongs.

The extended master functions, \eqref{emf} above, are master functions of the standard form (unlike the cyclotomic master functions~\eqref{cmf}). Modulo subtleties coming from the fact that the weight~$\Lambda_0$ at the origin need not be dominant integral, that means the generation procedure can be applied.

In this section we describe this generation procedure and go on to show how, given a cyclotomic critical point, one can obtain new cyclotomic critical points by applying the elementary generation steps in certain carefully chosen combinations. The resulting collections of cyclotomic critical points will be called ``cyclotomic populations''.

\subsection[Conditions on $\Lambda_0$]{Conditions on $\boldsymbol{\Lambda_0}$}\label{l0sec}
In the remainder of the paper we assume that $\sigma$ is a diagram automorphism obeying the linking condition~\eqref{lc}. That means for each $i\in I$, either $L_i=1$ or $L_i=2$.

In addition, in this section, Section~\ref{cycgensec}, we place the following conditions on the weight \mbox{$\Lambda_0\in \h^{\sigma,*}$}.

For each $i\in I$ such that $L_i=1$, we suppose that
\begin{gather}
 \langle \Lambda_0,\alpha_i^\vee \rangle \in {\mathbb Z}_{\geq 0} \label{s1L1}
\end{gather}
and
\begin{gather}
 \langle \Lambda_0,\alpha_i^\vee \rangle +1 \equiv 0 \mod M/M_i. \label{s1L2}
\end{gather}

For each $i\in I$ such that $L_i=2$, we suppose that
\begin{gather}
2\langle \Lambda_0,\alpha^\vee_i \rangle + 1 \in {\mathbb Z}_{\geq 0}. \label{s2L1}
\end{gather}

\begin{rem} One can verify that these conditions are satisf\/ied by the weight $\Lambda_0$ of~\eqref{l0def} in the case of diagram automorphisms of f\/inite-type Dynkin diagrams. Our assumptions on $\Lambda_0$ in the treatment of type~$A_R$ in Section~\ref{ARsec} below are weaker.
\end{rem}

\subsection{Tuples of polynomials}\label{prsec}

To any pair $(\bm t; \bm {\mathsf c})$ with $\bm t\in {\mathbb C}^{\tilde m}$ and $\bm{\mathsf c} \in I^{\tilde m}$, we may associate a tuple of polynomials ${\bm y} = (y_1(x),\dots, y_r(x))$, given by
\begin{gather}
y_i(x) := \prod_{\substack{j=1\\ {\mathsf c}(j) = i}}^{\tilde m} (x- \t_j), \qquad i\in I.\label{ydef}
\end{gather}
We say that this tuple $\bm y$ \emph{represents the pair $(\bm t; \bm {\mathsf c})$}.
We consider each coordinate~$y_i(x)$ only up to multiplication by a non-zero complex number, since we are only concerned with their zeros. So the tuple $\bm y$ def\/ines a point in the direct product ${\mathbb P}({\mathbb C}[x])^{|I|}$ of $|I|$ copies of the projective space~${\mathbb P}({\mathbb C}[x])$, where ${\mathbb C}[x]$ is the vector space of complex polynomials in~$x$.

Conversely, given any $\bm y\in {\mathbb P}({\mathbb C}[x])^{|I|}$ we may extract the pair $(\bm\t;\bm{\mathsf c})\in {\mathbb C}^{\tilde m}\times I^{\tilde m}$ such that \eqref{ydef} holds.  This pair is unique up to permutation by an element of $S_{\tilde m}$; see~\eqref{Stildemaction}.

Def\/ine $T_i(x)$, $i\in I$, to be
\begin{gather}
T_i(x)
 :=
\prod_{s=1}^N \prod_{k=0}^{M-1}\big(x- \omega^k z_s\big)^{\langle \sigma^k\Lambda_s,\alpha^\vee_i\rangle}.\label{Tdef}
\end{gather}

We say that a tuple of polynomials ${\bm y} = (y_i(x))_{i\in I} \in {\mathbb P}({\mathbb C}[x])^{|I|}$ is \emph{generic $($with respect to~$(T_i(x))_{i\in I})$} if for each $i\in I$,
$y_i(x)$ has no root in common with $T_i(x)$, or with any $y_j(x)$, $j\in I{\setminus}\{i\}$, such that $\langle\alpha_j,\alpha^\vee_i\rangle \neq 0$.

Note that if $\bm y$ represents a critical point of the extended master function $\widehat\Phi(\bm\t;\bm{\mathsf c};\bm z;\bm \Lambda)$, \eqref{emf}, i.e., its roots obey~\eqref{ebe}, then the tuple~$\bm y$ must be generic. (Indeed, if~\eqref{ebe} holds then in particular each summand on the left hand side of~\eqref{ebe} must have non-zero denominator. By def\/inition that implies that the corresponding tuple is generic.)

\subsection[Elementary generation: the $L_i=1$ case]{Elementary generation: the $\boldsymbol{L_i=1}$ case}

Throughout this subsection, we suppose $i\in I$ is such that $L_i=1$. That means that the simple roots $\alpha_{\sigma^k i}$, $i=1,\dots, M_i$, are mutually orthogonal. Equivalently it means that the ref\/lections ${\mathsf s}_{\sigma^ki}\in W$, $i=1,\dots,M_i$, are mutually commuting.

Let $y_i^{(i)}(x)$ be of the form
\begin{gather}
 y_i^{(i)}(x) = y_i(x) \int^x\xi^{\langle \Lambda_0,\alpha^\vee_i\rangle} T_i(\xi) \prod_{j\in I} y_j(\xi)^{-\langle\alpha_j,\alpha^\vee_i\rangle}d\xi
,\label{yiidef}
\end{gather}
so that $y_i^{(i)}(x)$ is a solution to the equation
\begin{gather}
 \operatorname{Wr}(y_i(x),y_i^{(i)}(x)) =x^{\langle \Lambda_0,\alpha^\vee_i\rangle} T_i(x) \prod_{j\in I{\setminus}\{i\}} y_j(x)^{-\langle\alpha_j,\alpha^\vee_i\rangle} ,\label{wre}
\end{gather}
where $\operatorname{Wr}(f(x),g(x)) := f(x) g'(x) -f'(x) g(x)$ denotes the Wronskian determinant.

\begin{prop}\label{pprop}
If  $\bm y$ represents a critical point then $y_i^{(i)}(x)$ is a polynomial.
\end{prop}

\begin{proof} We have $\langle \Lambda_0 ,\alpha^\vee_i \rangle \in {\mathbb Z}_{\geq 0}$ as in~\eqref{s1L1}, and for each $s\in\{1,\dots,N\}$, $\Lambda_s$ is integral dominant so $\langle \Lambda_s ,\alpha^\vee_i \rangle \in {\mathbb Z}_{\geq 0}$. So the integrand is a rational function with poles at most at the points~$\t_p$, $p\in \{1,\dots,m\}$, for which ${\mathsf c}(p)=i$. Consider such a point~$\t_p$.
Note that
\begin{gather} \frac{\partial}{\partial x} \log x^{\langle \Lambda_0,\alpha^\vee_i\rangle} T_i(x) (x-\t_p)^2 \prod_{j\in I} y_j(x)^{-\langle\alpha_j,\alpha^\vee_i\rangle}
\nonumber\\
\qquad{}=
 \sum_{k=0}^{M-1} \sum_{i=1}^N
\frac{\langle \sigma^k\Lambda_i, \alpha^\vee_{{\mathsf c}(p)} \rangle}{x-\omega^kz_i}
+ \frac{\langle  \Lambda_0,\alpha^\vee_{{\mathsf c}(p)} \rangle }x
-  \sum_{\substack{i=1\\i\neq p}}^{\tilde m}
\frac{\langle\alpha_{{\mathsf c}(i)}, \alpha_{{\mathsf c}(p)}\rangle}{x-\t_i}.\label{ld}
\end{gather}
This vanishes at $x=\t_p$ by virtue of the critical point equations \eqref{ebe}. It follows that the residue of the integrand at $\t_p$ vanishes: indeed, this residue is
\begin{gather*}
\left(\frac{\partial}{\partial x}x^{\langle \Lambda_0,\alpha^\vee_i\rangle} T_i(x) (x-\t_p)^2 \prod_{j\in I} y_j(x)^{-\langle\alpha_j,\alpha^\vee_i\rangle}\right)_{x=\t_p},
 \end{gather*}
 which vanishes if~\eqref{ld} vanishes.
This shows that $y_i^{(i)}(x)$ is an entire function. It is of polynomial growth for large $x$. Therefore it is a polynomial.
\end{proof}

If $y_{i}^{(i)}(x)$ is any solution to \eqref{wre} then so too is $y_i^{(i)}(x)+ cy_i(x)$ for any $c\in {\mathbb C}$.

Thus, given any tuple $\bm y$ representing a critical point we have,
for each value of a parameter $c\in {\mathbb C}$, a new tuple of polynomials $\bm y^{(i)}$, obtained from the tuple $\bm y$ by replacing $y_i(x)$ with $y_i^{(i)}(x)+ cy_i(x)$.
We say~$\bm y^{(i)}$ is obtained from~$\bm y$ by \emph{generation in the $i$th direction}, and we call~$\bm y^{(i)}$ the \emph{immediate descendant of~$\bm y$ in the $i$th direction}.

\begin{prop}[\cite{MV1}]\label{mv1prop}
The tuple of polynomials $\bm y^{(i)}$ is generic for almost all~$c$.
If $\bm y^{(i)}$ is generic then it represents a critical point.
\end{prop}
The tuples $\bm y^{(i)}$ describe a projective line in ${\mathbb P}({\mathbb C}[x])^{|I|}$.
It will be useful to have the following specif\/ic parameterization of this line.
There exists a unique solution $y_i^{(i)}(x)$ to the equation~\eqref{wre}, call it $y_{i}^{(i)}(x;0)$, such that the coef\/f\/icient of $x^{\deg y_i}$ in $y^{(i)}_i(x;0)$ is zero. Let us def\/ine
\begin{gather}
y_i^{(i)}(x;c) := y^{(i)}_i(x;0) + c y_i(x),\label{yicdef}
\end{gather}
and def\/ine $\bm y^{(i)}(c)\in {\mathbb P}({\mathbb C}[x])^{|I|}$ to be the tuple obtained from the tuple $\bm y$ by replacing $y_i(x)$ with $y^{(i)}_i(x;0)+ cy_i(x)$.

We say generation in the $i$th direction is \emph{degree-increasing} if $\deg y_i^{(i)} > \deg y_i$ for almost all~$c$.

Recall that there is a weight at inf\/inity, $\Lambda_\infty$, associated to any critical point. For the critical point represented by~$\bm y$ this weight is, cf.~\eqref{l8def},
\begin{gather}
 \Lambda_\infty(\bm y) = \Lambda_0 + \sum_{s=1}^N\sum_{k=0}^{M-1}\sigma^k\Lambda_s- \sum_{j\in I} \alpha_j\deg y_j .\label{l8y}
\end{gather}
For \looseness=-1 f\/ixed $\Lambda_0, \Lambda_1,\dots,\Lambda_N$ we can think of $\Lambda_\infty$ as encoding the degrees of the polynomials~$y_j$.
Note that $\deg y_i^{(i)}(x;0) = \deg y_i + \langle \Lambda_\infty, \alpha_i^\vee \rangle +1$. It follows that the weight at inf\/inity of~$\bm y^{(i)}(0)$ is
\begin{gather*}
 \Lambda_\infty - \alpha_i\big(\langle\Lambda_\infty,\alpha^\vee_i\rangle +1\big)
 = \Lambda_\infty - \langle \Lambda_\infty +\rho, \alpha^\vee_i\rangle \alpha_i
= {\mathsf s}_i \cdot \Lambda_\infty.
\end{gather*}
This establishes the following lemma.

\begin{lem}\label{delem}
Generation in the $i$th direction $($with $L_i=1)$ is degree-increasing if and only if~$\Lambda_\infty$ is $i$-dominant, i.e., $\langle\Lambda_\infty,\alpha^\vee_i\rangle\in {\mathbb Z}_{\geq 0}$.

If generation in the $i$th direction is degree-increasing, then the weight at infinity associated with the critical point represented by $\bm y_i^{(i)}(c)$ is ${\mathsf s}_i \cdot \Lambda_\infty$.
Otherwise it is $\Lambda_\infty$ for all $c\neq 0$ $($and ${\mathsf s}_i\cdot \Lambda_\infty$ for $c=0)$.
\end{lem}

\subsection[Cyclotomic generation: the $L_i=1$ case]{Cyclotomic generation: the $\boldsymbol{L_i=1}$ case}
We continue to suppose that $i$ is such that $L_i=1$.

If $\bm y$ represents a cyclotomic point then its immediate descendant $\bm y^{(i)}$ in the $i$th direction generically does not. However if, starting from a cyclotomic critical point, we successively generate in \emph{each} of the directions~$\sigma^k i$, $k=1,\dots, M_i$, in turn, in any order, then we can arrange to arrive at a (new) cyclotomic critical point. This is the content of Theorem~\ref{s1thm} below.

Let $\simeq$ denote equality up to a constant (independent of~$x$) nonzero factor. Recall the def\/inition~\eqref{symt} of a cyclotomic point.
\begin{lem} \label{cycptlem}
A tuple of polynomials $\bm y$ represents a cyclotomic point if and only if
\begin{gather*}
  y_{\sigma j}( \omega x) \simeq y_j(x)
\end{gather*}
for all $j\in I$.
If $y_j(x)$ and $y_{\sigma j}(x)$ share the same leading coefficient for all $j\in I$, then the tuple~$\bm y$ represents a cyclotomic point if and only if
\begin{gather*}
  y_{\sigma j}( \omega x) = \omega^{\deg y_j} y_j(x)
\end{gather*}
for all $j\in I$.
\end{lem}

For the rest of this subsection, we suppose $\bm y$ represents a cyclotomic critical point. Hence in particular $\sigma \Lambda_\infty=\Lambda_\infty$.
Let $y^{(i)}_i(x;c) = y^{(i)}_{i,0}(x) + c y_i(x)$ be as in~\eqref{yicdef}. (So $y^{(i)}_i(x;c)$ is a pa\-ra\-meterization of the space of solutions to~\eqref{wre}.) Def\/ine $\bm y^{(i,\sigma)}(c)$ to be the tuple of polynomials given by
\begin{gather*}
y^{(i,\sigma)}_{\sigma^ki}\big(\omega^kx;c\big)  := \omega^{k\deg y^{(i)}_i} y^{(i)}_i(x;c),
\qquad k=0,1,\dots,M_i-1,
 \end{gather*}
and $y^{(i,\sigma)}_{j}(x;c) = y_j(x)$ for $j \in I {\setminus} \sigma^{\mathbb Z} i$.
Recall ${\mathsf s}_i^\sigma$ from Lemma~\ref{Wsiglem}.

\begin{thm}\label{s1thm} For almost all $c\in {\mathbb C}$, the tuple $\bm y^{(i,\sigma)}(c)$ represents a cyclotomic critical point. The exceptional values of $c$ form a finite subset of ${\mathbb C}$.

The weight at infinity of  $\bm y^{(i,\sigma)}(c)$ is ${\mathsf s}^\sigma_i \cdot\Lambda_\infty$ if $\langle \Lambda_\infty,\alpha^{\vee,\sigma}_i \rangle \in {\mathbb Z}_{\geq 0}$. Otherwise it is $\Lambda_\infty$ for all $c\neq 0$, and ${\mathsf s}^\sigma_i \cdot\Lambda_\infty$ for $c=0$.
\end{thm}
\begin{proof}
First let us show that $\bm y^{(i,\sigma)}$ represents a cyclotomic point for all $c\in {\mathbb C}$. Comparing our def\/inition of  $\bm y^{(i,\sigma)}$ with the criterion in Lemma \ref{cycptlem}, one sees that it is enough to check that
\begin{gather*}
 y^{(i)}_i\big(\omega^{M_i} x;c\big) = \omega^{M_i\deg y^{(i)}_i} y^{(i)}_i(x;c).
 \end{gather*}
Inspecting \eqref{yiidef}, we see that this equality holds for all $c\in {\mathbb C}$ if and only if
\begin{gather}
 \omega^{M_i\langle\Lambda_\infty+\rho,\alpha_i^\vee \rangle} = 1.\label{l8cond}
\end{gather}
But now, given~\eqref{l8y} and the assumption that $\Lambda_s$, $s=1,\dots,n$ are integral, the following lemma implies that~\eqref{l8cond} holds if and only if we impose the condition~\eqref{s1L2} on~$\Lambda_0$.

\begin{lem} Suppose $\Lambda\in \h^*$ is an integral weight. Then, for any $j\in I$,
\begin{gather*}
  \sum_{k=0}^{M-1} \langle\sigma^k \Lambda, \alpha^\vee_j \rangle M_j \equiv 0 \mod M.
\end{gather*}
\end{lem}

\begin{proof}
We have
\begin{gather*}
\left\langle \sum_{k=0}^{M-1} \sigma^k\Lambda,\alpha^\vee_j \right\rangle M_j = \left\langle \Lambda, \sum_{k=0}^{M-1} \sigma^{-k} \alpha^\vee_j\right\rangle M_j \\
\hphantom{\left\langle \sum_{k=0}^{M-1} \sigma^k\Lambda,\alpha^\vee_j \right\rangle M_j }{}
= \left\langle \Lambda, \frac{M}{M_j}\sum_{k=0}^{M_j-1} \alpha^\vee_j \right\rangle M_j = M \left\langle \Lambda, \sum_{k=0}^{M_j-1} \alpha^\vee_j \right\rangle \in M{\mathbb Z}. \tag*{\qed}
\end{gather*}
\renewcommand{\qed}{}
\end{proof}

Now we show that $\bm y^{(i,\sigma)}$ represents a critical point for all but f\/initely many $c\in {\mathbb C}$.
Note that from def\/inition \eqref{Tdef} we have
\begin{gather}  T_{\sigma j}(\omega x) =
 \omega^{\big\langle \sum\limits_{s=1}^N \sum\limits_{k=0}^{M-1} \sigma^k \Lambda_s, \alpha^\vee_j \big\rangle} T_j(x) , \qquad j\in I.\label{oT}
\end{gather}
Hence, in view of~\eqref{l8y},
\begin{gather*}  x^{\langle\Lambda_0,\alpha^\vee_{\sigma i}\rangle} T_{\sigma i}(\omega x) \prod_{j\in I} y_j(\omega x)^{-\langle\alpha_j,\alpha^\vee_{\sigma i}\rangle}
 =  x^{\langle\Lambda_0,\alpha^\vee_i\rangle} T_{\sigma i}(\omega x) \prod_{j\in I} y_{\sigma j}(\omega x)^{-\langle\alpha_{\sigma j},\alpha^\vee_{\sigma i}\rangle} \\
\hphantom{x^{\langle\Lambda_0,\alpha^\vee_{\sigma i}\rangle} T_{\sigma i}(\omega x) \prod_{j\in I} y_j(\omega x)^{-\langle\alpha_j,\alpha^\vee_{\sigma i}\rangle}}{}
 =  \omega^{\langle \Lambda_\infty, \alpha^\vee_i \rangle}   \left(x^{\langle\Lambda_0,\alpha^\vee_i\rangle} T_{i}(x) \prod_{j\in I} y_j( x)^{-\langle\alpha_j,\alpha^\vee_{i}\rangle}\right).
\end{gather*}
Note also that since $L_i=1$, no node $j$ in the orbit of $i$ is linked by an edge of the Dynkin diagram to~$i$. That is, no~$y_j$ for~$j$ in the orbit of $i$ appears on the right of~\eqref{wre}.
Hence, for $k=1,\dots,M_i-1$,  $y^{(i,\sigma)}_{\sigma^k i}(x;c)$ obeys the equation
\begin{gather*}
 \operatorname{Wr}(y_{\sigma^k i}(x), y^{(i,\sigma)}_{\sigma^k i}(x;c)) = x^{\langle\Lambda_0,\alpha^\vee_{\sigma^k i}\rangle} T_{\sigma^k i}(x) \prod_{j\in I{\setminus} \{\sigma^ki\}} y_j(x)^{-\langle\alpha_j,\alpha^\vee_{\sigma^ki}\rangle}.
 \end{gather*}
and the tuple $\bm y^{(i,\sigma)}$ is indeed the result of generating in each of the directions $i, \sigma i, \dots, \sigma^{M_i-1}i$ (in any order).
It follows from Proposition~\ref{mv1prop} that~$\bm y^{(i,\sigma)}$ is generic for almost all~$c$, and represents a critical point whenever it is generic.

The statements about the weight at inf\/inity follow from Lemma~\ref{delem} and Section~\ref{foldsec}.
This completes the proof of Theorem~\ref{s1thm}.
\end{proof}

\subsection[Elementary generation: the $L_i=2$ case]{Elementary generation: the $\boldsymbol{L_i=2}$ case}\label{sgen1}

For this subsection we suppose that $i\in I$ is such that $L_i=2$. That implies $M_i$ is even and the restriction of the Dynkin diagram to the nodes $\sigma^{\mathbb Z} i$ consists of $\frac{M_i}{2}\in {\mathbb Z}_{\geq 1}$ disconnected copies of the Dynkin diagram of type ${\mathrm A}_2$, as sketched below:{\samepage
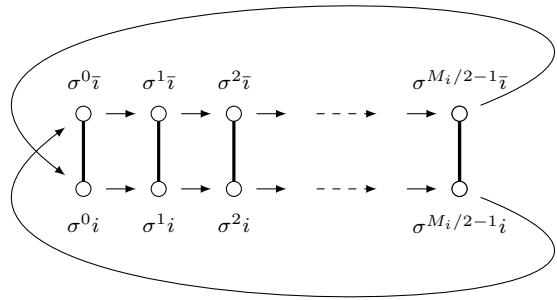
\begin{gather*}
\begin{tikzpicture}[baseline =-5,scale=1,font=\scriptsize]
\foreach \i in {0,1,2} {
  \draw[very thick] (\i+1,0) -- (\i+1,1);
   \filldraw[fill=white] (\i+1,0) circle (1mm) node[below=2mm] {$\sigma^\i i$}
                         (\i+1,1) circle (1mm) node[above=2mm] {$\sigma^\i \bi$};
  \draw[-latex,shorten >=3mm,shorten <=3mm] (\i+1,0) -- (\i+2,0);
  \draw[-latex,shorten >=3mm,shorten <=3mm] (\i+1,1) -- (\i+2,1);
}
\foreach \i in {6} {
  \draw[very thick] (\i,0) -- (\i,1);
   \filldraw[fill=white] (\i,0) circle (1mm) node[below=2mm] {$\sigma^{M_i/2-1} i$}
                         (\i,1) circle (1mm) node[above=2mm] {$\sigma^{M_i/2-1} \bi$};
  \draw[-latex,shorten >=3mm,shorten <=3mm] (\i-1,0) -- (\i,0);
  \draw[-latex,shorten >=3mm,shorten <=3mm] (\i-1,1) -- (\i,1);
}
  \draw[-latex,dashed,shorten >=1mm,shorten <=1mm] (4,0) -- (4+1,0);
  \draw[-latex,dashed,shorten >=1mm,shorten <=1mm] (4,1) -- (4+1,1);
\draw[very thick] (6,0) -- (6,1); \filldraw[fill=white] (6,0) circle (1mm) (6,1) circle (1mm);
\draw[-latex,shorten >=3mm,shorten <=3mm] (6,0) .. controls (11,-2) and  (-3,-2) ..  (1,1);
\draw[-latex,shorten >=3mm,shorten <=3mm] (6,1) .. controls (11,3) and  (-3,3) ..  (1,0);
\end{tikzpicture}
\end{gather*}
Here, for brevity, we write $\bi := \sigma^{M_i/2} i$.}

\begin{rem} Among f\/inite and af\/f\/ine types, only the case $M_i/2=1$ occurs.
\end{rem}

We def\/ine $y_i^{(i)}(x)$ by
\begin{gather*} y_i^{(i)}(x) := y_i(x)  x^{-{\langle\Lambda_0,\alpha^\vee_i\rangle}-1}  \int_0^x \xi^{{\langle\Lambda_0,\alpha^\vee_i\rangle}}T_i(\xi) \prod_{j\in I} y_j(\xi)^{-\langle\alpha_j,\alpha^\vee_i\rangle}d\xi.
\end{gather*}
Here the limits $\int_0^x$ mean that $y^{(i)}_i(x)$
is holomorphic at $x=0$. This condition def\/ines the integral uniquely, since $\langle\Lambda_0,\alpha^\vee_i\rangle \notin {\mathbb Z}$ by our assumption~\eqref{s2L1}.

\begin{prop}\label{pprop2}
If  $\bm y$ represents a critical point then $y^{(i)}_i(x)$ is a polynomial. It has degree $\deg y^{(i)}_i = \deg y_i + \langle \Lambda_\infty-\Lambda_0, \alpha^\vee_i\rangle$.
\end{prop}

\begin{proof}
The proof is as for Proposition~\ref{pprop}.
\end{proof}

Let $\bm y^{(i)}= (y^{(i)}_j(x))_{j\in I}$ be the tuple of polynomials whose $i$th component $y^{(i)}_i(x)$ is as above, and whose remaining components are the same as those of $\bm y$, i.e.,
\begin{gather*}
 y^{(i)}_j(x) = y_j(x)\qquad\text{ for all }\quad j\in I{\setminus} \{i\}.
 \end{gather*}
Let $(\bm \t^{(i)};\bm {\mathsf c}^{(i)})$ denote the pair represented by this tuple in the sense of Section~\ref{prsec}. It turns out that $\bm t^{(i)}$ is not in general a critical point of the extended master function $\widehat\Phi(\bm\t^{(i)};\bm{\mathsf c}^{(i)};\bm z;\bm \Lambda)$, i.e., it does not in general obey the equations \eqref{ebe}. Instead, the following result gives the analogous collection of equations that it does obey, provided~$\bm y^{(i)}$ is generic.

\begin{prop}\label{nb} If $\bm y$ represents a critical point
and $\bm y^{(i)}$ is generic, then
\begin{gather*}
 \frac{\big\langle{\mathsf s}_i\cdot\Lambda_0,\alpha^\vee_{{\mathsf c}^{(i)}(p)}\big\rangle}{\t^{(i)}_p}
+ \sum_{s=1}^N\sum_{k=0}^{M-1} \frac{\big\langle \sigma^k \Lambda_s,\alpha^\vee_{{\mathsf c}^{(i)}(p)}\big\rangle}{\t^{(i)}_p-\omega^k z_s}
- \sum_{r\colon r\neq p} \frac{\big\langle \alpha_{{\mathsf c}^{(i)}(r)},\alpha^\vee_{{\mathsf c}^{(i)}(p)} \big\rangle}{\t^{(i)}_p- \t^{(i)}_r} = 0
\end{gather*}
for each $p$.
\end{prop}
\begin{proof}
By \eqref{ebe} for each root $\t_p$ in the tuple $\bm t$ we have
\begin{gather} \frac{\big\langle\Lambda_0,\alpha^\vee_{{\mathsf c}(p)}\big\rangle}{\t_p} + \sum_{s=1}^N\sum_{k=0}^{M-1} \frac{\big\langle \sigma^k \Lambda_s,\alpha^\vee_{{\mathsf c}(p)}\big\rangle}{\t_p-\omega^k z_s} - \sum_{r\colon r\neq p} \frac{\big\langle \alpha_{{\mathsf c}(r)},\alpha^\vee_{{\mathsf c}(p)} \big\rangle}{\t_p- \t_r} = 0.\label{tpbe}
\end{gather}
For all roots of colours $j\in I$ such that $\langle \alpha_{j}, \alpha^\vee_i\rangle =0$ this is immediately equivalent to the required equation.
So we must consider roots of colour $i$, and roots of colours $j\in I$ such that $\langle \alpha_{j}, \alpha^\vee_i\rangle <0$.

By def\/inition of $y^{(i)}_i(x)$ we have
\begin{gather}
 \operatorname{Wr}(y_i(x),x^{{\langle\Lambda_0,\alpha^\vee_i\rangle}+1}y^{(i)}_i(x)) = x^{{\langle\Lambda_0,\alpha^\vee_i\rangle}} T_i(x) \prod_{j\neq i} y_j(x)^{-\langle\alpha_j,\alpha^\vee_i\rangle}\label{wreY}
\end{gather}
or equivalently
\begin{gather}
 \frac{y_i'(x)}{y_i(x)} - \frac{y^{(i)}_i{}'(x)}{y^{(i)}_i(x)} - \frac{1+{\langle\Lambda_0,\alpha^\vee_i\rangle}}{x}
 = \frac{T_i(x)  \prod\limits_{j\neq i} y_j(x)^{-\langle\alpha_j,\alpha^\vee_i\rangle}}{x y_i(x) y^{(i)}_i(x)} .
\label{yye}
\end{gather}
By def\/inition of $(\bm \t^{(i)}, \bm{\mathsf c}^{(i)})$, the left-hand side of \eqref{yye} is
\begin{gather}
\sum_{r\colon  {\mathsf c}(r) = i} \frac 1{x-\t_r} - \sum_{r\colon {\mathsf c}^{(i)}(r) = i} \frac 1 {x- \t^{(i)}_r} - \frac{1+{\langle\Lambda_0,\alpha^\vee_i\rangle}}{x}.\label{lyye}
\end{gather}

Now suppose $j\in I$ is such that $\langle \alpha_{j}, \alpha^\vee_i\rangle \in {\mathbb Z}_{<0}$. By def\/inition $y^{(i)}_j(x) = y_j(x)$. Suppose~$\t_p$ is a root of $y_j(x)$, i.e., suppose ${\mathsf c}(p)=j$. Since $\bm y$ represents a critical point, $\bm y$ must be generic, and hence $\t_p$ is not a root of $y_i(x)$. By our assumption that $\bm y^{(i)}$ is generic, $\t_p$ is not a root of~$y^{(i)}_i(x)$ either. Hence the right-hand side of~\eqref{yye} is zero at $x=\t_p$ and so, in view of~\eqref{lyye}, we have
\begin{gather*}
\sum_{r\colon {\mathsf c}(r) = i} \frac 1{\t_p-\t_r} - \sum_{r\colon {\mathsf c}^{(i)}(r) = i} \frac 1 {\t_p- \t^{(i)}_r} - \frac{1+{\langle\Lambda_0,\alpha^\vee_i\rangle}}{\t_p}
 = 0.
 \end{gather*}
On adding this equation multiplied by $\langle \alpha_{i}, \alpha^\vee_j\rangle$ to the equation~\eqref{tpbe}, we arrive at
\begin{gather*}
\frac{\langle\Lambda_0,\alpha^\vee_{j}\rangle - \langle \alpha_i,\alpha^\vee_j\rangle  \langle\Lambda_0+\rho, \alpha^\vee_i\rangle}{\t^{(i)}_p}
+ \sum_{s=1}^N\sum_{k=0}^{M-1} \frac{\langle \sigma^k \Lambda_s,\alpha^\vee_{j}\rangle}{\t^{(i)}_p-\omega^k z_s}
- \sum_{r\colon r\neq p} \frac{\langle \alpha_{{\mathsf c}^{(i)}(r)},\alpha^\vee_{j} \rangle}{\t^{(i)}_p- \t^{(i)}_r} = 0,
\end{gather*}
which is the required equality (since ${\mathsf s}_i\cdot \Lambda_0 = \Lambda_0 - \langle \Lambda_0 + \rho, \alpha^\vee_i\rangle \alpha_i$).

It remains to consider roots of colour $i$.
First note that $y_i(x)$ and $y^{(i)}_i(x)$ have no common roots. Indeed, if $\t$ were a common root of $y_i(x)$ and $y_i^{(i)}(x)$ then the right-hand side of~\eqref{wreY} would have to vanish at $x=\t$. In other words~$y_i(x)$ would have a root in common with the right-hand side of~\eqref{wreY}. But by our def\/inition of what it means for~$\bm y$ to be generic, Section~\ref{prsec}, this is impossible.

Suppose $\t^{(i)}_p$ is any root of $y^{(i)}_i(x)$. By our assumption that $\bm y^{(i)}$ is generic, it follows from~\eqref{wreY} and Lemmas~\ref{rlem} and~\ref{wrlem} below  that
\begin{gather*}
 \frac{ 2 (1+ {\langle\Lambda_0,\alpha^\vee_i\rangle})}{\t^{(i)}_p} - \frac{\langle\Lambda_0,\alpha^\vee_i\rangle}{\t^{(i)}_p}
- \sum_{s=1}^N\sum_{k=0}^{M-1} \frac{\langle \sigma^k \Lambda_s,\alpha^\vee_i\rangle}{\t^{(i)}_p-\omega^k z_s}
+ \sum_{r:r\neq p} \frac{\langle \alpha_{{\mathsf c}^{(i)}(r)},\alpha^\vee_i \rangle}{\t^{(i)}_p- \t^{(i)}_r} = 0,
\end{gather*}
which is the required equality.
\end{proof}

\begin{rem} Propositions \ref{pprop2} and \ref{nb} also follow from Theorem~3.5 in~\cite{MV2}.\end{rem}

\begin{lem}\label{rlem}
For any $\alpha\in {\mathbb C}$, if $g(x) = x^\alpha\prod\limits_{j=1}^J (x- s_j)$, where $(s_j)_{j=1}^J$ are all distinct and non-zero, then
\begin{gather*}
 \left.\frac{g''(x)}{g'(x)}\right|_{x=s_k} = \frac{2\alpha}{s_k} + \sum_{\substack{j=1\\j\neq k}}^J \frac 2{s_k-s_j}.
\end{gather*}
\end{lem}

\begin{lem}\label{wrlem} If $\operatorname{Wr}(f(x),g(x)) = W(x)$ then
\begin{gather*}
\frac{g''(x)}{g'(x)} - \frac{W'(x)}{W(x)} = \frac{g(x)\big(W(x) f''(x) - W'(x) f'(x)\big)}{f(x)g'(x)W(x)}.
\end{gather*}
\end{lem}
\begin{proof}
We have $W \operatorname{Wr}(f,g)' = W' \operatorname{Wr}(f,g)$. Hence
\begin{gather*}
W(x) f(x) g''(x) - W'(x) f(x) g'(x) =  W(x) f''(x) g(x) - W'(x) f'(x) g(x)
\end{gather*}
and hence the result.
\end{proof}

To deal with the case in which $\bm y^{(i)}$ fails to be generic, we  shall also need the following observation, which follows from~\eqref{wreY}.
\begin{lem}\label{nblem2}
For any $j\in I$  such that $\langle\alpha_j,\alpha^\vee_i\rangle <0$, if $\t$ is a root of both~$y_j(x)$ and~$y^{(i)}_i(x)$ then it is a root of $y^{(i)}_i(x)$ with multiplicity~$2$.
In particular, if~$\t$ is a root of both~$y_\bi(x)$ and $y^{(i)}_i(x)$ then it is a root of~$y^{(i)}_i(x)$ with multiplicity~$2$.
\end{lem}

Now we def\/ine
\begin{gather} y^{(\bi,i)}_\bi(x)
  := y_\bi(x) \int \xi^{{\langle\Lambda_0,\alpha^\vee_i\rangle}} T_\bi(\xi) \frac{\xi^{1+{\langle\Lambda_0,\alpha^\vee_i\rangle}}y^{(i)}_i(\xi) \prod\limits_{j\neq i,\bi} y_j(\xi)^{-\langle \alpha_j,\alpha^\vee_\bi\rangle}}{y_\bi(\xi)^2} d\xi\nonumber\\
\hphantom{y^{(\bi,i)}_\bi(x)}{}
 = y_\bi(x) \int \xi^{1+2{\langle\Lambda_0,\alpha^\vee_i\rangle}} T_\bi(\xi) \frac{y^{(i)}_i(\xi) \prod\limits_{j\neq i,\bi} y_j(\xi)^{-\langle \alpha_j,\alpha^\vee_\bi\rangle}}{y_\bi(\xi)^2} d\xi.\label{ybidef}
\end{gather}

\begin{prop} If $\bm y$ represents a critical point then $y^{(\bi,i)}_\bi(x)$ is a polynomial.
\end{prop}

\begin{proof} By our assumption~\eqref{s2L1} that $2{\langle\Lambda_0,\alpha^\vee_i\rangle}+1\in {\mathbb Z}_{\geq 0}$, the integrand is regular at $x=0$. Hence, by Lemma~\ref{nblem2}, it is a rational function with poles at most at those roots of~$y_\bi(x)$ that are not also roots of $y^{(i)}_i(x)$. Let $\t_p$ be any such root. The residue of the integrand at $\xi=\t_p$ is
\begin{gather*}
 \left.\frac{\partial}{\partial x} (x-\t)^2  x^{{\langle\Lambda_0,\alpha^\vee_i\rangle}} T_i(x) \frac{y^{(i)}_i(x) \prod\limits_{j\neq i,\bi} y^{(i)}_j(x)^{-\langle \alpha_j,\alpha^\vee_\bi\rangle}}{y_\bi(x)^2}\right|_{x=\t_p},
\end{gather*}
which must vanish, because according to Proposition~\ref{nb} the following vanishes:
\begin{gather*}
\left.\frac{\partial}{\partial x} \log (x-\t)^2  x^{{\langle\Lambda_0,\alpha^\vee_i\rangle}} T_i(x) \frac{y^{(i)}_i(x) \prod\limits_{j\neq i,\bi} y^{(i)}_j(x)^{-\langle \alpha_j,\alpha^\vee_\bi\rangle}}{y_\bi(x)^2}\right|_{x=\t_p}
\\
\qquad {} =
 \frac{1+ 2{\langle\Lambda_0,\alpha^\vee_i\rangle}}{\t_p}
+ \sum_{s=1}^N\sum_{k=0}^{M-1} \frac{\langle \sigma^k \Lambda_s,\alpha^\vee_{\bi}\rangle}{\t_p-\omega^k z_s}
- \sum_{r\colon r\neq p} \frac{\langle \alpha_{{\mathsf c}^{(i)}(r)},\alpha^\vee_{\bi} \rangle}{\t_p- \t_r^{(i)}} .
\end{gather*}
(Note ${\langle\Lambda_0,\alpha^\vee_i\rangle} = \langle \Lambda_0,\alpha^\vee_\bi\rangle$ since $\sigma \Lambda_0 = \Lambda_0$.)
\end{proof}

The polynomial $y^{(\bi,i)}_\bi(x)$ is def\/ined up a to the addition of a constant multiple of~$y_\bi(x)$, coming from the constant of integration in~\eqref{ybidef}.

We say generation in the $i$th direction from $\bm y$ is \emph{degree-increasing} if $\deg y^{(\bi,i)}_\bi(x) > \deg y_\bi(x)$. Generation in the $i$th direction is degree-increasing if and only if
\begin{gather}
\langle \Lambda_\infty+\rho, \alpha^\vee_i + \alpha^\vee_\bi \rangle>0.\label{pinc}
\end{gather}
Indeed, if \eqref{pinc} holds then
\begin{gather}
 \deg y^{(\bi,i)}_\bi(x) = \deg y_{\bi}(x) + \langle \Lambda_\infty+\rho, \alpha^\vee_i + \alpha^\vee_\bi \rangle > \deg y_{\bi}(x) \label{degybi}
\end{gather}
for all values of the constant of integration. If~\eqref{pinc} does not hold then $\deg y^{(\bi,i)}_\bi(x) \leq \deg y_\bi(x)$, with equality for all but one value of the constant of integration in~\eqref{ybidef}.

Let $y^{(\bi,i)}_{\bi}(x;0)$ be the unique solution to \eqref{ybidef} whose coef\/f\/icient of $x^{\deg y_{\bi}}$ is zero.  The degree of $y^{(\bi,i)}_{\bi}(x;0)$ is always given by
\begin{gather*}
 \deg y^{(\bi,i)}_{\bi}(x;0) =  \deg y_{\bi}(x) + \langle \Lambda_\infty+\rho, \alpha^\vee_i + \alpha^\vee_\bi \rangle,
 \end{gather*}
whether or not generation is degree-increasing. (Note that $\langle \Lambda_\infty+\rho, \alpha^\vee_i + \alpha^\vee_\bi \rangle$ is odd, by our assumption \eqref{s2L1}, and in particular not zero.)

Let then $\bm y^{(\bi,i)}(c)= (y^{(\bi,i)}_j(x;c))_{j\in I}$  be the tuple of polynomials whose $\bi$th component is
\begin{gather*}
y_\bi^{(\bi,i)}(x;c) := y^{(\bi,i)}_\bi(x;0) + c y_\bi(x)
\end{gather*}
and whose remaining components are the same as those of $\bm y^{(i)}$, i.e.,
\begin{gather*}
y^{(\bi,i)}_i(x;c) = y^{(i)}_i(x;c),\qquad\text{and}\qquad y^{(\bi,i)}_j(x) = y^{(i)}_j(x)=y_j(x)\qquad\text{for all}\quad j\in I{\setminus} \{i,\bi\}.
\end{gather*}
Let $(\bm \t^{(\bi.i)};\bm {\mathsf c}^{(\bi,i)})$ denote the pair represented by this tuple in the sense of Section~\ref{prsec}.

The following result says that whenever $\bm y^{(\bi,i)}(c)$ is generic, this new pair $(\bm\t^{(\bi,i)}(c),\bm{\mathsf c}^{(\bi,i)})$ obeys the same form of equations as did $(\bm\t^{(i)},\bm{\mathsf c}^{(i)})$.
\begin{prop}\label{nbprime}
If $\bm y$ represents a critical point then, for all $c\in {\mathbb C}$ such that $\bm y^{(\bi,i)}(c)$ is generic, we have
\begin{gather*}
 \frac{\big\langle{\mathsf s}_i\cdot\Lambda_0,\alpha^\vee_{{\mathsf c}^{(\bi,i)}(p)}\big\rangle}{\t^{(\bi,i)}_p(c)}
+ \sum_{s=1}^N\sum_{k=0}^{M-1} \frac{\big\langle \sigma^k \Lambda_s,\alpha^\vee_{{\mathsf c}^{(\bi,i)}(p)}\big\rangle}{\t^{(\bi,i)}_p(c)-\omega^k z_s}
- \sum_{r\colon r\neq p} \frac{\big\langle \alpha_{{\mathsf c}^{(\bi,i)}(r)},\alpha^\vee_{{\mathsf c}^{(\bi,i)}(p)} \big\rangle}{\t^{(\bi,i)}_p(c)- \t^{(\bi,i)}_r(c)} = 0
\end{gather*}
for each~$p$.
\end{prop}
\begin{proof}
The proof is analogous to that of Proposition~\ref{nb}.
\end{proof}

Finally, we def\/ine $y_i^{(i,\bi,i)}(x;c)$ by
\begin{gather*} y_i^{(i,\bi,i)}(x;c)
 =y^{(i)}_i(x)  x^{{\langle\Lambda_0,\alpha^\vee_i\rangle}+1}  \int_0^x \xi^{{\langle\Lambda_0,\alpha^\vee_i\rangle}} T_i(\xi) \frac{y^{(\bi,i)}_\bi(\xi;c)\prod\limits_{j\in I{\setminus}\{i,\bi\}} y_j(\xi)^{-\langle\alpha_j,\alpha^\vee_i\rangle}}{\left(\xi^{{\langle\Lambda_0,\alpha^\vee_i\rangle}+1} y^{(i)}_i(\xi)\right)^2}    d\xi\\
 \hphantom{y_i^{(i,\bi,i)}(x;c)}{}
 =y^{(i)}_i(x)  x^{{\langle\Lambda_0,\alpha^\vee_i\rangle}+1}  \int_0^x \xi^{-{\langle\Lambda_0,\alpha^\vee_i\rangle}-2} T_i(\xi) \frac{y^{(\bi,i)}_\bi(\xi;c)\prod\limits_{j\in I{\setminus}\{i,\bi\}} y_j(\xi)^{-\langle\alpha_j,\alpha^\vee_i\rangle}}{ y^{(i)}_i(\xi)^2}    d\xi.
\end{gather*}
Here the limits $\int_0^x$ mean that $y^{(i,\bi,i)}_i(x;c)$ is holomorphic at $x=0$. This condition def\/ines the integral uniquely.

\begin{prop}
For all $c\in {\mathbb C}$,
if $\bm y$ represents a critical point then $y^{(i,\bi,i)}_i(x;c)$ is a polynomial.
\end{prop}

\begin{proof} Pick any root $\t^{(\bi,i)}_p$ of $y_i^{(i)}(x)=y_i^{(\bi,i)}(x)$. The residue of the integrand at $\xi=\t^{(\bi,i)}$ is zero. Indeed, we have
\begin{gather*} \left.\frac{\partial}{\partial x} \log \big(x-\t^{(i)}_p\big)^2 x^{-{\langle\Lambda_0,\alpha^\vee_i\rangle}-2} T_i(x) \frac{y^{(\bi,i)}_\bi(x;c)\prod\limits_{j\in I{\setminus}\{i,\bi\}} y_j(x)^{-\langle\alpha_j,\alpha^\vee_i\rangle}}{ y^{(i)}_i(x)^2}\right|_{x=\t^{(\bi,i)}_p}
\\
\qquad{}=
 \frac{-{\langle\Lambda_0,\alpha^\vee_i\rangle}-2}{\t^{(\bi,i)}_p}
+ \sum_{s=1}^N\sum_{k=0}^{M-1} \frac{\langle \sigma^k \Lambda_s,\alpha^\vee_{\bi}\rangle}{\t^{(\bi,i)}_p-\omega^k z_s}
- \sum_{r\colon r\neq p} \frac{\langle \alpha_{{\mathsf c}^{(\bi,i)}(r)},\alpha^\vee_{i} \rangle}{\t^{(\bi,i)}_p- \t_r^{(\bi,i)}},
\end{gather*}
and this vanishes by Proposition \ref{nbprime}.
\end{proof}

Let $\bm y^{(i,\bi,i)}(c)\!=\! (y^{(i,\bi,i)}_j(x;c))_{j\in I}$ be the tuple of polynomials whose $i$th component is $y^{(i,\bi,i)}(x;c)$ as abo\-ve and whose remaining components are those of $\bm y^{(\bi,i)}(c)$, i.e.,
\begin{gather*}
 y^{(i,\bi,i)}_\bi(x;c) = y^{(\bi,i)}_\bi(x;c),\qquad\text{and}\qquad y^{(i,\bi,i)}_j(x) =
y_j(x)\qquad\text{for all}\quad j\in I{\setminus} \{i,\bi\}.
\end{gather*}
Let $(\bm \t^{(i,\bi.i)};\bm {\mathsf c}^{(i,\bi,i)})$ denote the pair represented by this tuple in the sense of Section~\ref{prsec}.

\begin{prop} \label{critprop} If $\bm y$ represents a critical point and $\bm y^{(i,\bi,i)}(c)$ is generic, then $\bm y^{(i,\bi,i)}(c)$ represents a critical point. That is, the pair $(\bm \t^{i,\bi,i}(c),\bm{\mathsf c}^{i,\bi,i})$ obeys the equations
\begin{gather*}
 \frac{\big\langle\Lambda_0,\alpha^\vee_{{\mathsf c}(p)}\big\rangle}{\t^{(i,\bi,i)}_p(c)} + \sum_{s=1}^N\sum_{k=0}^{M-1} \frac{\big\langle \sigma^k \Lambda_s,\alpha^\vee_{{\mathsf c}(p)}\big\rangle}{\t^{(i,\bi,i)}_p(c)-\omega^k z_s} - \sum_{r\colon r\neq p} \frac{\big\langle \alpha_{{\mathsf c}(r)},\alpha^\vee_{{\mathsf c}(p)} \big\rangle}{\t^{(i,\bi,i)}_p(c)- \t^{(i,\bi,i)}_r(c)} = 0
\end{gather*}
for each $p$.
\end{prop}

\begin{proof}
The proof is analogous to that of Proposition \ref{nb}.
\end{proof}

We say $\bm y^{(i,\bi,i)}(c)$ is obtained from $\bm y$ by \emph{generation in the $i$th direction}, and we call $\bm y^{(i,\bi,i)}(c)$ the \emph{immediate descendant of $\bm y$ in the $i$th direction}.
We have the following; cf.\ Lemma~\ref{delem}.

\begin{lem}\label{delemii}
Generation in the $i$th direction $($with $L_i=2)$ is degree-increasing if and only if $\langle\Lambda_\infty+\rho ,\alpha^\vee_i+\alpha^\vee_\bi\rangle\in {\mathbb Z}_{> 0}$.

If generation in the $i$th direction is degree-increasing, then the weight at inf\/inity associated with the critical point represented by $\bm y_i^{(i,\bi,i)}(c)$ is $({\mathsf s}_i{\mathsf s}_\bi{\mathsf s}_i) \cdot \Lambda_\infty$.
Otherwise it is $\Lambda_\infty$ for all $c\neq 0$ $($and $({\mathsf s}_i{\mathsf s}_\bi{\mathsf s}_i)\cdot \Lambda_\infty$ for $c=0)$.
\end{lem}

\begin{proof}
Recall that \eqref{degybi} holds if and only if  \eqref{pinc} holds.
Note also that
\begin{gather*}
 \deg y^{(i,\bi,i)}_i = \deg y_i + \langle \Lambda_\infty+\rho, \alpha^\vee_i+\alpha^\vee_\bi\rangle.
\end{gather*}
By direct calculation, one verif\/ies that
\begin{gather*}
({\mathsf s}_i{\mathsf s}_\bi {\mathsf s}_i) \cdot \Lambda_\infty  = \Lambda_\infty -  (\alpha_\bi+\alpha_i) \langle \Lambda_\infty + \rho, \alpha^\vee_i + \alpha^\vee_\bi \rangle,
\end{gather*}
so we have the result.
\end{proof}

\subsection[Cyclotomic generation: the $L_i=2$ case]{Cyclotomic generation: the $\boldsymbol{L_i=2}$ case}\label{sgen2}
We continue to suppose that $i\in I$ is such that $L_i=2$.

Suppose for the rest of this subsection that $\bm y$ represents a cyclotomic critical point.
Def\/ine  $\bm y^{(i,\sigma)}(c)$ to be the tuple of polynomials given by
\begin{gather}
y^{(i,\sigma)}_{\sigma^ki}(\omega^kx;c) := y^{(i,\bi,i)}_{i}(x;c), \nonumber\\
y^{(i,\sigma)}_{\sigma^k\bi}(\omega^kx;c) := y^{(i,\bi,i)}_{\bi}(x;c), \qquad k=0,1,\dots, M_i/2-1,\label{yis}
\end{gather}
and $y^{(i,\sigma)}_{j}(x;c) = y_j(x)$ for $j \in I {\setminus} \sigma^{\mathbb Z} i$.
\begin{thm}\label{s2thm}
For almost all $c\in {\mathbb C}$, the tuple $\bm y^{(i,\sigma)}(x;c)$ represents a cyclotomic critical point. The exceptional values of~$c$ form a f\/inite subset of~${\mathbb C}$.

The weight at inf\/inity of $\bm y^{(i,\sigma)}(x;c)$ is ${\mathsf s}^\sigma_i \cdot\Lambda_\infty$ if $\langle \Lambda_\infty+\rho,\alpha^{\sigma}_i+\alpha^\sigma_\bi \rangle \in {\mathbb Z}_{\geq 1}$. Otherwise it is~$\Lambda_\infty$ for all $c\neq 0$, and ${\mathsf s}^\sigma_i \cdot\Lambda_\infty$ for $c=0$.
\end{thm}

\begin{proof}
First let us show that $\bm y^{(i,\sigma)}(x;c)$ represents a critical point for all but f\/initely many $c\in {\mathbb C}$. As in the proof of Theorem~\ref{s1thm}, we f\/irst observe that~$\bm y^{(i,\sigma)}$ is indeed the result of generating in each of the directions $i, \sigma i, \dots, \sigma^{M_i/2-1}i$ (in any order). By Proposition~\ref{critprop} it is enough to check that $\bm y^{(i,\bi,i)}(c)$ is generic for all but f\/initely many $c\in {\mathbb C}$.
This follows from~\eqref{fee} and Lemma~\ref{VIL}, below.

The statements about the weight at inf\/inity follow from Lemma~\ref{delemii} and Section~\ref{foldsec}.

Finally we must check that $\bm y^{(i,\sigma)}(x;c)$ represents a cyclotomic point. Given Lemma~\ref{cycptlem} and the def\/inition~\eqref{yis}, it is enough to check that
\begin{gather}
y^{(i,\bi,i)}_\bi(- x;c) = (-1)^{\deg y^{(i,\bi,i)}_i} y^{(i,\bi,i)}_i(x;c).\label{fee}
\end{gather}
This is ef\/fectively a statement about the case of type $A_2$ and we are in the setting of Section~\ref{ARsec} below, with $R=2n$, $n=1$, $p=1$. The statement~\eqref{fee} follows from Theorem~\ref{flowthm} and Lemma~\ref{f2lem}.
\end{proof}

\begin{lem}\label{flem1} We have
\begin{gather*}
 y^{(\bi,i,\bi)}_j(-x;c) = (-1)^{\deg y^{(i,\bi,i)}_{\bar\jmath}} y^{(i,\bi,i)}_{\bar\jmath}(x;-c)
\end{gather*}
for all $j\in I$.
\end{lem}
\begin{proof}
Note f\/irst that from~\eqref{oT} we have
\begin{gather*}
 T_{\bar\jmath}(-x) = (-1)^{\big\langle \sum\limits_{k=0}^{M-1} \sum\limits_{s=1}^N \omega^k\Lambda_s, \alpha^\vee_j \big\rangle} T_j(x)
\end{gather*}
for all $j\in I$. It follows that
\begin{gather*}
 y^{(\bi)}_\bi(-x) = (-1)^{\deg y_i + \langle \Lambda_\infty - \Lambda_0, \alpha^\vee_i\rangle} y^{(i)}_i(x).
 \end{gather*}
Then, from the def\/inition of $y^{(\bi,i)}_\bi(x;c)$ and~\eqref{degybi} we have that
\begin{gather*}
 y^{(i,\bi)}_i(-x;c) = (-1)^{\deg y^{(i,\bi)}_\bi} y^{(\bi,i)}_\bi(x;\bar c)
\end{gather*}
if and only if $c$ and $\bar c$ are related by $c = (-1)^{2+\langle \Lambda_\infty, \alpha^\vee_\bi + \alpha^\vee_i \rangle} \bar c$. Since the~$\Lambda_s$, $s=1,\dots,N$, are integral, we have
\begin{gather*}
 (-1)^{\langle \Lambda_\infty, \alpha^\vee_i+\alpha^\vee_\bi \rangle}
   =(-1)^{\langle \Lambda_0, \alpha^\vee_i+\alpha^\vee_\bi \rangle}
   =(-1)^{2\langle \Lambda_0, \alpha^\vee_i \rangle}
   = -1,
\end{gather*}
using $\sigma\Lambda_0=\Lambda_0$ and the property~\eqref{s1L1}.
\end{proof}

\begin{lem}\label{VIL}
For all but finitely many $c\in {\mathbb C}$, $y^{(i,\bi,i)}_\bi(x;c)$ and $y^{(i,\bi,i)}_\bi(-x;c)$ have no root in common.
\end{lem}

\begin{proof}
Recall $y^{(i,\bi,i)}_\bi(x;c)=y^{(\bi,i)}_\bi(x;c)$. Consider the leading behaviour in small $c$. As $c\to 0$, $\deg y_\bi$ of the roots of $y^{(\bi,i)}_\bi(x;c)$ tend to the $\deg y_\bi$ roots of~$y_\bi(x)$. By the assumption that~$\bm y$ was generic and cyclotomic, none of these are roots of $y_\bi(-x)\simeq y_i(x)$.

Recall~\eqref{degybi} and the fact that  $\langle \Lambda_\infty+\rho, \alpha^\vee_i + \alpha^\vee_\bi \rangle$ is odd, by the assumption~\eqref{s2L1}.

If $\langle \Lambda_\infty+\rho, \alpha^\vee_i + \alpha^\vee_\bi \rangle < 0$, then these are all the roots of  $y^{(\bi,i)}_\bi(x;c)$.

If $\langle \Lambda_\infty+\rho, \alpha^\vee_i + \alpha^\vee_\bi \rangle \not< 0$ then the remaining
 $\langle \Lambda_\infty+\rho, \alpha^\vee_i + \alpha^\vee_\bi \rangle > 0$ roots of  $y^{(\bi,i)}_\bi(x;c)$ tend to the roots of the equation $cx^{\langle \Lambda_\infty+\rho, \alpha^\vee_i + \alpha^\vee_\bi \rangle} + 1 = 0$. This limiting set of roots multiplied by~$-1$ does not intersect itself. This implies the lemma.
\end{proof}

\subsection{Def\/inition of the cyclotomic population}\label{cycpopdefsec}

Suppose $\bm y\in {\mathbb P}({\mathbb C}[x])^{|I|}$ is a tuple of polynomials representing a cyclotomic critical point.

Recall the def\/inition of  $\bm y^{(i,\sigma)}(c)$, from Section~\ref{sgen1} when $L_i=1$ and from Section~\ref{sgen2} when $L_i=2$. We say  $\bm y^{(i,\sigma)}(c)$ is obtained from $\bm y$ by \emph{cyclotomic generation} in the direction~$i$.

Let us def\/ine the \emph{cyclotomic population} originated at $\bm y$ to be the Zariski closure of the set of all tuples of polynomials obtained from $\bm y$ by repeated cyclotomic generation, in all directions $i\in I$.

\section[The case of type $A_{R}$: vector spaces of quasi-polynomials]{The case of type $\boldsymbol{A_{R}}$: vector spaces of quasi-polynomials}\label{ARsec}
\subsection[Type $A$ data]{Type $\boldsymbol{A}$ data}\label{Asetup}

Throughout this section we specialise to $\g={\mathfrak{sl}}_{R+1}$. We shall treat in parallel the cases where $R = 2n-1$ and $R=2n$, $n\in {\mathbb Z}_{\geq 0}$.
We have the usual identif\/ication of $\h\cong \h^*$ with a subspace of $(R+1)$-dimensional Euclidean space, given by $\alpha_i = \alpha^\vee_i = \epsilon_{i+1} - \epsilon_i$, $i=1,\dots, R$,  where $(\epsilon_i)_{i=1}^{R+1}$ is the standard orthonormal basis.

Let $\sigma \colon \g \to \g$ be the unique non-trivial diagram automorphism, whose order is~2. The nodes~$I$ of the Dynkin diagram, and the action of $\sigma$ on these nodes, are as shown below:
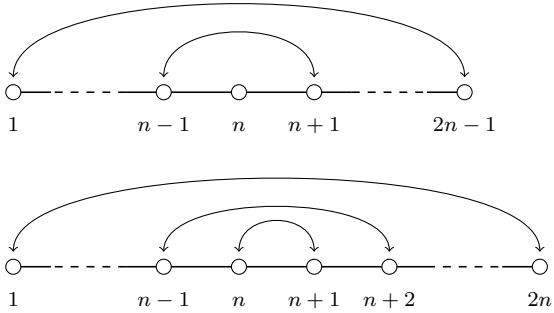
\begin{gather*}
\begin{tikzpicture}[baseline =-5,scale=1,font=\scriptsize]
\draw[semithick] (-.5,0) -- (2.5,0);
\draw[semithick,dashed] (-.5,0) -- (-1.5,0);
\draw[semithick,dashed] (2.5,0) -- (3.5,0);
\draw[semithick] (-2,0) -- (-1.5,0);
\draw[semithick] (3.5,0) -- (4,0);
\filldraw[fill=white] (0,0) circle (1mm) node [below=2mm] {$n-1$};
\filldraw[fill=white] (1,0) circle (1mm) node [below=2mm] {$\phantom1n\phantom1$};
\filldraw[fill=white] (2,0) circle (1mm)  node [below=2mm] {$n+1$};
\filldraw[fill=white] (4,0) circle (1mm)  node [below=2mm] {$2n-1$};
\filldraw[fill=white] (-2,0) circle (1mm)  node [below=2mm] {$1$};
\draw[<->,shorten >=2mm,shorten <=2mm] (0,0) .. controls (0,1) and (2,1) .. (2,0);
\draw[<->,shorten >=2mm,shorten <=2mm] (-2,0) .. controls (-2,1.5) and (4,1.5) .. (4,0);
\end{tikzpicture}
\\
\begin{tikzpicture}[baseline =-5,scale=1,font=\scriptsize]
\draw[semithick] (-.5,0) -- (3.5,0);
\draw[semithick,dashed] (-.5,0) -- (-1.5,0);
\draw[semithick,dashed] (3.5,0) -- (4.5,0);
\draw[semithick] (-2,0) -- (-1.5,0);
\draw[semithick] (4.5,0) -- (5,0);
\filldraw[fill=white] (0,0) circle (1mm) node [below=2mm] {$n-1$};
\filldraw[fill=white] (1,0) circle (1mm) node [below=2mm] {$\phantom1n\phantom1$};
\filldraw[fill=white] (2,0) circle (1mm)  node [below=2mm] {$n+1$};
\filldraw[fill=white] (3,0) circle (1mm)  node [below=2mm] {$n+2$};
\filldraw[fill=white] (5,0) circle (1mm)  node [below=2mm] {$\,\,2n\,\,$};
\filldraw[fill=white] (-2,0) circle (1mm)  node [below=2mm] {$1$};
\draw[<->,shorten >=2mm,shorten <=2mm] (1,0) .. controls (1,.75) and (2,.75) .. (2,0);
\draw[<->,shorten >=2mm,shorten <=2mm] (0,0) .. controls (0,1) and (3,1) .. (3,0);
\draw[<->,shorten >=2mm,shorten <=2mm] (-2,0) .. controls (-2,1.5) and (5,1.5) .. (5,0);
\end{tikzpicture}
\end{gather*}
When $R=2n-1$, then $L_i=1$ for all $i\in I$, and $M_i = \begin{cases} 1, & i= n, \\ 2, & i\neq n.\end{cases}$
When $R=2n$ then $L_i=\begin{cases} 2, & i = n, n+1 \\ 1, & \text{otherwise}
\end{cases}$ and $M_i =2$ for all $i\in I$.

Let $(z_i)_{i=1}^N$ be nonzero points $z_i\in {\mathbb C}^\times$ such that $z_i \pm z_j \neq 0$ whenever $i\neq j$.
Let $\Lambda_1,\dots,\Lambda_N$ be dominant integral weights.

We suppose the weight at the origin, $\Lambda_0\in \h^*$, obeys $\sigma\Lambda_0 = \Lambda_0$ (as always). That is,
\begin{gather*}
\langle\Lambda_0,\alpha^\vee_i \rangle = \langle\Lambda_0,\alpha^\vee_{R+1-i}\rangle,\qquad i=1,\dots,R.
\end{gather*}
In addition, we pick and f\/ix an integer $p\in \{0,1,\dots,n\}$, and suppose that
\begin{subequations}\label{l0a}
\begin{gather}
 \langle \Lambda_0, \alpha^\vee_i\rangle
\in 2{\mathbb Z}_{\geq 0}/M_i \qquad
\text{if $i\notin\{p,R+1-p\}$}
\end{gather}
and
\begin{gather}
 \langle \Lambda_0, \alpha^\vee_p \rangle
\in
\begin{cases}
\half (2{\mathbb Z}_{\geq 0}-1)  = \{-\half,\half,\frac 3 2,\dots\} & \text{if $p\leq R/2$},\\
2{\mathbb Z}_{\geq 0}+1  = \{1,3,\dots\} & \text{if $p=n$ and $R=2n-1$.}\end{cases}
\end{gather}
\end{subequations}
Note the following particular cases:
\begin{itemize}\itemsep=0pt
\item If $R=2n$ is even and $p=0$ then \eqref{l0a} just says that $\Lambda_0$ is dominant integral.
\item If $R=2n-1$ is odd and $p=0$ then $\Lambda_0$ is dominant integral and $\langle\Lambda_0,\alpha^\vee_n\rangle$ is even.
\item If $R=2n-1$ is odd and $p=n$ then $\Lambda_0$ is dominant integral and $\langle\Lambda_0,\alpha^\vee_n\rangle$ is odd.
\end{itemize}

In the case $p=n$ (and any $R$) our choice of $\Lambda_0$ obeys the assumptions set out in Section~\ref{l0sec}.

\subsection{Vector spaces of quasi-polynomials}\label{slrevsec}
Let
\begin{gather*}
 {\tilde T}_i(x)
= x^{\langle\Lambda_0,\alpha^\vee_i\rangle}
\prod_{s=1}^{N}(x- z_s)^{\langle \Lambda_s,\alpha^\vee_i\rangle}(x+ z_s)^{\langle \Lambda_{R+1-s},\alpha^\vee_i\rangle},\qquad i \in I.
\end{gather*}
Thus ${\tilde T}_i(x) = x^{\langle\Lambda_0,\alpha^\vee_i\rangle} T_i(x)$ with $T_i(x)$ as in~\eqref{Tdef}.

In view of \eqref{l0a}, ${\tilde T}_i(x)\in {\mathbb C}[x]$ for all $i\notin \{p,R+1-p\}$. If $0<p<R+1-p<R$ then  ${\tilde T}_p(x)$ and ${\tilde T}_{R+1-p}(x)$ belong to $x^{-\half} {\mathbb C}[x]$. If $p=R+1-p$ then ${\tilde T}_p(x)\in {\mathbb C}[x]$.

We def\/ine the \emph{degree}, $\deg p$, of a Laurent polynomial $p(x)\in {\mathbb C}[x^{\pm\half}]$ to be the leading power of~$x$ (for large $x$) that appears in~$p(x)$ with non-zero coef\/f\/icient.

We will call any polynomial in~$x^\half$ a \emph{quasi-polynomial}.

A vector space $V\subset {\mathbb C}\big[x^\half\big]$ of quasi-polynomials is \emph{decomposable} if
\begin{gather*}
 V = V\cap {\mathbb C}[x] \oplus V \cap x^\half {\mathbb C}[x].
 \end{gather*}
A tuple of quasi-polynomials is \emph{decomposable} if each element lies in either~${\mathbb C}[x]$ or~$x^\half{\mathbb C}[x]$. In particular, a \emph{decomposable basis} of a decomposable vector space $V\subset {\mathbb C}\big[x^\half\big]$ is one in which each basis vector lies in either~${\mathbb C}[x]$ or $x^\half{\mathbb C}[x]$.

Def\/ine the \emph{divided Wronksian determinant} of quasi-polynomials $u_1,\dots,u_k \in {\mathbb C}\big[x^\half\big]$ by
\begin{gather*} \operatorname{Wr}^\dag(u_1,\dots,u_k) := \frac{\operatorname{Wr}(u_1,\dots,u_k)}{{\tilde T}_1^{k-1} {\tilde T}_2^{k-2} \cdots {\tilde T}_{k-1}}
,\qquad \operatorname{Wr}(u_1,\dots,u_k):= \det\left( \frac{d^{j-1} u_i}{dx^{j-1}} \right)_{i,j=1}^k,
\end{gather*}
for $k=1,\dots,R+1$.

Def\/ine
\begin{gather}
\Lambda := \Lambda_0 + \sum_{s=1}^N (\Lambda_s+ \sigma\Lambda_s),\label{lambdadef}
\end{gather}
and suppose $\tilde\Lambda_\infty\in \h^*$ is a dominant weight such that $\Lambda-\tilde\Lambda_\infty= \sum\limits_{i\in I} k_i \alpha_i$ for some $k_i\in {\mathbb Z}_{\geq 0}$.
Such a weight def\/ines numbers $d_1,\dots,d_{R+1} \in {\mathbb Z}/2$, $0\leq d_1 < \dots < d_{R+1}$,
by
\begin{gather}
 d_1 := \langle\Lambda- \tilde\Lambda_\infty , \epsilon_1 \rangle ,\qquad d_k := \langle\Lambda- ({\mathsf s}_1\cdots{\mathsf s}_{k-1}) \cdot \tilde\Lambda_\infty , \epsilon_1 \rangle,\qquad k=2,\dots, R+1. \label{ddef}
\end{gather}
\begin{lem}\label{deglem} We have
\begin{gather}
 d_k = d_1 +  \langle \tilde\Lambda_\infty + \rho, \alpha^\vee_1 + \dots +\alpha^\vee_{k-1} \rangle, \qquad k=2,\dots,R+1. \label{du}
\end{gather}
Hence, for all $p>0$, $d_1,\dots, d_p$ and $d_{R+2-p},\dots,d_{R+1}$ are integers while $d_{p+1},\dots,d_{R+1-p}$ are half odd integers, i.e., have the form $m+\half$ for $m\in {\mathbb Z}$. If $p=0$ then $d_1,\dots, d_{R+1}$ are all integers.
\end{lem}

\begin{proof} We have  $d_k - d_1 = \langle \tilde \Lambda_\infty+\rho - ({\mathsf s}_1\cdots{\mathsf s}_{k-1})(\tilde\Lambda_\infty+\rho), \epsilon_1\rangle = \langle \tilde\Lambda_\infty  + \rho,\epsilon_1 - ({\mathsf s}_{k-1}\cdots {\mathsf s}_1) \epsilon_1\rangle = \langle \tilde\Lambda_\infty + \rho, \epsilon_1- \epsilon_k \rangle$ and hence~\eqref{du}.
\end{proof}

\begin{defn}\label{Kfdef}
We say a vector space of quasi-polynomials ${\mathcal K}\!\subset \! {\mathbb C}\big[x^\half\big]$ has \emph{frame} $\tilde T_1,\dots ,\tilde T_R;\tilde \Lambda_\infty$ if the following conditions hold:
\begin{enumerate}[(i)]\itemsep=0pt
\item\label{expcon}
There is a basis $(u_k(x))_{k=1}^{R+1}$ of ${\mathcal K}$ such that $\deg u_k = d_k$ for each $k=1,\dots, R+1$.
\item\label{ramcon}
For any $z\in {\mathbb C}{\setminus}\{0\}$  and $v_1,\dots, v_k\in {\mathcal K}$, $k=1,\dots,R+1$,  the divided Wronskian $\operatorname{Wr}^\dag(v_1$, $\dots,v_k)$ is regular at $z$, and moreover, $\operatorname{Wr}^\dag(v_1,\dots,v_k)$ is nonzero at $z$ for suitable $v_1,\dots,v_k$.
\item\label{zerocon}
For all $v_1,\dots, v_k\in {\mathcal K}$, $k=1,\dots,R+1$, the divided Wronskian $\operatorname{Wr}^\dag(v_1,\dots,v_k)$ has at $x=0$ an expansion of the form $\sum\limits_{m\in {\mathbb Z}_{\geq 0}/2} a_m x^m$ and moreover this expansion has nonzero~$a_0$ for suitable  $v_1,\dots,v_k$.
\end{enumerate}
\end{defn}

\emph{In the remainder of this section, ${\mathcal K}$ will denote a decomposable vector space of quasi-po\-ly\-no\-mials with frame $\tilde T_1,\dots ,\tilde T_R;\tilde \Lambda_\infty$.}

Conditions (\ref{ramcon}) and~(\ref{zerocon}) specify the \emph{ramification conditions} of ${\mathcal K}$ at every point $z\in {\mathbb C}$. Condition~(\ref{expcon}) specif\/ies the ramif\/ication conditions at $\infty$. See \cite[Section~5.5]{MV1}. The degrees $0\leq d_1 < d_2 < \dots < d_{R+1}$ will be called the \emph{exponents of ${\mathcal K}$ at infinity}.

Note that conditions~(\ref{ramcon}) and~(\ref{zerocon}) together imply in particular that~${\mathcal K}$ has no~\emph{base points}. That is, there is no $z\in {\mathbb C}$ such that $u(z) = 0$ for all $u\in {\mathcal K}$. They also imply the following important lemma.

\begin{lem}\label{dWlem}
For all $v_1,\dots, v_k\in {\mathcal K}$, $k=1,\dots,R+1$, the divided Wronskian  $\operatorname{Wr}^\dag(v_1,\dots,v_k)$ is a~quasi-polynomial.
\end{lem}

Since ${\mathcal K}$ is decomposable it follows from condition~(\ref{expcon}) that ${\mathcal K}$ admits a decomposable basis $(u_k)_{k=1}^{R+1}$ such that $\deg u_k=d_k$ for each $k$.
We call any such basis a \emph{special} basis.

\begin{lem}\label{sblem} Any two special bases $(u_k)_{k=1}^{R+1}$ and $(u'_k)_{k=1}^{R+1}$ are related by a triangular change of basis, $u'_k= \sum\limits_{j\leq k} a_{kj} u_j$, such that $a_{kj}=0$ whenever $d_k-d_j\notin {\mathbb Z}$.
\end{lem}
\begin{lem}\label{Wlem} Let $m\in {\mathbb Z}_{\geq 1}$. Let $n_1,\dots, n_m$ be  non-negative integers. Then
\begin{gather*}
 \operatorname{Wr}(x^{n_1},\dots,x^{n_m}) = x^{\sum\limits_{i=1}^m n_i- \frac{m(m-1)}2} \prod _{1\leq j<i\leq m} (n_i-n_j).
\end{gather*}
\end{lem}

\begin{lem}\label{uwrlem} Let $(u_i(x))_{i=1}^{R+1}$ be a special basis of ${\mathcal K}$. Then
\begin{gather*}
 \operatorname{Wr}^\dag(u_1,\dots,u_{R+1}) = \prod_{1\leq j<i\leq R+1} (d_i - d_j).
 \end{gather*}
\end{lem}

\begin{proof} By Lemma~\ref{dWlem}, $\operatorname{Wr}^\dag(u_1,\dots,u_{R+1})\in {\mathbb C}\big[x^\half\big]$. We must show that it has degree zero and compute the constant term.
From the condition that $\Lambda-\tilde\Lambda_\infty\in {\mathbb Z}_{\geq 0}[\alpha_i]_{i\in I}$ it follows that
$0=\langle \Lambda-\tilde\Lambda_\infty, -(R+1) \epsilon_1 +  R\alpha^\vee_1 +(R-1) \alpha^\vee_2 + \dots + 2 \alpha^\vee_{R-1} + \alpha^\vee_R \rangle$
and therefore
\begin{gather} (R+1) d_1 = \langle \Lambda- \tilde \Lambda_\infty, R\alpha^\vee_1 +(R-1) \alpha^\vee_2 + \dots + 2 \alpha^\vee_{R-1} + \alpha^\vee_R \rangle.\label{d1e}
\end{gather}
Then \eqref{du} implies
\begin{gather}
\sum_{i=1}^{R+1} d_i - \frac{(R+1)R}2 = \langle \Lambda, R\alpha^\vee_1 +(R-1) \alpha^\vee_2 + \dots + 2 \alpha^\vee_{R-1} + \alpha^\vee_R \rangle.\label{dde}
\end{gather}
The result follows by Lemma~\ref{Wlem}.
\end{proof}

\begin{cor}\label{constlem} $\operatorname{Wr}^\dag(v_1,\dots,v_{R+1})$ is a constant $($independent of~$x)$ for all $v_1,\dots, v_{R+1}\in {\mathcal K}$.
\end{cor}

Let $(u_k)_{k=1}^{R+1}$ be a special basis of ${\mathcal K}$. Introduce the subspaces
\begin{gather*} {\mathcal K}_{\textup{Sp}}  := \operatorname{span}_{\mathbb C}(u_1,\dots,u_p) \oplus \operatorname{span}_{\mathbb C}(u_{R+2-p},\dots,u_{R+1}),\\
 {\mathcal K}_{\textup{O}}  := \operatorname{span}_{\mathbb C}(u_{p+1},\dots,u_{R+1-p}),
 \end{gather*}
so that
\begin{gather*}
 {\mathcal K} = {\mathcal K}_{\textup{Sp}} \oplus {\mathcal K}_{\textup{O}}.
\end{gather*}
By Lemma~\ref{sblem}, these def\/initions of do not depend on the choice of special basis $(u_k)_{k=1}^{R+1}$.
By Lemma~\ref{deglem} we have that, whenever $p>0$, then
\begin{gather}
 {\mathcal K}_{\textup{Sp}} = {\mathcal K} \cap {\mathbb C}[x],\qquad {\mathcal K}_{\textup{O}} = {\mathcal K}\cap x^\half {\mathbb C}[x]. \label{kpmpoly}
\end{gather}
Exceptionally, when $p=0$, we have ${\mathcal K}_{\textup{Sp}} = \{0\}$, ${\mathcal K}_{\textup{O}} = {\mathcal K}\subset {\mathbb C}[x]$.

Given a decomposable subspace $V$, we write $\operatorname{sdim} V$ for the pair of numbers
\begin{gather*}
 \operatorname{sdim} V:= ( \dim V\cap {\mathcal K}_{\textup{Sp}}| \dim V\cap {\mathcal K}_{\textup{O}} ).
\end{gather*}

\subsection[ Flags in ${\mathcal K}$]{Flags in $\boldsymbol{{\mathcal K}}$}\label{flagsec}
Let $\FL({\mathcal K})$ denote the
space of full (i.e., $R+1$-step) f\/lags in~${\mathcal K}$.

We say an $r$-step f\/lag ${\mathcal F} = \{0= F_0 \subset F_1\subset F_2 \subset \dots \subset F_{r} = {\mathcal K}\}$ in ${\mathcal K}$ is \emph{decomposable} if each~$F_k$ is decomposable.

The space of decomposable full f\/lags in ${\mathcal K}$ has $\binom {R+1}{2p}$ connected components.
These connected components are labeled by $2p$-element subsets $Q\subset \{1,\dots,R+1\}$.  Def\/ine  $\FL_Q({\mathcal K})$ to be the subset consisting of the f\/lags ${\mathcal F}= \{0=F_0\subset F_1\subset \dots\subset F_{R+1} = {\mathcal K}\}$ such that for each $k$,
\begin{gather*}
 \operatorname{sdim} F_k - \operatorname{sdim} F_{k-1} =\begin{cases} (1|0) & \text{if}\quad  k\in Q,\\
      (0|1) & \text{if} \quad k\notin Q.\end{cases}
\end{gather*}
We call elements of $\FL_Q({\mathcal K})$ \emph{flags of type $Q$}.

For each $Q$ the variety $\FL_Q({\mathcal K})$ is isomorphic to the direct product of full f\/lag spaces $\FL({\mathcal K}_{\textup{Sp}})\times \FL({\mathcal K}_{\textup{O}})$. The isomorphism
\begin{gather}
\eta_Q \colon \ \FL({\mathcal K}_{\textup{Sp}})\times \FL({\mathcal K}_{\textup{O}}) \to \FL_Q({\mathcal K})\label{etaqdef}
\end{gather}
sends a pair of f\/lags $F_{1,+}\subset \dots\subset F_{2p,+}$, $F_{1,-}\subset \dots\subset F_{R+1-2p,-}$ to the f\/lag $F_{1}\subset \dots\subset F_{R}$, where
$F_k = F_{k_1,+}\oplus  F_{k_2,-}$,  $k_1=  |Q\cap \{1,\dots,k\} |$,  $k_2=k-k_1$.

Call a $2p$-element subset $Q\subset \{1,\dots,R+1\}$ \emph{symmetric} if $Q$ is
invariant with respect  to the involution $k\mapsto R+2-k$.
In particular, the following subset $S$ is symmetric
\begin{gather}
 S:=\{1,\dots, p, R+2-p,\dots,R+1\}. \label{sh}
\end{gather}
If $(u_k)_{k=1}^{R+1}$ is a special basis of ${\mathcal K}$ then the full f\/lag ${\mathcal F} = \{0= F_0 \subset F_1\subset F_2 \subset \dots \subset F_{R+1} = {\mathcal K}\}$ def\/ined by
\begin{gather}
 F_k = \operatorname{span}_{\mathbb C}(u_1,\dots,u_k), \qquad k=1,\dots,R+1,\label{smallcell}
\end{gather}
belongs to $\FL_S({\mathcal K})$. By Lemma~\ref{sblem} this f\/lag is independent of the choice of special basis.

To any full f\/lag ${\mathcal F} = \{  0 \subset F_1 \subset F_2 \subset \dots \subset F_{R+1} = {\mathcal K}\}$ in $\FL({\mathcal K})$ one can associate a tuple $\bm y^{\mathcal F} = (y_i(x))_{i=1}^{R}\in {\mathbb P}\big({\mathbb C}\big[x^\half\big]\big)^{R}$. Namely, let $(u^{\mathcal F}_k(x))_{k=1}^{R+1}$ be any basis of~${\mathcal K}$ such that
\begin{gather*}
F_k = \operatorname{span}_{\mathbb C}\big(u^{\mathcal F}_1,\dots,u^{\mathcal F}_k\big), \qquad k=1,\dots,R+1.
\end{gather*}
(we say such a basis is \emph{adjusted to} ${\mathcal F}$)
and then let
\begin{gather} y^{\mathcal F}_k := \operatorname{Wr}^\dag\big(u^{\mathcal F}_1,\dots,u^{\mathcal F}_k\big), \qquad k=1,\dots,R.
\label{yfdef}
\end{gather}
By Lemma \ref{dWlem}, these are quasi-polynomials.

We have the shifted action of the Weyl group of type $A_R$ on weights as in Section~\ref{KMa}. The weight at inf\/inity $\Lambda_\infty(\bm y^{\mathcal F})$, as in~\eqref{l8y}, belongs to the shifted Weyl orbit of $\tilde\Lambda_\infty$ \cite[Section~3.6]{MV1}. It is equal to $\tilde\Lambda_\infty$ if and only if~${\mathcal F}$ is the f\/lag given in~\eqref{smallcell}.

The map ${\mathcal F} \mapsto \bm y^{\mathcal F}$ def\/ines a morphism of varieties,
\begin{gather*}
 \beta\colon \ \FL({\mathcal K}) \to {\mathbb P}\big({\mathbb C}\big[x^{1/2}\big]\big)^{R}.
 \end{gather*}
This morphism $\beta$ def\/ines an isomorphism of $\FL({\mathcal K})$ onto its image, as in Lemmas 5.14--5.16 of~\cite{MV1}.

\begin{lem}\label{Simlem}
The image $\beta(\FL_S({\mathcal K}))$ of the variety of flags of type $S$ lies in ${\mathbb P}({\mathbb C}[x])^{R}$, i.e., consists of tuples of polynomials.
\end{lem}
\begin{proof}
In the exceptional case $p=0$ no fractional powers are present at all and the result is clear. Suppose $p>0$.
Let ${\mathcal F}\in \FL_S({\mathcal K})$ and let $(u_k^{\mathcal F})_{k=1}^{R+1}$ be a basis of ${\mathcal K}$ adjusted to ${\mathcal F}$. By inspection one sees that because ${\mathcal F}\in \FL_S({\mathcal K})$, $\operatorname{Wr}(u^{\mathcal F}_1,\dots,u^{\mathcal F}_k)$ lies in ${\mathbb C}[x]$ (resp.\ $x^\half{\mathbb C}[x]$) for precisely those $k$ such that the product ${\tilde T}_1^{k-1} \cdots {\tilde T}_{k-1}$ lies in ${\mathbb C}[x^{\pm 1}]$ (resp.\ $x^{\half}{\mathbb C}[x^{\pm 1}]$). For each~$k$, Lemma~\ref{dWlem} guarantees that $y^{\mathcal F}_k\in {\mathbb C}\big[x^\half\big]$. Hence in fact $y^{\mathcal F}_k\in {\mathbb C}[x]$.
\end{proof}

\begin{lem}\label{QSlem} The tuple $\beta({\mathcal F})=\bm y^{\mathcal F}$ is decomposable if and only if ${\mathcal F}$ is a decomposable flag.
If~${\mathcal F}$ is a decomposable flag of type $Q$ then
\begin{gather*} y_k^{\mathcal F} \in \begin{cases}\phantom{x^\half} {\mathbb C}[x] & \text{if} \quad \left| S\triangle Q \cap \{1,\dots,k\} \right| \in 2{\mathbb Z},\\
 x^\half {\mathbb C}[x] & \text{if} \quad \left| S\triangle Q \cap \{1,\dots,k\} \right| \in 2{\mathbb Z}+1,
\end{cases}
\end{gather*}
where $S\triangle Q := (S{\setminus} Q)\cup (Q {\setminus} S)$ denotes the symmetric difference of~$S$ and~$Q$.
In particular~$\bm y^{\mathcal F}$ is a tuple of polynomials if and only if $Q=S$.
\end{lem}

\subsection{Fundamental dif\/ferential operator and the recovery theorem}\label{Dsec}

To any given a tuple $\bm y= (y_i(x))_{i=1}^{R}\in {\mathbb P}\big({\mathbb C}\big[x^\half\big]\big)^{R}$ of quasi-polynomials, we may associate a~dif\/ferential operator ${\mathcal D}(\bm y)$, def\/ined by
\begin{gather*} {\mathcal D}(\bm y)  := \left(\partial - \log'\frac{{\tilde T}_1{\tilde T}_2\cdots {\tilde T}_{R}}{y_{R}}\right)
 \left(\partial - \log'\frac{y_{R}{\tilde T}_1{\tilde T}_2\cdots {\tilde T}_{R-1}}{y_{R-1}}\right)
\cdots
 \\
\hphantom{{\mathcal D}(\bm y) =}{}\times
\left(\partial - \log'\frac{y_2 {\tilde T}_1}{y_1}\right) \big(\partial - \log' y_1\big)=  \overset{\longrightarrow}{\prod_{i=0}^{R}}
 \left(\partial - \log'\frac{y_{R+1-i}\prod\limits_{j=1}^{R-i} {\tilde T}_j}{y_{R-i}}\right),
  \end{gather*}
with the understanding that $y_0=y_{R+1} = 1$. Here $\partial:= \partial/\partial x$ and $\log'f := f'(x)/f(x)$.

\begin{thm}[\protect{\cite[Lemma~5.6]{MV1}}] Let $\bm y\in \beta(\FL({\mathcal K}))$. Then ${\mathcal K}= \ker {\mathcal D}$.
\end{thm}

\subsection[The dual space ${\mathcal K}^\dag$]{The dual space $\boldsymbol{{\mathcal K}^\dag}$}\label{Kdagsec}

Let ${\mathcal K}^\dag$ be the complex vector space
\begin{gather*}
 {\mathcal K}^\dag := \operatorname{span}_{\mathbb C}\big\{ {\operatorname{Wr}}^\dag(v_1,\dots,v_{R}) \colon v_1,\dots,v_{R}\in {\mathcal K} \big\} \subset {\mathbb C}\big[x^\half\big].
\end{gather*}
The space ${\mathcal K}^\dag$ is a space of quasi-polynomials by Lemma \ref{dWlem}.
The spaces ${\mathcal K}^\dag$ and ${\mathcal K}$ are dual with respect to the pairing
\begin{gather*}
 ( \cdot, \cdot) \colon \  {\mathcal K}^\dag \times {\mathcal K} \to {\mathbb C}
\end{gather*}
def\/ined by
\begin{gather*}
 \big(v_1,\operatorname{Wr}^\dag(v_2,\dots,v_{R+1}) \big) := \operatorname{Wr}^\dag(v_1,v_2,\dots,v_{R+1}).
\end{gather*}
Given any basis $(u_i(x))_{i=1}^{R+1}$ of ${\mathcal K}$ there is a basis $(W_i(x))_{i=1}^{R+1}$ of ${\mathcal K}^\dag$ def\/ined by
\begin{gather*}
 W_i := \operatorname{Wr}^\dag(u_1,\dots,\widehat u_i,\dots,u_{R+1}) \in {\mathcal K}^\dag,\qquad i=1,\dots,R+1, \label{Wdef}
\end{gather*}
where $\widehat u_i$ denotes omission. We have
\begin{gather*}
 (u_i, W_j) = 0 \qquad\text{if}\quad i\neq j, \qquad (u_i,W_i) \neq 0.\label{Wd}
\end{gather*}

Let  $d^\dag_1> \dots >d^\dag_{R+1}$ be the numbers given by
\begin{gather*}
 d^\dag_{R+1} :=  - \langle \Lambda - \tilde\Lambda_\infty,\epsilon_{R+1} \rangle,\qquad  d^\dag_k := - \langle \Lambda- ({\mathsf s}_R \cdots {\mathsf s}_{k})\cdot \tilde\Lambda_\infty, \epsilon_{R+1} \rangle, \qquad k=1,\dots,R,
\end{gather*}
cf.~\eqref{ddef}. We have
\begin{gather}
 d^\dag_{k} = d^\dag_{R+1} + \langle \tilde\Lambda_\infty + \rho, \alpha^\vee_k + \dots + \alpha^\vee_R \rangle, \qquad k=1,\dots, R,\label{du2}
\end{gather}
by an argument as for Lemma~\ref{deglem}.

\begin{lem}\label{Wdlem} Let $(u_i(x))_{i=1}^{R+1}$ be a special basis of ${\mathcal K}$. Then $\deg W_k = d^\dag_k$, $k=1,\dots,R+1$, and the basis $(W_k)_{k=1}^{R+1}$ is decomposable.
\end{lem}

\begin{proof}
From $\Lambda-\tilde\Lambda_\infty\in {\mathbb Z}_{\geq 0}[\alpha_i]_{i\in I}$ we have
$0=\langle \Lambda-\tilde\Lambda_\infty, (R+1) \epsilon_{R+1} +  \alpha^\vee_1 + 2 \alpha^\vee_2 + \dots + (R-1) \alpha^\vee_{R-1} + R\alpha^\vee_R \rangle$, and hence
\begin{gather}
 (R+1)d^\dag_{R+1} = \langle \Lambda - \tilde \Lambda_\infty , \alpha^\vee_1 + 2 \alpha^\vee_2 + \dots + R \alpha^\vee_R \rangle.\label{dRe}
\end{gather}
Now
\begin{gather*} \deg W_{R+1}  = \deg \operatorname{Wr}^\dag(u_1,\dots, u_{R})
 = \sum_{i=1}^{R} d_i - \frac{R(R-1)}2 - \langle \Lambda, (R-1) \alpha^\vee_1 + \dots + \alpha^\vee_{R-1} \rangle \\
\hphantom{\deg W_{R+1}}{}
  = R d_1 + \langle \tilde \Lambda_\infty + \rho, (R-1)\alpha^\vee_1  + \dots + \alpha^\vee_{R-1} \rangle- \frac{R(R-1)}2 \\
\hphantom{\deg W_{R+1}=}{}
  - \langle \Lambda, (R-1) \alpha^\vee_1 + \dots + \alpha^\vee_{R-1} \rangle\\
\hphantom{\deg W_{R+1}}{}
= R d_1 + \langle \tilde \Lambda_\infty - \Lambda,  (R-1) \alpha^\vee_1 + \dots + \alpha^\vee_{R-1} \rangle,
\end{gather*}
where we used \eqref{du}. Hence, using \eqref{d1e}, we have
\begin{gather*}
\begin{split}
&(R+1) \deg W_{R+1}  = R \langle \Lambda- \tilde \Lambda_\infty, R\alpha^\vee_1 +(R-1) \alpha^\vee_2 + \dots + 2 \alpha^\vee_{R-1}
+ \alpha^\vee_R \rangle\\
& \hphantom{(R+1) \deg W_{R+1}  =}{}
  - (R+1) \langle \Lambda - \tilde \Lambda_\infty ,  (R-1) \alpha^\vee_1 + \dots + \alpha^\vee_{R-1} \rangle\\
&\hphantom{(R+1) \deg W_{R+1} }{}
  =   \langle \Lambda - \tilde \Lambda_\infty , \alpha^\vee_1 + 2\alpha^\vee_2 + \dots + R \alpha^\vee_R \rangle
\end{split}
\end{gather*}
since $R(R+1-k) - (R+1)(R-k) = k$.
Comparing this with~\eqref{dRe} we see that $d^\dag_{R+1} = \deg W_{R+1}$.
Then for the remaining $W_k$, we note that $\deg W_{R+1} - \deg W_k = d_k - d_{R+1}$ for $k=1,\dots,R$.
And  by~\eqref{du} and~\eqref{du2},
\begin{gather*}
 d_k - d_{R+1} = - \langle \tilde\Lambda_\infty + \rho, \alpha^\vee_k + \dots + \alpha^\vee_{R} \rangle = d^\dag_{R+1} - d^\dag_{k}.
 \end{gather*}
Thus $d^\dag_{k} = \deg W_{k}$ for $k=1,\dots, R+1$. Finally, since the basis $(u_k)_{k=1}^{R+1}$ is decomposable and each  ${\tilde T}_k$ lies in either ${\mathbb C}[x^{\pm 1}]$ or $x^\half {\mathbb C}[x^{\pm 1}]$, it follows that $(W_k)_{k=1}^{R+1}$ is decomposable.
\end{proof}

\subsection{Cyclotomic points and cyclotomic self-duality} \label{cycdualsec}

Let us f\/ix $(-1)^{m} := e^{m\pi i}$ for $m\in {\mathbb Z}/2$.
Then given a monomial  $q(x)= x^{m}$, $m\in {\mathbb Z}/2$, we def\/ine $q(-x) := (-1)^m x^m$. We extend the transformation~$q(x) \mapsto q(-x)$ to Laurent polynomials in~$x^\half$ by linearity.

We say that ${\mathcal K}$ is \emph{cyclotomically self-dual} if
\begin{gather*}
u(x)\in {\mathcal K} \ \Leftrightarrow \ u(-x)\in {\mathcal K}^\dag.
\end{gather*}

\begin{lem}\label{cd2lem}
If ${\mathcal K}$ is cyclotomically self-dual then
\begin{gather*}
 d_k + d_{R+2-k} =R+\langle \Lambda, \alpha^\vee_1 + \dots + \alpha^\vee_{R}\rangle, \qquad k=1,\dots,R+1.
\end{gather*}
 \end{lem}
\begin{proof} If ${\mathcal K}$ is cyclotomically self-dual then we must have $d_k = d^\dag_{R+2-k}$, $k=1,\dots, R+1$.
Comparing \eqref{du} and \eqref{du2} we see that this implies that
\begin{gather*}
 \langle \tilde\Lambda_\infty + \rho, \alpha^\vee_1 + \dots \alpha^\vee_k \rangle
= d_{k+1} - d_1 = d^\dag_{R+1-k} - d^\dag_{R+1}\\
\hphantom{\langle \tilde\Lambda_\infty + \rho, \alpha^\vee_1 + \dots \alpha^\vee_k \rangle}{}
=
 \langle \tilde\Lambda_\infty + \rho, \alpha^\vee_{R+1-k} + \dots +\alpha^\vee_R \rangle, \qquad k=1,\dots, R,
\end{gather*}
and hence
\begin{gather*}
 \langle \tilde\Lambda_\infty,\alpha^\vee_k \rangle = \langle \tilde\Lambda_\infty,\alpha^\vee_{R+1-k}\rangle, \qquad k=1,\dots,R.
\end{gather*}
Therefore
\begin{gather*}
 d_k + d_{R+2-k} = 2d_1 + \langle \tilde \Lambda_\infty + \rho, \alpha^\vee_1+\dots+\alpha^\vee_{R}\rangle,\qquad k=1,\dots,R+1,
\end{gather*}
and so, because the right-hand side here does not depend on $k$,
\begin{gather}
 d_k + d_{R+2-k} = \frac 2 {R+1} \sum_{j=1}^{R+1} d_j, \qquad k=1,\dots,R+1.\label{cd}
\end{gather}

Recall \eqref{dde} and the def\/inition \eqref{lambdadef} of $\Lambda$. Using now the fact that $\langle\Lambda,\alpha^\vee_i\rangle = \langle\Lambda,\alpha^\vee_{R+1-i}\rangle$, $i=1,\dots,R$, we have
\begin{gather*}
\sum_{j=1}^{R+1} d_j - \frac{(R+1)R}2 = \frac{R+1}{2} \langle \Lambda, \alpha^\vee_1 + \dots + \alpha^\vee_{R}\rangle.
\end{gather*}
Thus, given \eqref{cd}, we have the result.
\end{proof}

If ${\mathcal K}$ is cyclotomically self-dual then there is a non-degenerate bilinear form~$B$ on~${\mathcal K}$ def\/ined by
\begin{gather*}
 B(u(x),v(x)) := (u(x), v(-x)), 
\end{gather*}
i.e.,
\begin{gather*}
 B(u,v)
=\operatorname{Wr}^\dag(u,v_1,\dots,v_{R}), \qquad \text{where} \quad v(-x) = \operatorname{Wr}^\dag(v_1,\dots,v_{R}).
\end{gather*}

Let us call a tuple of quasi-polynomials $\bm y\in {\mathbb P}\big({\mathbb C}\big[x^\half\big]\big)^R$ \emph{cyclotomic} if
\begin{gather*}
 y_k(-x) \simeq y_{R+1-k}(x),\qquad k=1,\dots, R .
\end{gather*}

\begin{prop}\label{cycdec}
Let ${\mathcal F}\in \FL({\mathcal K})$. If the tuple $\beta({\mathcal F}) \in {\mathbb P}\big({\mathbb C}\big[x^\half\big]\big)^R$ is cyclotomic then ${\mathcal F}$ is a~decomposable flag.
\end{prop}

\begin{proof}
Let $\bm y^{\mathcal F}= \beta({\mathcal F})$. To prove that ${\mathcal F}$ is decomposable it is enough to show that each entry~$y^{\mathcal F}_k$ of this tuple lies in ${\mathbb C}[x]$ or in $x^{1/2}{\mathbb C}[x]$. For each $k=1,\dots,R$ we have $y^{\mathcal F}_k(x) = x^\half a_k(x) + b_k(x)$ for some polynomials $a_k(x)$ and $b_k(x)$ in $x$. If $\bm y^{\mathcal F}$ is cyclotomic then $y_{R+1-k}(x) \simeq y^{\mathcal F}_k(-x) = (-1)^\half x^\half a_k(-x) + b_k(-x)$ for each $k$. That is, $a_{R+1-k}(x) = (-1)^\half c_k  a_k(-x)$ and $b_{R+1-k}(x) = c_k b_k(-x)$ for some non-zero constants $c_k$. But that means
\begin{gather*} a_k(x)  = (-1)^\half c_{R+1-k} a_{R+1-k}(-x)  = - c_{R+1-k} c_k a_k(x),\\
b_k(x)  =  c_{R+1-k} b_{R+1-k}(-x)  = +c_{R+1-k} c_k b_k(x),
\end{gather*}
from which we conclude that for each $k$ at least one of $a_k(x)$ and $b_k(x)$ must vanish.
\end{proof}

\begin{thm}\label{cptthm}  Suppose $\beta(\FL({\mathcal K}))$ contains a cyclotomic tuple.
Then ${\mathcal K}$ is cyclotomically self-dual.
\end{thm}

\begin{proof}
We shall need the following identity among Wronskian determinants.

\begin{lem}[\cite{MV1}]\label{MVlem} Given integers $0\leq k\leq s$ and functions $f_1,\dots,f_{s+1}$, we have
\begin{gather*}   \operatorname{Wr}\big( \operatorname{Wr}(f_1,\dots,f_{s-k},\dots,f_{s},\widehat f_{s+1}),
 \operatorname{Wr}(f_1,\dots,f_{s-k},\dots,\widehat f_{s},f_{s+1}),\dots,  \\
      \qquad\operatorname{Wr}(f_1,\dots,f_{s-k},\widehat f_{s-k+1},\dots, f_{s+1}) \big)= \operatorname{Wr}( f_1,\dots,f_{s-k}) \big(\operatorname{Wr}(f_1,\dots,f_{s+1})\big)^k,
 \end{gather*}
 where $\widehat f$ denotes omission.
\end{lem}

To prove Theorem \ref{cptthm} we argue as for Theorem~6.8 in \cite{MV1}.
Let ${\mathcal F}\in \FL({\mathcal K})$ be a full f\/lag in~${\mathcal K}$ and $(u_i(x))_{i=1}^{R+1}$ a basis of ${\mathcal K}$ adjusted to this f\/lag.
Let $\bm y = \bm y^{\mathcal F}$ be the corresponding tuple of quasi-polynomials as in~\eqref{yfdef}, and $(W_i(x))_{i=1}^{R+1}$  the corresponding basis of ${\mathcal K}^\dag$ as in~\eqref{Wdef}.
Then Theorem~\ref{cptthm} follows from the case $k=R+1$ of the following lemma.
\begin{lem}\label{spanlem} If $\bm y$ is cyclotomic then
\begin{gather*}
 \operatorname{span}_{\mathbb C}(u_1(-x),\dots,u_k(-x))= \operatorname{span}_{\mathbb C}(W_{R+1},W_{R},\dots,W_{R+2-k}),\qquad k=1,\dots,R+1.
\end{gather*}
\end{lem}

\begin{proof}
Let us prove the lemma by induction on $k$. For $k=1$ we have
\begin{gather*}
 u_1(-x) = y_1(-x) \simeq  y_{R}(x) = \operatorname{Wr}^\dag(u_1,\dots,u_{R}) =  W_{R+1}
\end{gather*}
as required. Assume the statement holds for all values up to some~$k$. For the inductive step
it is enough to show that
\begin{gather}
 \operatorname{Wr}(u_1(-x),\dots,u_k(-x),W_{R+1-k}) \simeq \operatorname{Wr}(u_1(-x),\dots,u_k(-x),u_{k+1}(-x)).\label{Wue}
\end{gather}
Indeed, \eqref{Wue} is an inhomogeneous dif\/ferential equation in~$W_{R+1-k}(x)$ and if it holds then it must be that~$W_{R+1-k}(x)$ is proportional to $u_{k+1}(-x)$ modulo $\operatorname{span}_{\mathbb C}(u_1(-x),\dots,u_k(-x))$, which is suf\/f\/icient given the inductive assumption.

By the inductive assumption, we have
\begin{gather*}  \operatorname{Wr}(u_1(-x),\dots,u_k(-x),W_{R+1-k}) \\
\qquad{} \simeq \operatorname{Wr}(W_{R+1},W_{R},\dots,W_{R+1-k+1}, W_{R+1-k}) \\
\qquad{} = \operatorname{Wr}\big( \operatorname{Wr}(u_1,\dots,u_{R+1-k-1},\dots,u_{R},\widehat u_{R+1}), \\
 \qquad\quad{} \ \operatorname{Wr}(u_1,\dots,u_{R+1-k-1},\dots,\widehat u_{R},u_{R+1}),\dots,  \\
  \qquad\quad{} \ \operatorname{Wr}(u_1,\dots,u_{R+1-k-1},\widehat u_{R+1-k},\dots, u_{R+1}) \big)\big/ \big({\tilde T}_1^{R-1} {\tilde T}_2^{R+1-3} \cdots {\tilde T}_{R-1}^1\big)^{k+1} \\
\qquad{} = \operatorname{Wr}( u_1,\dots,u_{R-k}) (\operatorname{Wr}(u_1,\dots,u_{R+1}) )^k\big/ \big({\tilde T}_1^{R-1} {\tilde T}_2^{R+1-3} \cdots {\tilde T}_{R-1}^1\big)^{k+1},
 \end{gather*}
the f\/inal equality by Lemma \ref{MVlem}.
Since $\operatorname{Wr}^\dag(u_1,\dots,u_{R+1}) = \operatorname{Wr}(u_1,\dots,u_{R+1})/{\tilde T}_1^{R} {\tilde T}_2^{R-1} \cdots {\tilde T}_{R}^1$ is a nonzero constant by Lemma~\ref{uwrlem} we therefore have
\begin{gather*}  \operatorname{Wr}(u_1(-x),\dots,u_k(-x),W_{R+1-k})
\simeq \operatorname{Wr}(u_1,\dots, u_{R-k}) \frac{\big({\tilde T}_1^{R}\cdots {\tilde T}_{R}^1\big)^k}{\big({\tilde T}_1^{R-1}\cdots {\tilde T}_{R-1}\big)^{k+1}}  \\
\hphantom{\operatorname{Wr}(u_1(-x),\dots,u_k(-x),W_{R+1-k}) }{}
= \frac{\operatorname{Wr}(u_1,\dots, u_{R-k})}{{\tilde T}_1^{R+1-k-2} \cdots {\tilde T}_{R+1-k-2}^1}{\tilde T}_{R}^k \cdots {\tilde T}_{R+1-k}^1.
\end{gather*}
Now we may use again the fact that $\bm y$ is cyclotomic, so $y_k(-x) \simeq y_{R+1-k}(x)$.
In view of~\eqref{yfdef}, that implies
\begin{gather}
 \frac{\operatorname{Wr}(u_1,\dots, u_{R-k})}{{\tilde T}_1^{R+1-k-2} \cdots {\tilde T}_{R+1-k-2}^1} \simeq
    \frac{\operatorname{Wr}(u_1(-x),\dots, u_{k+1}(-x))}{{\tilde T}_1^{k}(-x) \cdots {\tilde T}_{k}^1(-x)}.\label{ucyc}
\end{gather}
Recall that ${\tilde T}_{R+1-k}(x) \simeq {\tilde T}_k(-x)$. Hence we have indeed that
\begin{gather*} \operatorname{Wr}(u_1(-x),\dots,u_k(-x),W_{R+1-k})
\simeq \operatorname{Wr}(u_1(-x),\dots, u_{k+1}(-x)),
\end{gather*}
as required.
\end{proof}

This completes the proof of Theorem \ref{cptthm}.
\end{proof}

Given a subspace $U\subset {\mathcal K}$, let
\begin{gather*}
U^\perp:= \{v\in {\mathcal K}\colon B(u,v) = 0 \text{ for all } u\in U\}
\end{gather*}
denote its orthogonal complement in ${\mathcal K}$ with respect to the bilinear form~$B$. Recall that a full f\/lag ${\mathcal F}= \{  0 =F_0 \subset F_1 \subset F_2 \subset \dots \subset F_{R} \subset F_{R+1} = {\mathcal K}\}\in \FL({\mathcal K})$ is called \emph{isotropic} with respect to $B$ if $F_k= F^\perp_{R+1-k}$ for $k=1,\dots,R$.
\begin{thm} \label{cycisothm}
Suppose ${\mathcal K}$ is cyclotomically self-dual. A full flag ${\mathcal F}\in \FL({\mathcal K})$ is isotropic if and only if the associated tuple $\bm y^{\mathcal F}$ is cyclotomic.
\end{thm}

\begin{proof}
Let $(u_i(x))_{i=1}^{R+1}$ be a basis of ${\mathcal K}$ adjusted to ${\mathcal F}$, so that we have \eqref{yfdef}.

{\sloppy For the ``only if'' direction, suppose $\bm y^{\mathcal F}$ is cyclotomic.
By Lemma \ref{spanlem},
\begin{gather*}
 F_k =  \operatorname{span}_{\mathbb C}(u_1,\dots,u_k) = \operatorname{span}_{\mathbb C}(W_{R+1}(-x),\dots,W_{R+2-k}(-x)).
\end{gather*}
We also have $F_{R+1-k}^\perp = \operatorname{span}_{\mathbb C}(u_1,\dots,u_{R+1-k})^\perp = \operatorname{span}_{\mathbb C}(W_{R+1}(-x),\dots,W_{R+2-k}(-x))$ by~\eqref{Wd}. Therefore $F_k= F_{R+1-k}^\perp$.

}

For the ``if'' direction, suppose ${\mathcal F}=\{F_k\}$ is isotropic. Since $F_k=F_{R+1-k}^\perp$, and given~\eqref{Wd}, we have two bases for~$F_k$, namely $(u_1,\dots,u_k)$ and $(W_{R+1}(-x),\dots,W_{R+2-k}(-x))$. So to prove that~$\bm y$ is cyclotomic it suf\/f\/ices to establish the following lemma, which is the converse of Lemma~\ref{spanlem}.

\begin{lem}\label{spanlemconverse}
If
\begin{gather*}
 \operatorname{span}_{\mathbb C}(u_1(-x),\dots,u_k(-x))= \operatorname{span}_{\mathbb C}(W_{R+1},W_{R},\dots,W_{R+2-k}),\qquad k=1,\dots,R+1,
\end{gather*}
then $\bm y$ is cyclotomic.
\end{lem}

\begin{proof} Examining the induction in the proof of Lemma~\ref{spanlem}, one sees that we also have, by a~similar induction, that if $\operatorname{span}_{\mathbb C}(u_1(-x),\dots,u_k(-x))= \operatorname{span}_{\mathbb C}(W_{R+1},W_{R},\dots,W_{R+2-k})$ for each~$k$ then~\eqref{ucyc} must hold for each~$k$, which says that~$\bm y$ is cyclotomic.
\end{proof}

This completes the proof of Theorem \ref{cycisothm}.
\end{proof}

In view of Proposition \ref{cycdec} we have the following corollary.
\begin{cor}\label{isodecomp} If ${\mathcal F}\in \FL({\mathcal K})$ is isotropic then ${\mathcal F}$ is decomposable.
\end{cor}

\subsection[Witt bases and the symmetries of the bilinear form $B$]{Witt bases and the symmetries of the bilinear form $\boldsymbol{B}$}

We say that $(r_k)_{k=1}^{R+1}$ is a \emph{Witt basis} of
the cyclotomically self-dual space ${\mathcal K}$ if
\begin{gather}\operatorname{Wr}^\dag(r_1,\dots,\widehat r_k,\dots,r_{R+1})
\simeq r_{R+2-k}(-x), \qquad k=1,\dots,R+1. \label{wittbasis}
\end{gather}

The following lemma gives a useful alternative characterization of Witt bases.
\begin{lem}\label{adiaglem}
The basis $(r_k)_{k=1}^{R+1}$ is a Witt basis if and only if
\begin{gather}
 B(r_i,r_j) = 0 \qquad\text{whenever}\quad i+j\neq R+2.\label{adiag}
\end{gather}
\end{lem}

\begin{proof}
Suppose $(u_k)_{k=1}^{R+1}$ is a basis of ${\mathcal K}$ and let $(W_k)_{k=1}^{R+1}$ be as in~\eqref{Wdef}. Then $(W_i(x))_{i=1}^{R+1}$ and $(u_i(-x))_{i=1}^{R+1}$ are two bases of ${\mathcal K}^\dag$ and so $u_i(-x) = \sum\limits_{j=1}^{R+1} C_{ij} W_j(x)$, for some invertible matrix $C_{ij}$. We have $B(u_i,u_{j}) = \sum\limits_{k=1}^{R+1}C_{jk} \operatorname{Wr}^\dag(u_i,u_1,u_2,\dots,\widehat u_k,\dots, u_{R+1}) = (-1)^{i-1} C_{ji} \operatorname{Wr}^\dag(u_1,\dots,u_{R+1})$. Hence~\eqref{wittbasis} is equivalent to~\eqref{adiag}.
\end{proof}

\begin{thm}\label{wittthm}
Every cyclotomically self-dual space ${\mathcal K}$ has a special basis $(r_k)_{k=1}^{R+1}$ which is also Witt basis, and in which in fact
\begin{gather}\label{rwitt}
\operatorname{Wr}^\dag(r_1,\dots,\widehat r_k,\dots,r_{R+1})
= (-1)^{-\deg  r_{R+2-k}}
r_{R+2-k}(-x), \qquad k=1,\dots,R+1.
\end{gather}
\end{thm}

\begin{proof}
Let $(u_k(x))_{k=1}^{R+1}$ be a special basis of~${\mathcal K}$. We may suppose that the $u_k(x)$ all have leading coef\/f\/icient~1. Let $(W_k(x))_{k=1}^{R+1}$ be the basis of ${\mathcal K}^\dag$ as in~\eqref{Wdef}. By Lemma~\ref{Wdlem}, $\deg W_k = d^\dag_k$.
By Lemma \ref{Wlem}, we have
\begin{gather*}
 W_k = \operatorname{Wr}^\dag(u_1,\dots,\widehat u_k,\dots, u_{R+1}) = D_k x^{d^\dag_{k}}  + \cdots,
 \end{gather*}
where the ellipsis indicates terms of lower degree in $x$ and where
\begin{gather*} D_k :=
\prod_{\substack{1\leq j<i\leq R+1 \\i\neq k,\,j\neq k}}
( d_i - d_j), \qquad k=1,\dots, R+1.
\end{gather*}
Since ${\mathcal K}$ is cyclotomically self-dual we must have
\begin{gather*}
 d_k = d^\dag_{R+2-k}, \qquad k=1,\dots, R+1,
 \end{gather*}
and
\begin{gather*}
 W_k = \operatorname{Wr}^\dag(u_1,\dots, \hat u_k,\dots, u_{R+1}) = D_k (-1)^{-d_{R+2-k}} u_{R+2-k}(-x) +\cdots.
\end{gather*}

Now from~\eqref{du} we have
\begin{gather*}
 d_k - d_l
    = d_{R+2-l} - d_{R+2-k}, \qquad 1\leq k< l \leq R+1,
\end{gather*}
using which one verif\/ies that
\begin{gather*}
 D_k = D_{R+2-k}, \qquad k=1,\dots,R+1.
\end{gather*}
Given this equality, if we set
\begin{gather*}
 q_k := u_k D_k^{\frac 1 2} \prod_{j=1}^{R+1} D_j^{-\frac 1 {2R-2}}
\end{gather*}
then we have
\begin{gather*}
 \operatorname{Wr}^\dag(q_1,\dots,\widehat q_k,\dots, q_{R+1}) = (-1)^{-d_{R+2-k}} q_{R+2-k}(-x) + \cdots .
\end{gather*}

In this way, we arrive at
\begin{gather}
 \operatorname{Wr}^\dag(q_1,\dots,\widehat q_k,\dots, q_{R+1})\nonumber \\
 \qquad {} = (-1)^{-d_{R+2-k}} q_{R+2-k}(-x) +  \sum_{j=2}^{R+1} q_{R+2-j}(-x) c^j_k,\qquad k=1,\dots,R+1,\label{lteq}
\end{gather}
for some constants  $c^j_k$. That is, we have
\begin{gather*} \operatorname{Wr}^\dag(q_1,\dots,q_{R-1},q_{R},\widehat q_{R+1})  = (-1)^{-d_1} q_1(-x), \\
 \operatorname{Wr}^\dag(q_1,\dots,q_{R-1},\widehat q_{R},q_{R+1})  = (-1)^{-d_2} q_2(-x) + c_{R}^1 q_1(-x),\\
\operatorname{Wr}^\dag(q_1,\dots,\widehat q_{R-1},q_{R},q_{R+1})  = (-1)^{-d_3} q_3(-x) + c_{R-1}^2 q_2(-x) + c_{R-1}^1 q_1(-x),\\
\cdots \cdots \cdots \cdots \cdots \cdots \cdots \cdots \cdots \cdots \cdots \cdots \cdots \cdots \cdots. \end{gather*}
We def\/ine
\begin{gather*} (-1)^{-d_1} r_{1}  := (-1)^{-d_1} q_{1}, \qquad
 (-1)^{-d_2} r_{2}  := (-1)^{-d_2} q_{2} + c_{R}^1 r_1,
 \end{gather*}
so that
\begin{gather*} \operatorname{Wr}^\dag(r_1,r_2,q_3,\dots,q_{R-1},q_{R},\widehat q_{R+1})  = (-1)^{-d_1} r_1(-x),\\
 \operatorname{Wr}^\dag(r_1,r_2,q_3,\dots,q_{R-1},\widehat q_{R},q_{R+1})  = (-1)^{-d_2} r_2(-x),\\
\operatorname{Wr}^\dag(r_1,r_2,q_3,\dots,\widehat q_{R-1},q_{R},q_{R+1})  = (-1)^{-d_3} q_3(-x) + \tilde c_{R-1}^2 r_2(-x) + \tilde c_{R-1}^1 r_1(-x),\\
\cdots \cdots \cdots \cdots \cdots \cdots \cdots \cdots \cdots \cdots \cdots \cdots \cdots \cdots  \cdots
\end{gather*}
for some new constants $\tilde c^j_k$, and we then def\/ine
\begin{gather*}
  (-1)^{-d_3} r_{3} := (-1)^{-d_3} q_{3} + \tilde c_{R-1}^1 r_2 + \tilde c_{R-1}^1 r_1,
  \end{gather*}
and so on. By an obvious induction, we arrive at a Witt basis  $(r_k)_{k=1}^{R+1}$. By construction $\deg r_k = d_k$.
Finally, note in~\eqref{lteq} that~$c^j_k$ can be non-zero only when $d_k-d_j\in {\mathbb Z}$ since both sides lie in either~${\mathbb C}[x]$ or $x^\half{\mathbb C}[x]$. Therefore this Witt basis  $(r_k)_{k=1}^{R+1}$ is special.
\end{proof}

\begin{lem}\label{wr1lem}
Let $(r_k)_{k=1}^{R+1}$ be the Witt basis of Theorem~{\rm \ref{wittthm}}. Then
$\operatorname{Wr}^\dag(r_1,\dots,r_{R+1}) = 1$.
\end{lem}

\begin{proof}
We have
\begin{gather*} \operatorname{Wr}^\dag(q_1,\dots,q_{R+1})
 = \operatorname{Wr}^\dag(u_1,\dots,u_{R+1}) \prod_{k=1}^{R+1} \left(D_k^{\frac 1 2} \prod_{j=1}^{R+1} D_j^{-\frac 1{2R-2}}\right)  \\
 \hphantom{\operatorname{Wr}^\dag(q_1,\dots,q_{R+1})}{}
 = \operatorname{Wr}^\dag(u_1,\dots,u_{R+1}) \prod_{k=1}^{R+1} D_k^{-\frac 1 {R-1}} .
 \end{gather*}
But then noting that
\begin{gather*}
 \prod_{k=1}^{R+1} D_k = \prod_{1\leq j<i\leq R+1}(d_i - d_j)^{R-1}
 \end{gather*}
and recalling Lemma \ref{uwrlem}, one f\/inds
\begin{gather*}
  \operatorname{Wr}^\dag(q_1,\dots,q_{R+1})  =1
\end{gather*}
and hence the result.
\end{proof}

\begin{thm}\label{Bthm} The subspaces ${\mathcal K}_{\textup{Sp}}$ and ${\mathcal K}_{\textup{O}}$ are mutually orthogonal with respect to~$B$.
  The bilinear form $B$ is skew-symmetric on ${\mathcal K}_{\textup{Sp}}$ and symmetric on ${\mathcal K}_{\textup{O}}$.
\end{thm}

\begin{proof}
Let $(r_k)_{k=1}^{R+1}$ be the special Witt basis constructed in Theorem \ref{wittthm}. From~\eqref{rwitt} and Lemma~\ref{wr1lem}, we have
\begin{gather}
 B(r_k,r_{R+2-k}) = (-1)^{d_{R+2-k}+k+1}, \qquad k=1,\dots, R+1,\label{Be}
 \end{gather}
and $B(r_i,r_j)=0$ if $i+j \neq R+2$. This implies in particular that ${\mathcal K}_{\textup{Sp}}$ and ${\mathcal K}_{\textup{O}}$ are mutually orthogonal.
By Lemma \ref{cd2lem} it also gives
\begin{gather*}
 B(r_k,r_{R+2-k}) B(r_{R+2-k},r_k) = (-1)^{\langle \Lambda, \alpha^\vee_1 + \dots + \alpha^\vee_{R}\rangle},\qquad k = 1,\dots, R+1.
\end{gather*}

Recall the def\/inition of $\Lambda$, \eqref{lambdadef}. Now $\langle \Lambda_s+ \sigma\Lambda_s  ,\alpha^\vee_1 + \dots + \alpha^\vee_{R}\rangle \in 2{\mathbb Z}$ for each $s=1,\dots,N$, since $\Lambda_s$ is integral. Therefore it follows from~\eqref{l0a} that
\begin{gather*}
\langle \Lambda, \alpha^\vee_1 + \dots + \alpha^\vee_{R}\rangle \in \begin{cases} 2{\mathbb Z}+1, & p>0,\\ 2{\mathbb Z}, & p=0.\end{cases}
\end{gather*}

Consider the case $p>0$. Then we have
\begin{gather}
B(r_k,r_{R+2-k}) B(r_{R+2-k},r_k) = -1, \qquad k=1,\dots,R+1.\label{m1}
\end{gather}
Recall from \eqref{kpmpoly} that $\deg r_k$ and $\deg r_{R+2-k}$ are both half odd integers if $k=p+1,\dots,R+1-p$, and are integers otherwise. Hence, by~\eqref{Be}, $B(r_k,r_{R+2-k})$ and $B(r_{R+2-k},r_k)$ lie in $\{(-1)^{\half},(-1)^{-\half}\}$ if $k=p+1,\dots,R+1-p$ and in $\{1,-1\}$ otherwise. Combining this statement with~\eqref{m1} we f\/ind
\begin{gather*}
  B(r_k,r_{R+2-k}) = \begin{cases} - B(r_{R+2-k},r_k),  & k=1,\dots,p,R+2-p,\dots,R+1, \\
                                        + B(r_{R+2-k},r_k),  & k= p+1,\dots,R+1-p, \end{cases}
\end{gather*}
which is the required result.

Finally, consider the case $p=0$. Then
\begin{gather*}
 B(r_k,r_{R+2-k}) B(r_{R+2-k},r_k) = 1, \qquad k=1,\dots,R+1,
 \end{gather*}
and since in this case $\deg r_k$ is integral for all $k$, this implies
\begin{gather*}
 B(r_k,r_{R+2-k}) = B(r_{R+2-k},r_k), \qquad k=1,\dots,R+1
 \end{gather*}
as required.
\end{proof}

The following are corollaries of Theorem \ref{Bthm} together with Lemma~\ref{adiaglem}.
\begin{cor}\label{wittdecomp}
Every Witt basis $(r_k)_{k=1}^{R+1}$  of ${\mathcal K}$ is decomposable.
\end{cor}

A basis  $(r_k)_{k=1}^{R+1}$ of $K$  such that
\begin{gather*}
 B_{ij} := B(r_i,r_j) = \delta_{R+2-i,j} b_i
 \end{gather*}
with
\begin{gather*} b_k := \begin{cases} (-1)^{k}, & k= 1,\dots,p,\\
                         \,\,+1, & k= p+1,\dots,R+1-p,\\
                        (-1)^{R+1-k}, & k= R+2-p,\dots,R+1 \end{cases}
\end{gather*}
is called  a \emph{reduced Witt basis}. By Lemma \ref{adiaglem}, reduced Witt bases are Witt bases.

\begin{cor} Any Witt basis can be transformed to a reduced Witt basis by a suitable diagonal transformation followed by a suitable permutation of the basis vectors.
\end{cor}

\begin{cor}\label{flaglem}
For any Witt basis $(r_k)_{k=1}^{R+1}$ of ${\mathcal K}$, the full flag ${\mathcal F}=\{F_1\!\subset F_2\! \subset \!\cdots\! \subset F_{R+1} = {\mathcal K}\}$ given by $F_k = \operatorname{span}_{\mathbb C}(r_1,\dots,r_k)$, $k=1,\dots, R+1$, is isotropic $($and hence the corresponding tuple~$\bm y^{\mathcal F}$ is cyclotomic by Theorem~{\rm \ref{cycisothm})}.

Conversely, given any isotropic full flag ${\mathcal F}=\{F_1\subset F_2\subset \dots \subset F_{R+1} = {\mathcal K}\}$ there is a Witt basis $(r_k)_{k=1}^{R+1}$ such that  $F_k = \operatorname{span}_{\mathbb C}(r_1,\dots,r_k)$, $k=1,\dots, R+1$. If in addition~${\mathcal F}$ is of type~$S$ then this basis can be chosen to be a reduced Witt basis.
\end{cor}

\begin{lem} The full flag ${\mathcal F}$ given in~\eqref{smallcell} is isotropic  and hence the corresponding tuple $\bm y^{\mathcal F}$ is cyclotomic.
\end{lem}

\begin{proof} We can choose the special basis $(u_k)_{k=1}^{R+1}$ def\/ining ${\mathcal F}$ to be the Witt basis of Theo\-rem~\ref{wittthm}. Then the result follows from Corollary~\ref{flaglem}.
\end{proof}

\subsection{Isotropic f\/lags}
Recall from Section~\ref{flagsec} the notion of a symmetric subset of $\{1,\dots,R+1\}$.
\begin{lem} Let $Q\subset \{1,\dots, R+1\}$ be a $2p$-element subset. The variety $\FL_Q({\mathcal K})$ contains an isotropic flag if and only if~$Q$ is symmetric.
\end{lem}

\begin{lem} If $Q$ is symmetric then the variety $\FL_Q^\perp({\mathcal K})$ of isotropic flags is isomorphic to the direct product of spaces of isotropic flags $\FL^\perp({\mathcal K}_{\textup{Sp}})\times \FL^\perp({\mathcal K}_{\textup{O}})$ and the isomorphism of these varieties is given by the map~$\eta_Q$ defined in~\eqref{etaqdef}.
\end{lem}

In view of these lemmas and Theorem~\ref{cycisothm}, we have the following description of the subspace of all cyclotomic tuples within the image $\beta(\FL({\mathcal K})) \subset {\mathbb P}\big({\mathbb C}\big[x^\half\big]\big)^{R}$.
\begin{thm}\label{cdt} The irreducible components of the space $\beta(\FL^\perp({\mathcal K}))$  of all cyclotomic tuples  are labeled by symmetric subsets $Q\subset \{1,\dots,R+1\}$. The components do not intersect and each is isomorphic to $\FL^\perp({\mathcal K}_{\textup{Sp}})\times \FL^\perp({\mathcal K}_{\textup{O}})$.
\end{thm}

\subsection[Inf\/initesimal deformation of isotropic f\/lags of type $S$]{Inf\/initesimal deformation of isotropic f\/lags of type $\boldsymbol{S}$}\label{idif}

The connected Lie group of endomorphisms of ${\mathcal K}$ preserving $B$ acts transitively on the variety of isotropic full f\/lags of type $Q$, $\FL_Q^\perp({\mathcal K})$, for each symmetric subset $Q\subset\{1,\dots,R+1\}$. In particular it acts transitively on $\FL_S^\perp({\mathcal K})$, and hence on the cyclotomic tuples of polynomials in the image $\beta(\FL_S^\perp({\mathcal K})) \subset {\mathbb P}({\mathbb C}[x])^{R}$.
We shall describe  the inf\/initesimal action of this group on~$\beta(\FL_S^\perp({\mathcal K}))$.

The connected Lie group of endomorphisms of ${\mathcal K}$ preserving $B$ preserves each of the sub\-spa\-ces~${\mathcal K}_{\textup{Sp}}$ and ${\mathcal K}_{\textup{O}}$. Thus this group is the product $\mathrm{Sp}({\mathcal K}_{\textup{Sp}}) \times   \mathrm{SO}({\mathcal K}_{\textup{O}})$ of the group of special symplectic transformations in $\operatorname{End}({\mathcal K}_{\textup{Sp}})$ and the group of special orthogonal transformations in $\operatorname{End}({\mathcal K}_{\textup{O}})$.
Its Lie algebra ${\mathfrak{sp}}({\mathcal K}_{\textup{Sp}}) \oplus {\mathfrak{so}}({\mathcal K}_{\textup{O}})$ consists of all
traceless
endomorphisms $X$ of ${\mathcal K}$ such that
\begin{gather*} B(Xu,v) + B(u, Xv) = 0
\end{gather*}
for all $u,v\in {\mathcal K}$.

Pick any isotropic full f\/lag ${\mathcal F}=\{F_1\subset F_2\subset \dots \subset F_{R+1} = {\mathcal K}\}$ of type $S$. Then $\beta({\mathcal F})=\bm y^{\mathcal F}$ is a cyclotomic tuple of polynomials by Lemma~\ref{Simlem}.  Let $(r_k)_{k=1}^{R+1}$ be a reduced Witt basis such that $F_k = \operatorname{span}_{\mathbb C}(r_1,\dots,r_k)$, $k=1,\dots, R+1$. Such a basis exists by Corollary~\ref{flaglem}.

This choice of basis gives identif\/ications ${\mathcal K}_{\textup{Sp}}\cong{\mathbb C}^{2p}$ and ${\mathcal K}_{\textup{O}} \cong {\mathbb C}^{R+1-2p}$ and hence ${\mathfrak{sp}}({\mathcal K}_{\textup{Sp}})\cong {\mathfrak{sp}}_{2p}$ and ${\mathfrak{so}}({\mathcal K}_{\textup{O}}) \cong {\mathfrak{so}}_{R+1-2p}$.
The Lie algebra ${\mathfrak{sp}}_{2p}$ has root system of type $C_p$.
The Lie algebra ${\mathfrak{so}}_{R+1-2p}$ has root system of type $D_{n-p}$ if $R=2n-1$ is odd and of type $B_{n-p}$ if $R=2n$ is even.

Let $(E_{i,j})_{i,j=1}^{R+1}$ be the basis of $\operatorname{End}({\mathcal K})$ def\/ined by
\begin{gather*} E_{i,j} r_k = \delta_{ik} r_j.
\end{gather*}

The lower-triangular subalgebra of ${\mathfrak{sp}}({\mathcal K}_{\textup{Sp}})\cong {\mathfrak{sp}}_{2p}$ is generated by
\begin{gather*} X_{k} := E_{k+1,k} + E_{R+2-k,R+1-k}, \qquad k=1,\dots, p-1,
\end{gather*}
and
\begin{gather*}
 X_{p} := E_{R+2-p,p}.
 \end{gather*}

When $R=2n-1$, the lower-triangular subalgebra of ${\mathfrak{so}}({\mathcal K}_{\textup{O}})\cong {\mathfrak{so}}_{2n-2p}$ is generated by
\begin{gather*} Y_{k} := E_{k+p,k+p-1} - E_{2n-p-k+1,2n-p-k}, \qquad k=1,\dots,n-p-1 ,
\end{gather*}
and
\begin{gather*}
 \tilde Y_{n-p-1} := E_{k+p+1,k+p-1} - E_{2n-p-k+1,2n-p-k-1}.
 \end{gather*}

When $R=2n$, the lower-triangular subalgebra of ${\mathfrak{so}}({\mathcal K}_{\textup{O}})\cong {\mathfrak{so}}_{2n-2p+1}$ is generated by
\begin{gather*}
 Z_{k} := E_{k+p,k+p-1} - E_{2n-p-k+2,2n-p-k+1}, \qquad k=1,\dots,n-p.
 \end{gather*}

These generators def\/ine linear transformations belonging to  $\operatorname{End}({\mathcal K}_{\textup{Sp}})\oplus \operatorname{End}({\mathcal K}_{\textup{O}})$.

\begin{rem} The Lie algebra ${\mathfrak{so}}({\mathcal K}_{\textup{O}}) \oplus {\mathfrak{sp}}({\mathcal K}_{\textup{Sp}})$ is contained in the simple Lie superalgeb\-ra~${\mathfrak{osp}}({\mathcal K})$ of all orthosymplectic transformations of the space~${\mathcal K}$. See~\cite{KacSuper} for the def\/inition. It would be interesting to understand whether this superalgebra plays a role here.
\end{rem}

For any $k=1,\dots, p$ and all $c\in {\mathbb C}$, the basis $e^{c X_k} \bm r$ is again a Witt basis of ${\mathcal K}$. Let $e^{cX_k}{\mathcal F}$ denote the corresponding isotropic f\/lag and $\beta(e^{cX_k}{\mathcal F})$ the corresponding tuple representing a cyclotomic
point.
Let us describe the dependence on $c$ of this tuple.

For $k=1,\dots,p-1$, we have
\begin{gather*}
 e^{cX_k} \bm r = (r_1,\dots,r_{k-1}, r_k + c r_{k+1}, r_{k+1},\dots,
                               r_{R-k}, r_{R+1-k} + cr_{R+2-k}, r_{R+2-k},\dots, r_{R+1})
\end{gather*}
and hence
\begin{gather*}
 \beta\big(e^{cX_k}{\mathcal F}\big)
= \big(y_1^{\mathcal F}, \dots, y_{k-1}^{\mathcal F}, y_k(x,c), y_{k+1}^{\mathcal F},\dots, y_{R+1-k}^{\mathcal F}, y_{R+1-k}(x,c), y_{R+2-k}^{\mathcal F},\dots, y_{R+1}^{\mathcal F}\big),
\end{gather*}
where
\begin{subequations}\label{yfe}
\begin{gather} y_k(x,c)  :=  \operatorname{Wr}^\dag(r_1,\dots,r_{k-1},r_k+c r_{k+1})  = y^{\mathcal F}_k + c \operatorname{Wr}^\dag(r_1,\dots,r_{k-1}, r_{k+1})
\end{gather}
and
\begin{gather} y_{R+1-k}(x,c)  :=  \operatorname{Wr}^\dag(r_1,\dots,r_{R-k},r_{R+1-k}+c r_{R+2-k}) \nonumber\\
\hphantom{y_{R+1-k}(x,c)}{}
 = y^{\mathcal F}_{R+1-k} + c \operatorname{Wr}^\dag(r_1,\dots,r_{R-k}, r_{R+2-k}).
 \end{gather}
\end{subequations}
Finally (for $k=p$) we have
\begin{gather*}
 e^{cX_p} \bm r = (r_1,\dots,r_{p-1}, r_p + c r_{R+2-p}, r_{p+1},\dots, r_{R+1})
 \end{gather*}
and hence
\begin{gather} \beta\big(e^{cX_p}{\mathcal F}\big)
= \big(y_1^{\mathcal F}, \dots, y_{p-1}^{\mathcal F}, y_p(x,c), y_{p+1}^{\mathcal F},\dots, \dots, y_{R+1}^{\mathcal F}\big),\nonumber\\
 y_{p}(x,c) :=  \operatorname{Wr}^\dag(r_1,\dots,r_{p-1},r_{p}+c r_{R+2-p})
= y^{\mathcal F}_{p} + c \operatorname{Wr}^\dag(r_1,\dots,r_{p-1}, r_{R+2-p}).\label{yfe2}
\end{gather}

The f\/lows in ${\mathbb P}({\mathbb C}[x])^R$ corresponding to the generators of ${\mathfrak{so}}({\mathcal K}_{\textup{O}})$ can be described similarly.

\subsection[Populations of cyclotomic critical points in type $A$]{Populations of cyclotomic critical points in type $\boldsymbol{A}$}\label{rsec}

Recall the def\/inition of the extended master function $\widehat\Phi$, \eqref{emf}. In the setting of the present section (see Section~\ref{Asetup}) it has the explicit form
\begin{gather}
\widehat\Phi(\bm t;\bm {\mathsf c};\bm z;\bm \Lambda; \Lambda_0) =
\sum_{i=1}^N \lf \Lambda_0, \Lambda_i\rf (\log (-z_i) + \log(z_i)) +\sum_{i=1}^N (\Lambda_i,\Lambda_{R+1-i}) \log 2 z_i\nonumber\\
\hphantom{\widehat\Phi(\bm t;\bm {\mathsf c};\bm z;\bm \Lambda; \Lambda_0) =}{}
+ \sum_{1\leq i<j\leq  n} \lf  \Lambda_i,\Lambda_j\rf \log (z_i- z_j)
+ \sum_{1\leq i<j\leq  n} \lf  \Lambda_{R+1-i},\Lambda_j\rf \log (-z_i- z_j)\nonumber\\
\hphantom{\widehat\Phi(\bm t;\bm {\mathsf c};\bm z;\bm \Lambda; \Lambda_0) =}{}
+ \sum_{1\leq i<j\leq  n} \lf  \Lambda_i,\Lambda_{R+1-j}\rf \log (z_i+ z_j)\nonumber\\
\hphantom{\widehat\Phi(\bm t;\bm {\mathsf c};\bm z;\bm \Lambda; \Lambda_0) =}{}
+ \sum_{1\leq i<j\leq  n} \lf  \Lambda_{R+1-i},\Lambda_{R+1-j} \rf \log (-z_i+ z_j)
- \sum_{j=1}^{\tilde m}  \lf \alpha_{{\mathsf c}(j)},\Lambda_0\rf \log (\t_j)\nonumber\\
\hphantom{\widehat\Phi(\bm t;\bm {\mathsf c};\bm z;\bm \Lambda; \Lambda_0) =}{}
- \sum_{i=1}^N\sum_{j=1}^{\tilde m} \lf \alpha_{{\mathsf c}(j)},\Lambda_i\rf \log (\t_j-z_i)
- \sum_{i=1}^N\sum_{j=1}^{\tilde m} \lf \alpha_{{\mathsf c}(j)},\Lambda_{R+1-i}\rf \log (\t_j+z_i)
\nonumber\\
\hphantom{\widehat\Phi(\bm t;\bm {\mathsf c};\bm z;\bm \Lambda; \Lambda_0) =}{}
+ \sum_{1\leq i< j \leq \tilde m} \lf \alpha_{{\mathsf c}(i)},\alpha_{{\mathsf c}(j)}\rf \log (\t_i-\t_j)\label{emfA}
\end{gather}
and the critical point equations \eqref{ebe} become
\begin{gather}
0=
\sum_{i=1}^N
\frac{\lf \alpha_{{\mathsf c}(j)},\Lambda_i\rf}{\t_j-z_i}
+
\sum_{i=1}^N
\frac{\lf \alpha_{{\mathsf c}(j)},\Lambda_{R+1-i}\rf}{\t_j+z_i}\nonumber\\
\hphantom{0=}{}
+  \frac{\lf \alpha_{{\mathsf c}(j)},\Lambda_0\rf}{\t_j}
-  \sum_{\substack{i=1\\i\neq j}}^{\tilde m}
\frac{\lf \alpha_{{\mathsf c}(j)},\alpha_{{\mathsf c}(i)}\rf}{\t_j-\t_i} , \qquad j=1,\dots,\tilde m.
\label{ebeA}
\end{gather}

Given a tuple of polynomials $\bm y\in {\mathbb P}({\mathbb C}[x])^{R}$, we have the pair $(\bm t, \bm {\mathsf c})\in {\mathbb C}^{\tilde m}\times I^{\tilde m}$ represented by~$\bm y$ in the sense of Section~\ref{prsec}. We say the tuple~$\bm y$ \emph{represents a critical point of $\widehat\Phi$} if $\bm t$ is a critical point of $\widehat\Phi(\bm t; \bm{\mathsf c}; \bm z; \bm \Lambda;\Lambda_0)$, i.e., if $(\bm t,\bm{\mathsf c})$ satisfy the equations~\eqref{ebeA}.

The following theorem says that we can go from cyclotomic critical points of the extended master function~$\widehat\Phi$,~\eqref{emfA}, to decomposable cyclotomically self-dual vector spaces of quasi-polynomials.

\begin{thm}\label{KCkerthm}
Suppose $\bm y\in {\mathbb P}({\mathbb C}[x])^{R}$ represents a cyclotomic critical point of~$\widehat\Phi$,~\eqref{emfA}.

The kernel $\ker {\mathcal D}(\bm y)$ of the fundamental differential operator ${\mathcal D}(\bm y)$, Section~{\rm \ref{Dsec}}, is a decomposable cyclotomically self-dual vector space of quasi-polynomials with frame $\tilde T_1,\dots, \tilde T_{R};\tilde \Lambda_\infty$, where~$\tilde\Lambda_\infty$ is the unique dominant weight in the orbit of~$\Lambda_\infty(\bm y)$,~\eqref{l8y}, under the shifted action of the Weyl group of type~$A_R$.

There exists an isotropic flag ${\mathcal F}\in \FL^\perp_S(\ker {\mathcal D}(\bm y))$ such that $\bm y= \beta({\mathcal F})$.
\end{thm}

\begin{proof}
Arguing as in \cite{MV1} -- see especially Lemma~5.10 -- we have that  $\ker {\mathcal D}(\bm y)$ is a vector space of quasi-polynomials with frame $\tilde T_1,\dots, \tilde T_{R};\tilde \Lambda_\infty$, and that the f\/lag ${\mathcal F}\in \FL(\ker({\mathcal D}(\bm y)))$ such that $\beta({\mathcal F}) = \bm y$ can be constructed as follows.
Def\/ine  quasi-polynomials $y^{(i,i+1,\dots,k)}_i$, $1\leq i\leq k\leq R$, recursively by
\begin{gather*} \operatorname{Wr}(y^{(k)}_k , y_k)  = y_{k-1} {\tilde T}_k y_{k+1},\qquad
 \operatorname{Wr}(y^{(i,i+1,\dots,k)}_i, y_i)  = y_{i-1} {\tilde T}_i y^{(i+1,\dots,k)}_{i+1},\qquad i<k
\end{gather*}
(recall we set $y_0 = y_{R+1} = 1$ for convenience). Set $u_1 = y_1$ and $u_k= y^{(1,\dots,k-1)}_1$ for $k=2,\dots,R+1$. Then $(u_k)_{k=1}^{R+1}$ is a basis of $\ker {\mathcal D}(\bm y)$. Moreover $\operatorname{Wr}^\dag(u_1,\dots,u_k) = y_k$, $k=1,\dots, R$. That is, $\beta({\mathcal F}) = \bm y$ for the f\/lag ${\mathcal F} = \{F_k\}$ given by $F_k = \operatorname{span}_{\mathbb C}(u_1,\dots, u_k)$, $k=1,\dots, R+1$.
Since $\bm y$ is cyclotomic, Theorem~\ref{cptthm} states that $\ker{\mathcal D}(\bm y)$ is cyclotomically self-dual. By Lemma~\ref{QSlem}, ${\mathcal F}$ is a decomposable f\/lag of type $S$, and by Theorem~\ref{cycisothm} it is isotropic.
\end{proof}

Conversely, we have the following, arguing as in Lemmas 3.1, 3.2 and 5.15 in~\cite{MV1} and using Theorem~\ref{cycisothm}.

\begin{thm}\label{KDCthm}
Let ${\mathcal K}$ be a decomposable cyclotomically self-dual vector space of quasi-po\-ly\-no\-mials with frame $\tilde T_1,\dots, \tilde T_{R};\tilde \Lambda_\infty$.

Suppose there exists an isotropic flag ${\mathcal F}\in \FL_S^\perp({\mathcal K})$ such that the tuple~$\bm y^{\mathcal F}$ is generic. Then~$\bm y^{\mathcal F}$ represents a cyclotomic critical point of~$\widehat\Phi$,~\eqref{emfA}.
\end{thm}

Since being generic is an open condition, the set of generic tuples in the image  $\beta(\FL_S^\perp({\mathcal K}))$ is either empty or it is open and dense in $\beta(\FL_S^\perp({\mathcal K}))$.

Starting from an initial tuple $\bm y$ that represents a cyclotomic critical point of~$\widehat\Phi$,~\eqref{emfA}, we may let ${\mathcal K} = \ker {\mathcal D}(\bm y)$ as in Theorem~\ref{KCkerthm}.
Then we have the variety
\begin{gather}
 \beta(\FL_S^\perp({\mathcal K})) \cong \FL^\perp({\mathcal K}_{\textup{Sp}}) \times \FL^\perp({\mathcal K}_{\textup{O}}),\label{popdef}
 \end{gather}
where the isomorphism is by Theorem~\ref{cdt}.
Almost all of the tuples in $\beta(\FL_S^\perp({\mathcal K}))$ are generic and hence represent cyclotomic critical points of $\widehat\Phi$.
Call this variety $\beta(\FL_S^\perp({\mathcal K}))\subset {\mathbb P}({\mathbb C}[x])^{R}$ the \emph{cyclotomic population originated at~$\bm y$}.

\subsection[The case $p=n$]{The case $\boldsymbol{p=n}$}

Consider the case $p=n$ in \eqref{l0a}.
Namely, suppose that we are either in
\begin{itemize}\itemsep=0pt
\item  type $A_{2n-1}$ with $\Lambda_0$ integral and $\langle\Lambda_0,\alpha_n\rangle$ odd, or
\item  type $A_{2n}$ with $\langle\Lambda_0,\alpha^\vee_i\rangle\in {\mathbb Z}$ for all $i<n$ and $\langle\Lambda_0,\alpha^\vee_n\rangle \in \half (2{\mathbb Z}_{\geq 0}-1)  = \big\{{-}\half,\half,\frac 3 2,\dots\big\}$.
\end{itemize}
Then $\Lambda_0$ obeys the assumptions from Section~\ref{l0sec} and so we are in the setting of Section~\ref{cycgensec}. That means we have two notions of a cyclotomic population: the one in the previous subsection, and the one in Section~\ref{cycpopdefsec}. Let us show that these two notions coincide.

\begin{thm}\label{flowthm} Let $p=n$ in~\eqref{l0a}. Let $\bm y$ represent a cyclotomic critical point of the extended master function~$\widehat\Phi$ of~\eqref{emfA}. Then the variety $\beta(\FL_S^\perp({\mathcal K}))$ is isomorphic to the
variety of isotropic full flags in a complex symplectic vector space of dimension $2n$.
The cyclotomic population in ${\mathbb P}({\mathbb C}[x])^R$ originated at $\bm y$ in the sense of Section~{\rm \ref{cycpopdefsec}}
coincides with this va\-riety~$\beta(\FL_S^\perp({\mathcal K}))$.
\end{thm}

\begin{proof}
When $p=n$ we have either ${\mathcal K}_{\textup{O}}=\{0\}$, if $R=2n-1$, or ${\mathcal K}_{\textup{O}}\cong {\mathbb C}$, if $R=2n$. In either case $\FL^\perp({\mathcal K}_{\textup{O}})$ is a point, and \eqref{popdef} reduces to
\begin{gather*}
\beta\big(\FL_S^\perp({\mathcal K})\big) \cong \FL^\perp({\mathcal K}_{\textup{Sp}}),
\end{gather*}
i.e., $\beta(\FL_S^\perp({\mathcal K}))$ is isomorphic to the variety of isotropic full f\/lags in the vector space ${\mathcal K}_{\textup{Sp}} \cong {\mathbb C}^{2n}$ endowed with the symplectic form $B|_{{\mathcal K}_{\textup{Sp}}}$.

Starting from any such isotropic full f\/lag, ${\mathcal F}\in \FL_S^\perp({\mathcal K})$, we choose a reduced Witt basis adapted to~${\mathcal F}$ (Corollary~\ref{flaglem}). Then every other f\/lag in $\FL_S^\perp({\mathcal K})$ can be reached by an element of the lower-triangular (as in Section~\ref{idif}) unipotent subgroup of~$\mathrm{Sp}_{2n}$. This subgroup is generated by the one-parameter groups corresponding to negative simple root generators $X_k$ of Section~\ref{idif}. Lemmas~\ref{f1lem} and~\ref{f2lem} below show that the f\/lows in $\beta(\FL_S^\perp({\mathcal K}))$ generated by the $X_k$ coincide with notion of cyclotomic generation from Section~\ref{cycgensec}.  That shows that the set of all tuples of polynomials obtained from $\bm y$ by repeated cyclotomic generation, in all directions $i\in I$, contains a non-empty open subset of $\beta(\FL_S^\perp({\mathcal K}))$. Therefore it is dense in $\beta(\FL_S^\perp({\mathcal K}))$. Hence its Zariski closure is $\beta(\FL_S^\perp({\mathcal K}))$ itself.
\end{proof}

\begin{lem}\label{f1lem}
The image $\beta(e^{cX_k}{\mathcal F})\in {\mathbb P}({\mathbb C}[x])^{R}$ coincides with the tuple $\bm y^{(k,\sigma)}(1/c)$ of Theorem~{\rm \ref{s1thm}}, for every  $k=1,\dots,n-1$ $($and also for $k=n$ when we are in type $A_R=A_{2n-1})$.
\end{lem}

\begin{proof}
It is enough to note that, in view of~\eqref{yfe} and Lemma~\ref{MVlem}, we have
\begin{gather*} \operatorname{Wr}(y_k^{\mathcal F},y_k(x,c))  = c{\tilde T}_k y_{k-1}^{\mathcal F} y_{k+1}^{\mathcal F},\\
    \operatorname{Wr}(y_{R+1-k}^{\mathcal F},y_{R+1-k}(x,c))  = c{\tilde T}_{R+1-k} y_{R-k}^{\mathcal F} y_{R+2-k}^{\mathcal F}.\tag*{\qed}
    \end{gather*}
\renewcommand{\qed}{}
\end{proof}

It remains to consider the case $k=n$ in type $A_{2n}$.
\begin{lem}\label{f2lem}
In type $A_{2n}$, the image $\beta(e^{-cX_n}{\mathcal F})\!\in\! {\mathbb P}({\mathbb C}[x])^{2n}$ coincides with the tuple $\bm y^{(n,\sigma)}(1/c)$ of Theo\-rem~{\rm \ref{s2thm}}.
\end{lem}

\begin{proof}
We have \eqref{yfe2} with $p=n$. Namely,
\begin{gather*}
 e^{cX_n} \bm r = (r_1,\dots, r_{n-1}, r_n + c r_{n+2}, r_{n+1}, r_{n+2}, \dots, r_{2n+1})
 \end{gather*}
and hence
\begin{gather*} \beta\big(e^{cX_n}{\mathcal F}\big)
= \big(y_1^{\mathcal F}, \dots, y_{n-1}^{\mathcal F}, y_n(x,c), y_{n+1}(x,c), y_{n+2}^{\mathcal F} ,\dots, y_{2n+1}^{\mathcal F}\big),
\end{gather*}
where
\begin{gather*} y_n(x,c)  :=  \operatorname{Wr}^\dag(r_1,\dots,r_{n-1},r_n+c r_{n+2})
 = y_n^{\mathcal F} + c \operatorname{Wr}^\dag(r_1,\dots,r_{n-1}, r_{n+2})
 \end{gather*}
and
\begin{gather*} y_{n+1}(x,c)  :=  \operatorname{Wr}^\dag(r_1,\dots,r_{n-1},r_n+c r_{n+2},r_{n+1})
 = y_{n+1}^{\mathcal F} + c \operatorname{Wr}^\dag(r_1,\dots,r_{n-1}, r_{n+2}, r_{n+1}).
 \end{gather*}
Now let
\begin{gather*}
 y^{(n)}_n := \operatorname{Wr}^\dag(r_1,\dots,r_{n-1},r_{n+1}).
 \end{gather*}
Then by Lemma~\ref{MVlem} we have
\begin{gather*}
 \operatorname{Wr}\big( y^{\mathcal F}_n, y^{(n)}_n\big)  =  {\tilde T}_n y^{\mathcal F}_{n-1} y^{\mathcal F}_{n+1},  \\
\operatorname{Wr}\big(y_{n+1}^{\mathcal F},y_{n+1}(x,c)\big)  = -c{\tilde T}_{n+1} y^{(n)}_n y_{n+2}^{\mathcal F},\\
    \operatorname{Wr}\big(y^{(n)}_n,y_n(x,c)\big)  = {\tilde T}_{n} y_{n-1} y_{n+1}(x,c).
    \end{gather*}
This establishes the lemma.
\end{proof}

\subsection*{Acknowledgments}

The research of AV is supported in part by NSF grant DMS-1362924. CY is grateful to the Department of Mathematics at UNC Chapel Hill for hospitality during a visit in October 2014 when this work was initiated. CY thanks Benoit Vicedo for valuable discussions.


\pdfbookmark[1]{References}{ref}
\LastPageEnding

\end{document}